\documentclass[10pt,a4paper]{amsart}
\usepackage{amssymb,amscd,amsmath,amsthm, young,mathrsfs}
\usepackage{mathdots} 
\usepackage{mathrsfs}
\usepackage[mathcal]{eucal} 
\usepackage{float}
\usepackage[usenames]{color}
\usepackage{soul}
\usepackage{hyperref}
\hypersetup{
    colorlinks=true,
    linkcolor=blue,
    filecolor=blue,      
    urlcolor=blue,
    citecolor=blue,
}

\theoremstyle{plain}
\newtheorem{thm}{Theorem}[section]
\newtheorem{lem}[thm]{Lemma}

\newtheorem{cor}[thm]{Corollary}
\newtheorem*{claim*}{Claim}

\newtheorem{exm}[thm]{Example}
\newtheorem{qun}[thm]{Question}
\newtheorem{rem}[thm]{Remark}

\theoremstyle{remark}

\newtheorem{dfn}[thm]{Definition}

\numberwithin{equation}{section}

\numberwithin{table}{section}

\newcommand{\bsm}{\boldsymbol{m}}

\newcommand{\N}{\mathbb{N}}
\newcommand{\Z}{\mathbb{Z}}
\newcommand{\Q}{\mathbb{Q}}
\newcommand{\F}{\mathbb{F}}

\newcommand{\tensor}{\otimes}

\newcommand{\mfp}{\mathfrak{p}}

\newcommand{\Gri}{\ensuremath{\mathcal{O}}}

\renewcommand{\epsilon}{\varepsilon}

\renewcommand{\phi}{\varphi}
\renewcommand{\theta}{\vartheta}

\newcommand{\mcO}{\mathcal{O}}

\newcommand{\ideal}{\triangleleft}

\newcommand{\Zp}{\mathbb{Z}_{p}}

\DeclareMathOperator{\Spec}{Spec}

\def \bfx {{\bf x}}

\def \mcL {\ensuremath{\mathcal{L}}}

\def \Fp {\ensuremath{\mathbb{F}_p}}

\def \Fq {\ensuremath{\mathbb{F}_q}}
\def \Zp  {\mathbb{Z}_p}

\usepackage[colorinlistoftodos]{todonotes}

\author{Seungjai Lee} \address{Department of Mathematics, Incheon National University, Incheon 22012, South Korea}\email{seungjai.lee@inu.ac.kr}

\begin{document}

\title{Zeta functions of solvable Lie algebras over finite fields - with calculations in detail}
\begin{abstract}
	Let $L$ be a solvable Lie algebra of dimension less than or equal to 4 over finite fields. We compute and record, in explicit symbolic form, the zeta functions enumerating subalgebras or ideals of $L$, and study their properties. We also discuss the implications of our data, in particular in relation to the general theory of Lie algebras over finite fields and zeta functions of Lie algebras over commutative rings.
\end{abstract}


\allowdisplaybreaks
\maketitle
\tableofcontents
\section{Introduction}

\subsection{Background and motivations}
Let $R$ be a Dedekind domain, finitely generated as a ring and let $L$ be an $R$-Lie algebra,  free and finitely generated as an $R$-module. Let $a_{m}^{\leq}(L)$ and $a_{m}^{\triangleleft}(L)$ denote the number of subalgebras and ideals of index $m$ in $L$, respectively. For $*\in\{\leq, \triangleleft\}$, we define the \textit{subalgebra ($*=\leq$)} and \textit{ideal ($*=\triangleleft$) zeta functions of $L$}  as the Dirichlet generating series
\begin{equation*}
	\zeta_{L}^{*}(s)=\sum_{m=1}^{\infty}a_{m}^{*}(L)m^{-s},
\end{equation*}
enumerating subalgebras or ideals in $L$ of finite index in $L$. Here, and in what follows,
$s$ is a complex variable and $*$ stands for either $\leq$ or $\triangleleft$.

Suppose $R=\mcO$ is the ring of integers of a number field $K$.
For a (non-zero) prime ideal $\mfp\in\Spec(\mcO)$ we write $
\mcO_{\mfp}$ for the completion of $\mcO$ at $\mfp$, a complete
discrete valuation ring of characteristic zero and residue field
$\mcO/\mfp$ of cardinality $q$ and characteristic $p$,
say. Let $L(\Gri_{\mfp}):=L\tensor_\mcO\Gri_{\mfp}$. Primary decomposition yields the Euler product 
\begin{equation*}
	\zeta^{*}_{L}(s)
	=\prod_{\mfp\in\Spec(\mcO)\setminus\{(0)\}}\zeta^{*}_{L(\Gri_{\mfp})}(s),\label{equ:euler}
\end{equation*}
expressing  $\zeta^{*}_{L}(s)$ as an infinite
product of ``local'' zeta function $\zeta^{*}_{L(\Gri_{\mfp})}(s)$. It is a deep result that each individual $\zeta^{*}_{L(\Gri_{\mfp})}(s)$ is a rational function in the parameter $q^{-s}$ (cf. \cite[Theorem 3.5]{GSS/88}). 

Zeta functions counting subalgebras or ideals of $R$-Lie algebras have been extensively studied over the last few decades; see \cite{GSS/88, duS-ennui/02, duSWoodward/08, Voll/10, Voll/15, Rossmann/16} and references therein. One of the fundamental questions in this field is to study the variation of $\zeta_{L(\Gri_{\mfp})}^{*}(s)$ with $\mfp$: called the \textit{uniformity} problem.

\begin{dfn}\label{def:uniform.zp}
	The (global) zeta function $\zeta_{L}^{*}(s)$ is \emph{$\Gri_{\mfp}$-finitely uniform} if there exist finitely many rational functions $W_{1}^{*}(X,Y),\ldots,W_{k}^{*}(X,Y)\in\Q(X,Y)$ for $k\in\mathbb{N}$ such that, for every (non-zero) prime ideal $\mfp\in\Spec(\mcO)$, 
	\[\zeta_{L(\Gri_{\mfp})}^{*}(s)=W_{i}^{*}(q,q^{-s})\]
	for some $i\in\{1,\ldots,k\}$. It is  \emph{$\Gri_{\mfp}$-uniform} if $k=1$ for all but finitely many $\mfp\in\Spec(\mcO)$, and \emph{$\Gri_{\mfp}$-non-uniform} if it is not finitely uniform.  
\end{dfn}

Du Sautoy and Grunewald proved in \cite[Theorem 1.3]{duSG/00} that for a given $\Gri$-Lie algebra $L$, there are smooth quasi-projective varieties $V_i^{*}$ ($i\in I^{*}$, $I^{*}$ finite), defined over $K$, and rational functions $W_{i}^{*}(X,Y)\in\mathbb{Q}(X,Y)$ such that for almost all $\mfp\in\Spec(\mcO)$,
\begin{equation}\label{eq:duSG}
	\zeta^{*}_{L(\Gri_{\mfp})}(s)=\sum_{i\in I}\left|\overline{V_i^{*}}(\mathbb{F}_q)\right|W_i^{*}(q,q^{-s}),
\end{equation}
where $\overline{V_i^{*}}$ denote the reduction modulo $\mfp$ of a fixed $\Gri_{\mfp}$-model of $V_i^{*}$ and $\left|\overline{V_i^{*}}(\mathbb{F}_q)\right|$ denote the number of $\Fq$-points of $\overline{V_i^{*}}$. This result suggests that one way to understand $\zeta_{L}^{*}(s)$ and its uniformity is to identify the set of varieties $V_i$ as in \eqref{eq:duSG} and compute the number of their $\Fq$-points. Unfortunately, while the proofs of \cite[Theorem 1.3]{duSG/00} and its variations are constructive, they all rely on some form of resolution of singularities for $K$-varieties, which often makes the computation impractical.

Motivated by this problem, in \cite{Lee25} the author considered a finite approximation of zeta functions of $\mcO$-Lie algebras. Let $L$ be a $\mcO$-Lie algebra as above, and let $L(\mcO/\mfp):=L\tensor_\mcO\mcO/\mfp=L(\Fq)$, where $q$ is the cardinality of the residue field $\mcO/\mfp$. Then $L(\Fq)$ is a finite-dimensional $\Fq$-Lie algebra. Write $n$ to denote its dimension as a $\Fq$-vector space. The zeta function 
\begin{equation*}
	\zeta_{L(\Fq)}^{*}(s):=\sum_{H* L(\Fq)}|L(\Fq):H|^{-s}=\sum_{i=0}^{n}a_{q^{i}}^{*}(L(\Fq))q^{-is},
\end{equation*}
is in fact a finite Dirichlet polynomial, where $a_{q^{i}}^{*}(L(\Fq))$ is the number of subalgebras/ideals of $L(\Fq)$ of codimension $i$ in $L(\Fq)$.

\begin{exm}
	For $n\in\N$, let $\Fq^{n}$ be an $\Fq$-Lie algebra with the null product. We have
	\begin{align*}
		\zeta_{\Fq^{n}}(s):=\zeta_{\Fq^{n}}^{\leq}(s)=\zeta_{\Fq^{n}}^{\ideal}(s)=\sum_{i=0}^{n}a_{q^{i}}^{*}(\Fq^{n})q^{-is}=\sum_{i=0}^{n}\binom{n}{i}_{q}\,q^{-is},
	\end{align*}
	where 
	\[\binom{n}{i}_{q}=\begin{cases}\frac{(1-q^n)(1-q^{n-1})\cdots(1-q^{n-i+1})}{(1-q)(1-q^2)\cdots(1-q^i)}&i\leq n,\\0&i>n,\end{cases}\] is the Gaussian binomial coefficient that counts the number of subspaces of (co)dimension $i$ in a vector space of dimension $n$ over $\Fq$. 
\end{exm}

In \cite[Section 2]{Lee25}, the author gave a constructive proof for the following result:
\begin{lem}[\cite{Lee25}, Lemma 1.3]\label{lem:fp.general}
	For a given $\Gri$-Lie algebra $L$, there are quasi-projective varieties $U_j^{*}$ ($j\in J^{*}$, $J^{*}$ finite), defined over $K$, not necessarily irreducible, and polynomials $P_{j}^{*}(X,Y)\in\mathbb{Z}[X,Y]$ such that, for almost all $\mfp\in\Spec(\mcO)$, 
	\begin{equation}\label{eq:duSG.fp}
		\zeta_{L(\mcO/\mfp)}^{*}(s)=\sum_{j\in J^{*}}|\overline{U_{j}^{*}}(\mcO/\mfp)|P_{j}^{*}(q,q^{-s}),
	\end{equation} 
	where $q=|\mcO/\mfp|$.
\end{lem}
It implies that like $\zeta_{L(\mcO_{\mfp})}^{*}(s)$, the behavior of $\zeta_{L(\Fq)}^{*}(s)$ also depends on the number of points on the reduction modulo $q$ of some associated system of algebraic varieties. Let us define the $\Fq$-uniformity of $\zeta_{L}^{*}(s)$ as follows:
\begin{dfn}[\cite{Lee25}, Definition 1.4]\label{def:uniform.fp}
	The (global) zeta function $\zeta_{L}^{*}(s)$ is $\Fq$-\emph{finitely uniform} if there exist polynomials $P_{1}^{*}(X,Y),\ldots,P_{k}^{*}(X,Y)\in\Q[X,Y]$ for $k\in\mathbb{N}$ such that, for every (non-zero) prime ideal $\mfp\in\Spec(\mcO)$, 
	\[\zeta_{L(\mcO/\mfp)}^{*}(s)=P_{i}^{*}(q,q^{-s})\]
	for some $i\in\{{1,\ldots,k}\}$, where $q=|\mcO/\mfp|$. It is $\Fq$-\emph{uniform} if $k=1$ for all but finitely many $\mfp\in\Spec(\mcO)$, and $\Fq$-\emph{non-uniform} if it is not $\Fq$-finitely uniform.
\end{dfn}
In \cite[Question 1.5]{Lee25}, the author asked the following question:
\begin{qun}
	For a given $\mcO$-Lie algebra $L$, is $\zeta_{L}^{*}$ $\Fq$-(finitely) uniform if and only if it is $\mcO_{\mfp}$-(finitely) uniform?
\end{qun}

Although the author could not rigorously answer this question, explicit computations in \cite{Lee25} suggested this in a very favorable way, and called for more explicit computations and applications to develop a general theory of zeta functions of Lie algebras over finite fields.

\subsection{Main results and applications}
In this article, using the method developed in \cite{Lee25}, we compute and study the zeta functions of all solvable $\Fq$-Lie algebras of dimension $n\leq4$.

Throughout this article, let
\[L_{1,1}(\Fq)=\langle e_1\rangle_{\Fq}\]
be an abelian 1-dimensional $\Fq$-Lie algebra,
\begin{align*}
	L_{2,1}(\Fq)&=\langle e_1,e_2\rangle_{\Fq},&L_{2,2}(\Fq)&=\langle e_1,e_2:[e_1,e_2]=e_2\rangle_{\Fq},
\end{align*}
be two non-isomorphic 2-dimensional solvable $\Fq$-Lie algebras, and let $L^{i}(\Fq)$, $L_{a}^{i}(\Fq)$, $M^i(\Fq)$, $M_{a}^{i}(\Fq)$, and $M_{a,b}^{i}(\Fq)$ denote $\Fq$-Lie algebras of dimension 3 and 4, as classified in  \cite{Lee25}. First, we give explicit, symbolic computations of $\zeta_{L(\Fq)}^{*}(s)$ for aforementioned solvable $\Fq$-Lie algebras.

Throughout this article, let us write $t:=q^{-s}$, and define
\begin{align*}
	|V_{3,a}(\Fq)|&=\left|\{x\in\Fq:ax^2-x-1=0\}\right|,\\	
	|V_{4,a}(\Fq)|&=\left|\{x\in\Fq:ax^2-1=0\}\right|,\\	
	|V_{6,a,b}^{(1)}(\Fq)|&=\left|\{x\in\Fq:-a^2x^3+ax^2+bx+1=0\}\right|,\\		
	|V_{6,a,b}^{(2)}(\Fq)|&=\left|\{x\in\Fq:ax^3-bx^2+x+1=0\}\right|,\\
    |V_{6,a,b}^{(3)}(\Fq)|&=|\{x\in\Fq:(a+b)x^3+(a+b^2)x^2+2abx+a^2\}|,\\
      |V_{6,a,b}^{(4)}(\Fq)|&=|\{x\in\Fq:(a+b^2)x^2+2abx+a^2\}|,\\
	|V_{7,a,b}^{(1)}(\Fq)|&=\left|\{x\in\Fq:a^2x^3-bx-1=0\}\right|,\\	
	|V_{7,a,b}^{(2)}(\Fq)|&=\left|\{x\in\Fq:
	ax^3-bx^2+1=0\}\right|,\\	
    |V_{7,a,b}^{(3)}(\Fq)|&=|\{x\in\Fq:ax^3+b^2x^2+2abx+a^2\}|,\\
	|V_{13,a}(\Fq)|&=\left|\{x\in\Fq:x^2+x-a=0\}\right|,\\
	|V_{14,a}(\Fq)|&=\left|\{x\in\Fq:x^2-a=0\}\right|.		
\end{align*}

\begin{thm}(Dimension $\leq3$, ideals)\label{thm:sol3} For $a\in\Fq^{\times}$, 
	\begin{align*}
		\zeta_{L_{1,1}(\Fq)}^{\triangleleft}(s)&=1+q,\\
		\zeta_{L_{2,1}(\Fq)}^{\triangleleft}(s)&=1+(1+q)t+t^2,\\
		\zeta_{L_{2,2}(\Fq)}^{\triangleleft}(s)&=1+t+t^2,\\
		\zeta_{L^{1}(\Fq)}^{\triangleleft}(s)&=1+(1+q+q^2)t+(1+q+q^2)t^2+t^3,\\
		\zeta_{L^{2}(\Fq)}^{\triangleleft}(s)&=1+t+(1+q)t^2+t^3,\\
		\zeta_{L_a^{3}(\Fq)}^{\triangleleft}(s)&=1+t+|V_{3,a}(\Fq)|t^2+t^3, \\
		\zeta_{L_0^{3}(\Fq)}^{\triangleleft}(s)&=1+(1+q)t+2t^2+t^3,\\
		\zeta_{L_a^{4}(\Fq)}^{\triangleleft}(s)&=1+t+|V_{4,a}(\Fq)|t^2+t^3,\\
		\zeta_{L_0^{4}(\Fq)}^{\triangleleft}(s)&=1+(1+q)t+t^2+t^3.
	\end{align*}
\end{thm}	
\begin{thm}(Dimension  $\leq3$, subalgebras)\label{thm:sol3.subalg} For $a\in\Fq^{\times}$, 
	\begin{align*}
		\zeta_{L_{1,1}(\Fq)}^{\leq}(s)&=1+q,\\
		\zeta_{L_{2,1}(\Fq)}^{\leq}(s)&=1+(1+q)t+t^2,\\
		\zeta_{L_{2,2}(\Fq)}^{\leq}(s)&=1+(1+q)t+t^2,\\
		\zeta_{L^{1}(\Fq)}^{\leq}(s)&=1+(1+q+q^2)t+(1+q+q^2)t^2+t^3,\\
		\zeta_{L^{2}(\Fq)}^{\leq}(s)&=1+(1+q+q^2)t+(1+q+q^2)t^2+t^3,\\
		\zeta_{L_a^{3}(\Fq)}^{\leq}(s)&=1+(1+|V_{3,a}(\Fq)|q)t+(1+q+q^2)t^2+t^3, \\
		\zeta_{L_0^{3}(\Fq)}^{\leq}(s)&=1+(1+2q)t+(1+q+q^2)t^2+t^3,\\
		\zeta_{L_a^{4}(\Fq)}^{\leq}(s)&=1+(1+|V_{4,a}(\Fq)|q)t+(1+q+q^2)t^2+t^3,\\
		\zeta_{L_0^{4}(\Fq)}^{\leq}(s)&=1+(1+q)t+(1+q+q^2)t^2+t^3.
	\end{align*}
\end{thm}	
\begin{thm}(Dimension 4, ideals)\label{thm:sol4} For $a,b\in\Fq^{\times}$, we have
	\begin{align*}
		\zeta_{M^{1}(\Fq)}^{\triangleleft}(s)=&1+(1+q+q^2+q^3)t+(1+q+q^2)(1+q^2)t^2+(1+q+q^2+q^3)t^3+t^4,\\
		\zeta_{M^{2}(\Fq)}^{\triangleleft}(s)=&1+t+(1+q+q^2)t^2+(1+q+q^2)t^3+t^4,\\
		\zeta_{M_a^3(\Fq)}^{\triangleleft}(s)=&
		1+t+(2+q)t^2+(2+q)t^3+t^4\;\;\;\;(a\neq1),\\
		\zeta_{M_1^3(\Fq)}^{\triangleleft}(s)=&
		1+t+(1+q)t^2+(1+q)t^3+t^4,\\
		\zeta_{M_0^{3}(\Fq)}^{\triangleleft}(s)=&1+(1+q)t+(2+q)t^2+(2+q)t^3+t^4,\\
		\zeta_{M^{4}(\Fq)}^{\triangleleft}(s)=&1+(1+q+q^2)t+(2+q+q^2)t^2+(2+q)t^3+t^4,\\
		\zeta_{M^{5}(\Fq)}^{\triangleleft}(s)=&1+(1+q+q^2)t+(1+q+q^2)t^2+(1+q)t^3+t^4,\\
		\zeta_{M_{a,b}^{6}(\Fq)}^{\triangleleft}(s)=&1+t+|V_{6,a,b}^{(2)}(\Fq)|t^2+|V_{6,a,b}^{(1)}(\Fq)|t^3+t^4,\\		
		\zeta_{M_{a,0}^{6}(\Fq)}^{\triangleleft}(s)=&1+t+|V_{6,a,0}^{(2)}(\Fq)|t^2+|V_{6,a,0}^{(1)}(\Fq)|t^3+t^4,\\		
		\zeta_{M_{0,b}^{6}(\Fq)}^{\triangleleft}(s)=&1+(1+q)t+(1+|V_{3,b}(\Fq)|)t^2+(1+|V_{3,b}(\Fq)|)t^3+t^4,\\	
		\zeta_{M_{0,0}^{6}(\Fq)}^{\triangleleft}(s)=&1+(1+q)t+2t^2+2t^3+t^4,\\
		\zeta_{M_{a,b}^{7}(\Fq)}^{\triangleleft}(s)=&1+t+|V_{7,a,b}^{(2)}(\Fq)|t^2+|V_{7,a,b}^{(1)}(\Fq)|t^3+t^4\\		
		\zeta_{M_{a,0}^{7}(\Fq)}^{\triangleleft}(s)=&1+t+|V_{7,a,0}^{(2)}(\Fq)|t^2+|V_{7,a,0}^{(1)}(\Fq)|t^3+t^4,\\		
		\zeta_{M_{0,b}^{7}(\Fq)}^{\triangleleft}(s)=&1+(1+q)t+(1+|V_{4,b}(\Fq)|)t^2+\left(1+|V_{4,b}(\Fq)|\right)t^3+t^4,\\		
		\zeta_{M_{0,0}^{7}(\Fq)}^{\triangleleft}(s)=&1+(1+q)t+t^2+t^3+t^4,\\
		\zeta_{M^{8}(\Fq)}^{\triangleleft}(s)=&1+(1+q)t+3t^2+2t^3+t^4,\\
		\zeta_{M_{a}^{9}(\Fq)}^{\triangleleft}(s)=&1+(1+q)t+t^2+t^4,\\
		\zeta_{M^{12}(\Fq)}^{\triangleleft}(s)=&1+t+(1+q)t^2+t^3+t^4,\\		
		\zeta_{M_{a}^{13}(\Fq)}^{\triangleleft}(s)=&1+t+|V_{13,a}(\Fq)|t^2+t^3+t^4,\\	
		\zeta_{M_{0}^{13}(\Fq)}^{\triangleleft}(s)=&1+(1+q)t+2t^2+t^3+t^4,\\	
		\zeta_{M_{a}^{14}(\Fq)}^{\triangleleft}(s)=&1+t+|V_{14,a}(\Fq)|t^2+t^3+t^4.\\	
	\end{align*}
\end{thm}	
\begin{thm}(Dimension 4, subalgebras)\label{thm:sol4.subalg} For $a,b\in\Fq^{\times}$, we have	\begin{align*}
		\zeta_{M^{1}(\Fq)}^{\leq}(s)=&1+(1+q+q^2+q^3)t+(1+q+q^2)(1+q^2)t^2+(1+q+q^2+q^3)t^3+t^4,\\
		\zeta_{M^{2}(\Fq)}^{\leq}(s)=&1+(1+q+q^2+q^3)t+(1+q+q^2)(1+q^2)t^2+(1+q+q^2+q^3)t^3+t^4,\\
		\zeta_{M_a^3(\Fq)}^{\leq}(s)=&
		1+(q^2+2q+1)t+(q^3+3q^2+q+1)t^2+(1+q+q^2+q^3)t^3+t^4\;\;\;\;\;(a\neq1),\\
		\zeta_{M_1^3(\Fq)}^{\leq}(s)=&
		1+(q^2+q+1)t+(q^3+2q^2+q+1)t^2+(1+q+q^2+q^3)t^3+t^4,\\
		\zeta_{M_0^{3}(\Fq)}^{\leq}(s)=&1+(q^2+2q+1)t+(q^3+3q^2+q+1)t^2+(1+q+q^2+q^3)
		t^3+t^4,\\
		\zeta_{M^{4}(\Fq)}^{\leq}(s)=&1+(q^2+2q+1)t+(q^3+3q^2+q+1)t^2+(1+q+q^2+q^3)t^3+t^4,\\
		\zeta_{M^{5}(\Fq)}^{\leq}(s)=&1+(q^2+q+1)t+(q^3+2q^2+q+1)t^2+(1+q+q^2+q^3)t^3+t^4,\\			
          \zeta_{M_{a,b}^6(\Fq)}^{\leq}(s)=& 1+(1+q|V_{6,a,b}^{(2)}(\Fq)|)t+(1+q+q^2+(q^2-1)|V_{6,a,b}^{(3)}(\Fq)|+|V_{6,a,b}^{(1)}(\Fq)|)t^2\\
		&+(1+q+q^2+q^3)t^3+t^4,\;\;\;\;\;(a\neq-b)\\
    \zeta_{M_{a,-a}^6(\Fq)}^{\leq}(s)=& 1+(1+q|V_{6,a,-a}^{(2)}(\Fq)|)t+(1+q+q^2+(q^2-1)(1+|V_{6,a,-a}^{(4)}(\Fq)|)+|V_{6,a,-a}^{(1)}(\Fq)|)t^2\\
		&+(1+q+q^2+q^3)t^3+t^4,\\   
        \zeta_{M_{a,0}^{6}(\Fq)}^{\leq}(s)=&1+(1+q|V_{6,a,0}^{(2)}(\Fq)|)t+(1+q+q^2+(q^2-1)|V_{6,a,0}^{(2)}(\Fq)|+|V_{6,a,0}^{(1)}(\Fq)|)t^2\\
		&+(1+q+q^2+q^3)t^3+t^4,\\	
		\zeta_{M_{0,b}^{6}(\Fq)}^{\leq}(s)=&1+(1+q+q|V_{3,b}(\Fq)|)t+(1+q+2q^2+q^2|V_{3,b}(\Fq)|)t^2+(1+q+q^2+q^3)t^3+t^4,\\	
		\zeta_{M_{0,0}^{6}(\Fq)}^{\leq}(s)=&1+(1+2q)t+(1+q+3q^2)t^2+(1+q+q^2+q^3)t^3+t^4,\\		
        \zeta_{M_{a,b}^7(\Fq)}^{\leq}(s)=& 1+(1+q|V_{7,a,b}^{(2)}(\Fq)|)t+(1+q+q^2+(q^2-1)|V_{7,a,b}^{(3)}(\Fq)|+|V_{7,a,b}^{(1)}(\Fq)|)t^2\\
		&+(1+q+q^2+q^3)t^3+t^4,\\
		\zeta_{M_{a,0}^7(\Fq)}^{\leq}(s)=& 1+(1+q|V_{7,a,0}^{(2)}(\Fq)|)t+(1+q+q^2+(q^2-1)|V_{7,a,0}^{(2)}(\Fq)|+|V_{7,a,0}^{(1)}(\Fq)|)t^2\\
		&+(1+q+q^2+q^3)t^3+t^4,\\
		\zeta_{M_{0,b}^7(\Fq)}^{\leq}(s)=& 1+(1+q+q|V_{4,b}(\Fq)|)t+(1+q+2q^2+q^2|V_{4,b}(\Fq)|)t^2+(1+q+q^2+q^3)t^3+t^4,\\
		\zeta_{M_{0,0}^7(\Fq)}^{\leq}(s)=&1+(1+q)t+(1+q+2q^2)t^2+(1+q+q^2+q^3)t^3+t^4,\\
		\zeta_{M^{8}(\Fq)}^{\leq}(s)=&1+(1+3q)t+(1+q+4q^2)t^2+(1+q+q^2+q^3)t^3+t^4,\\
		\zeta_{M_{a}^{9}(\Fq)}^{\leq}(s)=&1+(1+q)t+(1+q+2q^2)t^2+(1+q+q^2+q^3)t^3+t^4,\\
		\zeta_{M^{12}(\Fq)}^{\leq}(s)=&1+(1+q+q^2)t+(1+q+2q^2+q^3)t^2+(1+q+q^2+q^3)t^3+t^4,\\		
		\zeta_{M_{a}^{13}(\Fq)}^{\leq}(s)=&1+(1+q|V_{13,a}(\Fq)|)t+( 1+q+q^2+q^2|V_{13,a}(\Fq)|)t^2+(1+q+q^2+q^3)t^3+t^4,\\	
		\zeta_{M_{0}^{13}(\Fq)}^{\leq}(s)=&1+(2q+1)t+(3q^2+q+1)t^2+(1+q+q^2+q^3)t^3+t^4,\\	
		\zeta_{M_{a}^{14}(\Fq)}^{\leq}(s)=&1+(1+q|V_{14,a}(\Fq)|)t+(1+q+q^2+q^2|V_{14,a}(\Fq)|)t^2+(1+q+q^2+q^3)t^3+t^4.
	\end{align*}
\end{thm}
Among these solvable Lie algebras, $L_{1,1}(\Fq)$, $L_{2,1}(\Fq)$, $L^1(\Fq)$, $L_0^4(\Fq)$, $M^1(\Fq)$, $M^5(\Fq)$, and $M_{0,0}^{7}(\Fq)$ are nilpotent, while the rest are solvable but non-nilpotent. These symbolic computations immediately allow us to deduce:
\begin{cor}
	Let $L$ be a solvable $\mcO$-Lie algebra of dimension up to 4. Then 
	\begin{enumerate}
		\item $\zeta_{L}^{*}(s)$ is always $\Fq$-finitely uniform,
		\item if $L$ is nilpotent, then $\zeta_{L}^{*}(s)$ is always $\Fq$-uniform, and
		\item $\zeta_{L}^{\leq}(s)$ is $\Fq$-(finitely) uniform if and only if $\zeta_{L}^{\ideal}(s)$ is.
	\end{enumerate}
\end{cor}
Note that while $\zeta_{L}^{*}(s)$ is  known to be $\mcO_{\mfp}$-uniform as well for nilpotent $L$ of dimension up to 4 from various sources (\cite{GSS/88,duSWoodward/08,Ross.Zeta}), for solvable but non-nilpotent cases we can already observe Lie algebras of dimension 3 (e.g., $L_a^{3}(\Fq)$) with $\Fq$-finitely uniform but not $\Fq$-uniform zeta functions.

In fact, we can ask much finer questions.

\subsubsection{The period of $\zeta_{L}^{*}(s)$ over $\mfp$}

\begin{dfn}
	For a given $\mcO$-Lie algebra $L$, let us define $k^{*}(L,\mcO_\mfp)$ to be the smallest $k\in\N$ where
	\[\zeta_{L(\Gri_{\mfp})}^{*}(s)=W_{i}^{*}(q,q^{-s})\]
	for some $i\in\{1,\ldots,k\}$ for all but finitely many $\mfp\in\Spec(\mcO)$. We call this the \emph{$\mcO_{\mfp}$-period} of $\zeta_{L}^{*}(s)$. 
	For instance, $\zeta_{L}^{*}(s)$ is $\mcO_{\mfp}$-uniform if $k^{*}(L,\mcO_\mfp)=1$,  $\mcO_{\mfp}$-finitely uniform if $k^{*}(L,\mcO_\mfp)<\infty$, and  $\mcO_{\mfp}$-non-uniform if there is no such $k^{*}(L,\mcO_\mfp)$.
	
	Let us analogously define $k^{*}(L,\Fq)$, and call this the \emph{$\Fq$-period} of $\zeta_{L}^{*}(s)$.
\end{dfn}
The ``periods'' of $\zeta_{L}^{*}(s)$ may be viewed as a measure of the degree of uniformity of the zeta function, counting the minimum number of rational functions or integer polynomials required to describe $\zeta_{L(\Gri_{\mfp})}^{*}(s)$ or $\zeta_{L(\mcO/\mfp)}^{*}(s)$ for all but finitely many $\mfp\in\Spec(\mcO)$. For instance, abelian $\mcO$-Lie algebras of finite dimension will always have $k^{*}(L,\mcO_\mfp)=k^{*}(L,\Fq)=1$.

Since every ideal of a Lie algebra is also a subalgebra, but not conversely, computing $\zeta_{L}^{\leq}(s)$ is typically more difficult than computing $\zeta_{L}^{\ideal}(s)$ in practice. For example, comparing the formulas for $\zeta_{M_{a,b}^6(\Fq)}^{\leq}(s)$ and $\zeta_{M_{a,b}^6(\Fq)}^{\ideal}(s)$ in Theorems  \ref{thm:sol4} and \ref{thm:sol4.subalg}, one sees that additional point-counting terms  $|V_{6,a,b}^{(3)}(\Fq)|$ and $|V_{6,a,b}^{(4)}(\Fq)|$ only appear in $\zeta_{M_{a,b}^6(\Fq)}^{\leq}(s)$. These extra varieties reflect the added complexity inherent in enumerating all subalgebras rather than just ideals.

The question is, in what precise sense is $\zeta_{L}^{\leq}(s)$ more ``complicated'' than $\zeta_{L}^{\ideal}(s)$? For instance, do the additional varieties $|V_{6,a,b}^{(3)}(\Fq)|$ and $|V_{6,a,b}^{(4)}(\Fq)|$ increase the periodicity complexity, so that
\[k^{\ideal}(M_{a,b}^6,\Fq)<k^{\leq}(M_{a,b}^6,\Fq)?\]
 Somewhat surprisingly, the answer turns out to be negative, as the following result demonstrate:

\begin{thm}\label{thm:extra.var}
    \begin{enumerate}
        \item For $a,b\neq0$ and $a\neq -b$, we have $|V_{6,a,b}^{(2)}(\Fq)|=|V_{6,a,b}^{(3)}(\Fq)|$.
        \item For $a=-b\neq0$, we have $|V_{6,a,b}^{(2)}(\Fq)|-1=|V_{6,a,b}^{(4)}(\Fq)|$.
        \item For $a,b\neq0$, we have $|V_{7,a,b}^{(2)}(\Fq)|=|V_{7,a,b}^{(3)}(\Fq)|$.
    \end{enumerate}
\end{thm}
\begin{proof}
    \begin{enumerate}
        \item  Let $u\in\Fq$ be a root of $V_{6,a,b}^{(3)}:(a+b)x^3+(a+b^2)x^2+2abx+a^2$. Since we assume $a\neq0$, clearly $u\neq0$. Write $u=av$. Then 
        \[(a+b)u^3+(a+b^2)u^2+2abu+a^2=a^2(a(a+b)v^3+(a+b^2)v^2+2bv+1).\]
        Thus $|V_{6,a,b}^{(3)}(\Fq)|=\left|\{x\in\Fq:a(a+b)x^3+(a+b^2)x^2+2bx+1=0\}\right|$.

         Now let $w\in\Fq$ be a root of $V_{6,a,b}^{(2)}:ax^3-bx^2+x+1=0$. Since $a,b\neq0$, one has $w\neq0$. In addition, if $w=\frac{b}{a}$, then $aw^3-bw^2+w+1=\frac{b}{a}+1=0$ forces $a=-b$. As we are assuming $a\neq-b$, we assure $w\neq\frac{b}{a}$. Provided $au-b\neq0$, write $y=\frac{1}{aw-b}$. One gets
        \[a(a+b)y^3+(a+b^2)y^2+2by+1=\frac{a^2}{(aw-b)}(aw^3-bw^2+w+1)=0.\]

        On the other hand, write $z=\frac{bv+1}{av}$ where $v$ is a root of $a(a+b)x^3+(a+b^2)x^2+2bx+1=0$. Substituting $z$ into $V_{6,a,b}^{(2)}$ gives
                \[V_{6,a,b}^{(2)}(z)=a(a(a+b)v^3+(a+b^2)v^2+2bv+1)=0.\]
Therefore one has
\[|V_{6,a,b}^{(3)}(\Fq)|=\left|\{x\in\Fq:a(a+b)x^3+(a+b^2)x^2+2bx+1=0\}\right|=|V_{6,a,b}^{(2)}(\Fq)|.\]
        
\item For $a=-b\neq0$, note that $V_{6,a,-a}^{(2)}:ax^3+ax^2+x+1=(ax^2+1)(x+1)=0$ implies 
\[|V_{6,a,-a}^{(2)}(\Fq)|=\begin{cases}
    1+r&a\neq-1,\\
    2&a=-1,
\end{cases}\]
where $r$ is the number of $\Fq$-solutions of $ax^2+1=0$. 

Now, consider $V_{6,a,-a}^{(4)}:(a+a^2)x^2-2a^2x+a^2=a((1+a)x^2-2ax+a)=0$. 

If $a=-1$ we have $|V_{6,-1,1}^{(4)}|=1$.

If $a\neq-1$, then we get a bijection
\[(1+a)x^2-2ax+a=0\iff ay^2+1=0\textrm{ via }x=\frac{1}{y+1}.\]

Hence the number of solutions $|V_{6,a,-a}^{(4)}|=r$, and  we always have $|V_{6,a,b}^{(2)}(\Fq)|-1=|V_{6,a,b}^{(4)}(\Fq)|$ as required.

\item Analogous to (1).   \qedhere     
    \end{enumerate}
\end{proof}

By Theorem \ref{thm:extra.var}, together with the computations in Theorem \ref{thm:sol3} to Theorem \ref{thm:sol4.subalg}, we get \begin{cor}\label{thm:fq.period}
	Let $L$ be a solvable $\mcO$-Lie algebra of dimension up to 4. Then $k^{\leq}(L,\Fq)=k^{\ideal}(L,\Fq)$.
\end{cor}

The conclusion of Corollary \ref{thm:fq.period} can in fact be extended to computations in \cite{Lee25}. For those $L$ whose $k^{*}(L,\mcO_\mfp)$ and $k^{*}(L,\Fq)$ are all known, to the author's knowledge there is no known example of $L$ where not all of $k^{*}(L,\mcO_\mfp)$ and $k^{*}(L,\Fq)$ are identical. One can naturally ask the following questions:
\begin{qun}\label{qun:L.period}
	For a given finite dimensional $\mcO$-Lie algebra $L$,
    \begin{enumerate}
        \item $k^{\leq}(L,\Fq)=k^{\ideal}(L,\Fq)$?
        \item $k^{*}(L,\Fq)=k^{*}(L,\mcO_\mfp)?$
         \item $k^{\leq}(L,\Fq)=k^{\ideal}(L,\Fq)=k^{\leq}(L,\mcO_\mfp)=k^{\ideal}(L,\mcO_\mfp)?$       
    \end{enumerate}
	If not, can one identify a family of $\mcO$-Lie algebras satisfying given properties?
\end{qun}

To date, significantly fewer explicit computations are known for the subalgebra zeta functions  $\zeta_{L}^{\leq}(s)$ than for the ideal zeta functions $\zeta_{L}^{\ideal}(s)$. Since explicit computation of  $\zeta_{L(\Fq)}^{\ideal}(s)$ and $k^{\ideal}(L,\Fq)$ is generally much more accessible than computing the rest, a positive answer to  Question \ref{qun:L.period}, or even identifying a broad class of Lie algebras $L$ where (parts of) Question \ref{qun:L.period} is true would represent major progress toward understanding the uniformity phenomena of the zeta functions $\zeta_{L}^{*}(s)$.

\subsubsection{Solvable Lie algebras and their splitting behavior over $\Fp$}
If we let $L$ to be a $\Z$-Lie algebra and consider $\zeta_{L(\Zp)}^{*}(s)$ and $\zeta_{L(\Fp)}^{*}(s)$, one can ask a subtler question than just $\mcO_{\mfp}$- or $\Fq$-uniformity. In \cite{Lee/22,Lee25}, the author also studied how the local zeta functions $\zeta_{L(\Zp)}^{*}(s)$ and $\zeta_{L(\Fp)}^{*}(s)$ change on congruence classes as $p$ varies. To be precise, let $L$ be a $\Z$-Lie algebra.

\begin{dfn}[Essentially \cite{Lee/22}, Definition 1.1]
	We say $\zeta_{L}^{*}(s)$ is Rational Function On Residue Classes (RFORC), if there exists a fixed integer $N$ and finitely many rational functions $W_{i}(X,Y)\in\Q(X,Y)$ for $i\in\{1,\ldots,N\}$ such that for almost all primes $p$, if $p\equiv i\mod p$, then
	\[\zeta_{L(\Zp)}^{*}(s)=W_{i}^{*}(p,p^{-s}).\]
	We say $\zeta_{L}^{*}(s)$ is non-RFORC if no such $N$ exists.
\end{dfn}

\begin{dfn}[\cite{Lee25}, Definition 4.1]\label{dfn:porc}
	We say  $\zeta_{L(\Fp)}^{*}(s)$ is Polynomial On Residue Classes (PORC), if there exists a fixed integer $N$ and finitely many polynomials $W_{i}^{*}(X,Y)\in\Q[X,Y]$ for $i\in[N-1]_{0}$ such that for any primes $p$, if $p\equiv i\mod N$ then 
	\[\zeta_{L(\Fp)}^{*}(s)=W_{i}^{*}(p,p^{-s}).\] 
	We say $\zeta_{L(\Fp)}^{*}(s)$ is non-PORC if there is no such $N$.
\end{dfn}

It is clear that being RFORC or PORC is a stronger condition than being $\Zp$- or $\Fp$-finitely uniform. For instance, in \cite[Theorem 5.1]{Lee/22} the author constructed an 8-dimensional nilpotent $\Z$-Lie algebra $L_G$ whose ideal zeta function  $\zeta_{L_G}^{\ideal}(s)$ is $\Zp$-finitely uniform but non-RFORC. In the same article, the author also showed that if $L$ is nilpotent class-2 $\Z$-Lie algebra of dimension up to 6, then $\zeta_{L}^{\ideal}(s)$ is always RFORC.

Like the uniformity, the RFORC/PORC behavior of $\zeta_{L}^{(*)}(s)$ also depends on $\left|\overline{V_i^{*}}(\mathbb{F}_p)\right|$ and $\left|\overline{U_j^{*}}(\mathbb{F}_p)\right|$ in \eqref{eq:duSG} and \eqref{eq:duSG.fp}, in particular their splitting behavior over the primes. There is a deep result in number theory that if $f(x)$ is an integer polynomial in $x$, then the number of roots of $f(x)\mod p$ is Polynomial On Residue Classes (or simply PORC) if and only if the Galois group of $f(x)$ over the rationals is abelian (e.g., see \cite[Abelian Polynomial Theorem]{Wyman/72} and \cite{Weinstein/2016}). 

This result, together with the computations from Theorem \ref{thm:sol3} to Theorem \ref{thm:sol4.subalg}, allows us to deduce the following:

\begin{thm}
	Let $L$ be a solvable $\Z$-Lie algebras of dimension $n$.
	\begin{enumerate}
		\item If $n\leq3$, then $\zeta_{L(\Fp)}^{*}(s)$ is always PORC.
		\item If $n=4$, then there exists a Lie algebra $L$ where $\zeta_{L(\Fp)}^{*}(s)$ is non-PORC.
	\end{enumerate}
\end{thm}
\begin{proof}
	
	\begin{enumerate}
		\item  Theorem     \ref{thm:sol3} and   Theorem \ref{thm:sol3.subalg} suggest that one only needs to understand $|V_{3,a}(\Fp)|$ and $|V_{4,a}(\Fp)|$.  Since both  $V_{3,a}:ax^2-x-1=0$ and $V_{4,a}:ax^2-1=0$ are quadratic polynomials in $x$, their Galois group over $\Q$ is always abelian. Therefore $|V_{3,a}(\Fp)|$ and $|V_{4,a}(\Fp)|$ are PORC, which also makes $\zeta_{L(\Fp)}^{*}(s)$ PORC for $n\leq3$. 
		\item Consider $M_{2,0}^{7}(\Fp)$. One can identify $|V_{7,2,0}^{(2)}(\Fp)|$ in both $a_{q}^{\ideal}(M_{2,0}^{7}(\Fp))$ and $a_{q}^{\leq}(M_{2,0}^{7}(\Fp))$, where $|V_{7,2,0}^{(2)}(\Fp)|=|\{x\in\Fp:2x^3+1=0\}$. In \cite[Corollary 2.3]{Lee/22} the author showed that 
		\[|V_{7,2,0}^{(2)}(\Fp)|=\left\{
		\begin{array}{ll}
			0 & \text{if $p\equiv1\bmod3$, and $p\neq a^{2}+27b^{2}$ for integers $a$ and $b$,} \\
			3 & \text{if $p\equiv1\bmod3$, and $p=a^{2}+27b^{2}$ for integers $a$ and $b$,}\\
			1 & \text{if $p\equiv2\bmod3$},
		\end{array}
		\right.\]
		and in particular $|V_{7,2,0}^{(2)}(\Fp)|$ is non-PORC. This implies that $\zeta_{M_{2,0}^{7}(\Fp)}^{*}(s)$ is non-PORC. In fact, there are infinitely many 4-dimensional solvable $\Fp$-Lie algebras with non-PORC zeta functions.   \qedhere
	\end{enumerate}
\end{proof}

Note that $M_{2,0}^{7}(\Fp)$ is a solvable but non-nilpotent $\Fp$-Lie algebra. One interesting observation is that the zeta functions $\zeta_{L(\Fp)}^{*}(s)$ for nilpotent $\Fp$-Lie algebras in this article are all PORC (in fact, uniform). To the author's knowledge, the smallest dimension of nilpotent $\Z$-Lie algebras whose zeta function is known to be $\Zp$- and $\Fp$-non-uniform is 9 (the elliptic curve example in \cite{duS-ecI/01}), and that of whose zeta function is known to be RFORC/PORC is 8. Since being nilpotent is informally being ``almost abelian'', do our results suggest that being non-nilpotent makes the intrinsic structure of $L$, in particular  $\left|\overline{V_i^{*}}(\mathbb{F}_p)\right|$ and $\left|\overline{U_j^{*}}(\mathbb{F}_p)\right|$ more complicated?

\begin{qun}
	Can we find a solvable $\mcO$-Lie algebra of dimension $n=5$ or $n=6$ whose zeta function is $\mcO_{\mfp}$- or $\Fq$-non-uniform?
\end{qun}

\subsection{Assumptions and Notations}
We write $\N=\{1,2,\dots\}$ for the set of all natural numbers and $\N_{0}$ for the set of all natural numbers including zero. Given $n\in\N$, we write $[n]$ for $\{1,2,\dots,n\}$ and $[n]_0$ for $\{0,1,2,\dots,n\}$. We write $\textnormal{Tr}_n(\mathbb{F}_{q})$ to denote the set of $n\times n$ upper-triangular matrices over $\Fq$. We write $\mcO$ for the ring of integers of a number field, $\Gri_{\mfp}$ for the completion of  $\Gri$  at a nonzero prime ideal $\mfp$ of $\Gri$, $q$ for the
cardinality of the residue field of $\Gri_{\mfp}$, and $p$ for its residue
characteristic.	For a given $\mcO$-Lie algebra $L$, we write $L(\Gri_{\mfp})= L\tensor_{\mcO}\Gri_{\mfp}$, and $L(\Fq)= L\tensor_{\mcO}\F_q$ as above.  Whenever there is a presentation for a Lie algebra or a group, we always assume that up to anti-symmetry, all other unlisted commutators are trivial. We write $t:=q^{-s}$, where $s$ is a complex variable.

\section{Computing zeta functions of $\Fq$-Lie algebras}\label{sec:2}
To make this paper self-contained, we summarize the method introduced in \cite[Section 2]{Lee25}.

Suppose $L(\Fq)$ is a $\Fq$-Lie algebra of dimension $n$, and let  $\mathcal{B}=(e_1,e_2,\ldots,e_n)$ be an $\Fq$-basis of $L$. It is well-known that each $k$-dimensional subspace $W$ of $L(\Fq)$ can be uniquely represented by an $n\times n$ matrix $N$ of rank $k$ in Reduced Row Echelon Form (RREF) with respect to this basis, where the first $k$-rows of $N$ span $W$. In \cite[Section 2]{Lee25}, we introduced another form of matrices, called Reduced Row Diagonal Form (RRDF), to help our computation:
\begin{dfn}[\cite{Lee25}, Definition 2.2]
	We say an $n\times n$ matrix $M$ is in \emph{Reduced Row Diagonal Form (RRDF)} if it is of the form
	\begin{equation}\label{eq:matrix}
		M=(m_{i,j})_{i,j\in[n]}=\left(\begin{matrix}0_{a_1}&0&0&0&\cdots&0&0\\0&I_{b_1}&D_{1,2}&0&\cdots&D_{1,r}&0\\0&0&0_{a_2}&0&\cdots&0&0\\0&0&0&I_{b_2}&\cdots&D_{2,r}&0\\0&0&0&0&\ddots&0&0\\0&0&0&0&0&0_{a_r}&0\\0&0&0&0&0&0&I_{b_r}\end{matrix}\right)\in\textnormal{Tr}_n(\mathbb{F}_{q}),
	\end{equation}
	where $1\leq r\leq \frac{n+1}{2}$, $a_{1},b_{r}\geq0$, $a_{2}\ldots,a_{r},b_{1},\ldots,b_{r-1}>0$, and $D_{i,j}$ denote an arbitrary $b_{i}\times a_{j}$ matrix over $\Fq$ for $i,j\in[r]$.	We write $\textnormal{DTr}_{n}(\Fq)$ to denote the set of $n\times n$ matrices of the form \eqref{eq:matrix} over $\Fq$.
\end{dfn}	
For each $M\in \textnormal{DTr}_{n}(\Fq)$, the positivity conditions $a_{2}\ldots,a_{r},b_{1},\ldots,b_{r-1}>0$ ensure that there are uniquely determined $r\in\N$ and $\underline{a}=(a_{1},\ldots,a_{r}),\underline{b}=(b_{1},\ldots,b_{r})$ such that 
\[(m_{1,1},m_{2,2}\ldots,m_{n,n})=(0^{(a_{1})},1^{(b_{1})},\dots,0^{(a_{r})},1^{(b_{r})}),\]
where $0^{(a_{i})}$ denote the vector of zeroes of length $a_{i}$ and $1^{(b_{i})}$ denote the vector of ones of length $b_{i}$ for $i\in[r]$. For example, if a matrix $M\in\textnormal{DTr}_{4}(\Fq)$ has $(m_{1,1},m_{2,2},m_{3,3},m_{4,4})=(0,0,1,1)$, we  say $r=1$ with $a_1=b_1=2$, not $r=2$ with $a_1=1,b_1=0,a_2=1,b_2=2$ or $r=4$ with $a_1=1,b_1=0,a_2=0,b_2=0,a_3=0,b_3=0,a_4=1,b_4=2$.  In addition, $a_1,b_r\geq0$ allows us to have matrices with $m_{1,1}=1$ or $m_{n,n}=0$.

\begin{dfn}[\cite{Lee25}, Definition 2.3]
	We call $(\underline{a},\underline{b})$ the \textit{diagonal type} of $M$, and write $d(M)=(\underline{a},\underline{b})$. We write $\mathcal{M}_{\underline{a},\underline{b}}$ to denote the set of $M\in\textnormal{DTr}_{n}(\Fq)$ with diagonal type $(\underline{a},\underline{b})$, and $M_{\underline{a},\underline{b}}$ to represent an arbitrary matrix in $\mathcal{M}_{\underline{a},\underline{b}}$.
\end{dfn}

In \cite[Lemma 2.4]{Lee25} we showed that there is a one-to-one correspondence between the subspaces of $L(\Fq)$ of codimension $k$ and the upper triangular matrices $M\in\textnormal{DTr}_{n}(\Fq)$ of rank $n-k$ with $\sum_{i=1}^{r}a_i=k$, where the non-zero rows of $M$ spans the corresponding subspace.

Let us call the subset of the Grassmannian where its elements have a particular diagonal type a \textit{diagonal cell}. Each $\mathcal{M}_{\underline{a},\underline{b}}$ is a diagonal cell of diagonal type $(\underline{a},\underline{b})$. Since each $m_{i,i}$ can be either $0$ or $1$, there are $2^{n}$ different diagonal cells of distinct diagonal types, and each of them satisfies $\left|\mathcal{M}_{\underline{a},\underline{b}}\right|=q^{\sum_{i=1}^{r}\sum_{j=1}^{i-1}a_{i}b_{j}}.$

\begin{exm} \label{exm:Heisenberg} Let 
	\[H=\langle e_{1},e_{2},e_{3}:[e_1,e_2]=e_3\rangle_{\mcO}\]
	be the Heisenberg $\mcO$-Lie algebra. Then  $H(\Fq)$ is a 3-dimensional $\Fq$-Lie algebra of order $q^3$. The subspaces of  $H(\Fq)$ with index $q^k$ are precisely the following, represented by the diagonal cells:
	\begin{itemize}
		\item $M_{(0),(3)}$ if $k=0$,
		\item $M_{(1),(2)}$,  $M_{(0,1),(1,1)}$, or
		$M_{(0,1),(2,0)}$ if $k=1$,
		\item $M_{(2),(1)}$, $M_{(1,1),(1,0)}$, or
		$M_{(0,2),(1,0)}$ if $k=2$,
		\item $M_{(3),(0)}$ if $k=3$.
	\end{itemize}
\end{exm}

Unfortunately, not every element of $\mathcal{M}_{\underline{a},\underline{b}}$ may give rise to subalgebras or ideals of $L(\Fq)$. Write $M\leq L(\Fq)$ if the subspace generated by the rows of $M$ is a subalgebra of $L(\Fq)$, and $M\triangleleft L(\Fq)$ if the subspace generated by the rows of $M$ is an ideal of $L(\Fq)$.

Let	
\[F_{q}^{*}:=\left\{M\in\textnormal{DTr}_{n}(\mathbb{F}_{q}): M* L(\Fq)\right\}\] and
$F_{\underline{a},\underline{b}}^{*}=F_{q}^{*}\cap \mathcal{M}_{\underline{a},\underline{b}}$. Then we have
\begin{align*}
	\zeta_{L(\Fq)}^{*}(s)=\sum_{\underline{a},\underline{b}}\left|F_{\underline{a},\underline{b}}^{*}\right|q^{-(\sum_{i=1}^{r}a_{i})s}.
\end{align*}

Hence we need to be able to describe $F_{\underline{a},\underline{b}}^{*}$ and compute $\left|F_{\underline{a},\underline{b}}^{*}\right|$.

For a given $M=(m_{i,j})\in\mathcal{M}_{\underline{a},\underline{b}}$, we may consider the rows $\bsm_{1},\ldots,\bsm_{n}$ of $M$ to be additive generators of an $\Fq$-subspace of $L(\Fq)$. This will be a subalgebra of $L(\Fq)$ if
\begin{equation*}
	[\bsm_{i},\bsm_{j}]\in\langle\bsm_{1},\ldots,\bsm_{n}\rangle_{\Fq}\;\textrm{for all }1\leq i<j\leq n,
\end{equation*}
and an ideal if 
\begin{equation*}
	[\bsm_{i},e_{j}]\in\langle\bsm_{1},\ldots,\bsm_{n}\rangle_{\Fq}\;\textrm{for all }1\leq i,j\leq n.
\end{equation*}
In other words, $M\in F_{\underline{a},\underline{b}}^{\leq}$ if and only if for each $1\leq i,j\leq n$, there exist $\left(Y^{(1)}_{i,j},\ldots,Y^{(n)}_{i,j}\right)\in\Fq^{n}$ such that 
\begin{equation}\label{eq:subalg.con}
	[\bsm_{i},\bsm_{j}]=\sum_{k=1}^{n}Y_{i,j}^{(k)}\bsm_{k},
\end{equation}
and $M\in F_{\underline{a},\underline{b}}^{\ideal}$ if and only if for each $1\leq i,j\leq n$, there exist $\left(Y^{(1)}_{i,j},\ldots,Y^{(n)}_{i,j}\right)\in\Fq^{n}$ such that 
\begin{equation}\label{eq:ideal.con}
	[\bsm_{i},e_{j}]=\sum_{k=1}^{n}Y_{i,j}^{(k)}\bsm_{k}.
\end{equation}

For $i,j\in[n]$, let $C_j$ denote the matrix whose rows are $\bold{c}_{i}=[e_{i},e_{j}]$. Then one can re-write \eqref{eq:subalg.con} and say $M\in F_{\underline{a},\underline{b}}^{\leq}$ if and only if for each $1\leq i,j\leq n$, there exist $\left(Y^{(1)}_{i,j},\ldots,Y^{(n)}_{i,j}\right)\in\Fq^{n}$ such that 

\begin{equation}\label{eq:subalg} 
	\bold{m}_{i}\left(\sum_{l=j}^{n}m_{j,l}C_{l}\right)=\left(Y^{(1)}_{i,j},\ldots,Y^{(n)}_{i,j}\right)M.
\end{equation}
Similarly, \eqref{eq:ideal.con} says that $M\in F_{\underline{a},\underline{b}}^{\triangleleft}$ if and only if for each $1\leq i,j\leq n$, there exist $\left(Y^{(1)}_{i,j},\ldots,Y^{(n)}_{i,j}\right)\in\Fq^{n}$ such that 
\begin{equation}\label{eq:ideal} 
	\bold{m}_{i}C_{j}=\left(Y^{(1)}_{i,j},\ldots,Y^{(n)}_{i,j}\right)M.
\end{equation}

For a given $M\in\mathcal{M}_{\underline{a},\underline{b}}$, as described in \eqref{eq:matrix},	let
\[M^{\sharp}=\left(\begin{matrix}I_{a_1}&0&0&0&\cdots&0&0\\0&I_{b_1}&-D_{1,2}&0&\cdots&-D_{1,r}&0\\0&0&I_{a_2}&0&\cdots&0&0\\0&0&0&I_{b_2}&\cdots&-D_{2,r}&0\\0&0&0&0&\ddots&0&0\\0&0&0&0&0&I_{a_r}&0\\0&0&0&0&0&0&I_{b_r}\end{matrix}\right)\]
and
\[M^{\flat}=\left(\begin{matrix}0_{a_1}&0&0&0&\cdots&0&0\\0&I_{b_1}&0&0&\cdots&0&0\\0&0&0_{a_2}&0&\cdots&0&0\\0&0&0&I_{b_2}&\cdots&0&0\\0&0&0&0&\ddots&0&0\\0&0&0&0&0&0_{a_r}&0\\0&0&0&0&0&0&I_{b_r}\end{matrix}\right)=\textrm{diag}(0^{(a_{1})},1^{(b_{1})},\dots,0^{(a_{r})},1^{(b_{r})}).\]
A direct computation shows that
\begin{equation}\label{eq:matrix1}
	MM^{\sharp}=M^{\flat}.
\end{equation}

Let us consider the ideals first. 	By \eqref{eq:matrix1} we can rewrite \eqref{eq:ideal} as 
\begin{equation}\label{eq:ideal.final}
	\bold{m}_{i}C_{j}M^{\sharp}=\left(Y^{(1)}_{i,j},\ldots,Y^{(n)}_{i,j}\right)M^{\flat},
\end{equation}
and $M\in F_{\underline{a},\underline{b}}^{\triangleleft}$ if and only if for each $1\leq i,j\leq n$, there exist $\left(Y^{(1)}_{i,j},\ldots,Y^{(n)}_{i,j}\right)\in\Fq^{n}$ that satisfies \eqref{eq:ideal.final}.

Let $g_{i,j,k}^{\triangleleft}(M)$ denote the $k$-th entry of the $n$-tuple $\bold{m}_{i}C_{j}M^{\sharp}$.  Since $M^{\flat}$ is a diagonal matrix $\textrm{diag}(0^{(a_{1})},1^{(b_{1})},\dots,0^{(a_{r})},1^{(b_{r})})$, it is clear that the $k$-th entry of the $n$-tuple $\left(Y^{(1)}_{i,j},\ldots,Y^{(n)}_{i,j}\right)M^{\flat}$ is either 0 or $Y^{(k)}_{i,j}$. 	Therefore, by \eqref{eq:ideal.final} $M\in F_{\underline{a},\underline{b}}^{\triangleleft}$ if and only if  $g_{i,j,k}^{\triangleleft}(M)=0$ (if $m_{k,k}=0)$ or $g_{i,j,k}^{\triangleleft}(M)=Y^{(k)}_{i,j}$ (if $m_{k,k}=1$) for each $1\leq i,j\leq n$. 

Note that when $m_{k,k}=1$ the condition $g_{i,j,k}^{\triangleleft}(M)=Y^{(k)}_{i,j}$ is redundant, since for any  $M$ there always exists $Y^{(k)}_{i,j}\in\Fq$ where $Y^{(k)}_{i,j}=g_{i,j,k}^{\triangleleft}(M)$ (see Example \ref{exm:L3a} below). Hence we conclude that $M$ gives an ideal if and only if we can solve a system of equations $g_{i,j,k}^{\triangleleft}(M)=0$ for $1\leq i,j,k\leq n$, where $m_{k,k}=0$.	

This implies that, analogous to \cite[Theorem 5.5]{duSG/00}, there exist  polynomials
\[g_{i,j,k}^{\ideal}(\boldsymbol{X})\in\Gri[X_{u,v}:1\leq u\leq v\leq n]\]
of degree at most 2 such that
\begin{align*}
	F_{\underline{a},\underline{b}}^{\triangleleft}&=F_{q}^{\triangleleft}\cap \mathcal{M}_{\underline{a},\underline{b}}\\
	&=\left\{M\in\mathcal{M}_{\underline{a},\underline{b}} : g_{i,j,k}^{\triangleleft}(M)=0\textrm{ for }1\leq i,j,k\leq n, \textrm{ where } m_{k,k}=0\right\}.
\end{align*}
Let $\bold{U}_{\underline{a},\underline{b}}^{\triangleleft}$ be the subvariety in $\mathbb{A}^{\log_{q}|\mathcal{M}_{\underline{a},\underline{b}}|}$ over the number field $K$, defined by those equations $g_{i,j,k}^{\triangleleft}(M)=0$ for $m_{k,k}=0$, where   $|\mathcal{M}_{\underline{a},\underline{b}}|=q^{\sum_{i=1}^{r}\sum_{j=1}^{i-1}a_{i}b_{j}}$. We have
\[\left|\overline{\bold{U}_{\underline{a},\underline{b}}^{\triangleleft}}(\Fq)\right|=\left|F_{\underline{a},\underline{b}}^{\triangleleft}\right|.\]

\begin{exm}\label{exm:L3a}
	For $a\in\Fq$, let
	\[L_{a}^3(\Fq)=\langle e_1,e_2,e_3\,\mid\,[e_3,e_1]=e_2, [e_3,e_2]=ae_1+e_2\rangle_{\Fq}\]
	be a 3-dimensional solvable $\Fq$-Lie algebra. Let    \[M=\left(\begin{matrix}1 & m_{1,2}& m_{1,3}\\ & 0 & 0 \\&&0\end{matrix}\right)\in\mathcal{M}_{(0,2),(1,0)}.\]
	Suppose we want to compute $\left|F_{(0,2),(1,0)}^{\triangleleft}\right|$
	First, for example, if $k=1$, we get $g_{1,3,1}^{\ideal}(M)=am_{1,2}$. Note that in this case for a given $M$ there always exists $Y_{1,3}^{(1)}\in\Fq$ where $Y_{1,3}^{(1)}=am_{1,2}$, making the equation $g_{1,3,1}^{\ideal}(M)=Y_{1,3}^{(1)}$ redundant. 
	
	Therefore, as explained above, we need to solve the system of equations $g_{i,j,k}^{\ideal}(M)=0$ for $1\leq i,j\leq 3$ and $k\in\{2,3\}$. The system of equations we need to solve then becomes
	\begin{align*}
		m_{1,3}&=0,&1+m_{1,2}-am_{1,2}^2&=0,
	\end{align*}
	and $\left|F_{(0,2),(1,0)}^{\triangleleft}\right|$ is given by the solution set. Hence we get $\left|F_{(0,2),(1,0)}^{\triangleleft}\right|=|V_{3,a}(\Fq)|$, where $|V_{3,a}(\Fq)|=|\{x\in\Fq:ax^2-x-1=0\}|$.
\end{exm}

For subalgebras, one can similarly 	rewrite \eqref{eq:subalg} as 
\begin{equation}\label{eq:subalg.final}
	\bold{m}_{i}\left(\sum_{l=j}^{n}m_{j,l}C_{l}\right)M^{\sharp}=\left(Y^{(1)}_{i,j},\ldots,Y^{(n)}_{i,j}\right)M^{\flat},
\end{equation}
and analogously define $g_{i,j,k}^{\leq}(M)$ and  $\bold{U}_{\underline{a},\underline{b}}^{\leq}$ as before.	

To ease the notation, write $c_{\underline{a},\underline{b}}^{*}=\left|\overline{\bold{U}_{\underline{a},\underline{b}}^{*}}(\Fq)\right|=\left|F_{\underline{a},\underline{b}}^{*}\right|$. To summarize, we get

\begin{equation}\label{eq:zeta.fq}
	\zeta_{L(\Fq)}^{*}(s)=\sum_{\underline{a},\underline{b}}c_{\underline{a},\underline{b}}^{*}\cdot q^{-(\sum_{i=1}^{r}a_{i})s}.
\end{equation} 

\begin{exm}\label{exm:L22.ideal}
	Let
	\[L_{2,2}(\Fq)=\langle e_1, e_2:[e_1,e_2]=e_2\rangle_{\Fq}\]
	be a 2-dimensional solvable $\Fq$-Lie algebra of order $q^2$. The subspaces of $L_{2,2}(\Fq)$ with codimension $k$ are precisely the following, represented by the diagonal cells:
	\begin{itemize}
		\item $M_{(0),(2)}$ if $k=0$,
		\item $M_{(0,1),(1,0)}$ and $M_{(1),(1)}$ if $k=1$,
		\item $M_{(2),(0)}$ if $k=2$.
	\end{itemize}
	
	Also for $L_{2,2}(\Fq)$ we have
	\begin{align*}
		C_1=\begin{pmatrix}
			0&0\\
			0&-1
		\end{pmatrix},\textrm{ and }
		C_2=\begin{pmatrix}0&1\\
			0&0\end{pmatrix}.
	\end{align*}
	We can explicitly calculate  $\zeta_{L_{2,2}(\Fq)}^{\triangleleft}(s)$ using the method we just described.
	
	First, one can easily check that $M_{(0),(2)}$ and $M_{(2),(0)}$ trivially give one ideal of index 1 and $q^{2}$, respectively, giving
	\[a_1^{\ideal}(L_{2,2}(\Fq))=a_{q^2}^{\ideal}(L_{2,2}(\Fq))=1\]
	
	For  $M_{(1),(1)}$, we have
	\begin{align*}
		i=1,j=1:\,&(0,0)=(0,Y_{1,1}^{(2)}),&
		i=1,j=2:\,&(0,0)=(0,Y_{1,2}^{(2)}),\\
		i=2,j=1:\,&(0,0)=(0,Y_{2,1}^{(2)}),&
		i=2,j=2:\,&(0,0)=(0,Y_{2,2}^{(2)}),
	\end{align*}
	showing  there is no system of equation to be solved. This gives $c_{(1),(1)}^{\triangleleft}=1$. 
	
	The interesting case is  $M_{(0,1),(1,0)}=\left(\begin{matrix}1&m_{1,2}\\&0\end{matrix}\right)$. In this case we have
	\begin{align*}
		i=1,j=1:\,&(0,m_{1,2})=(Y_{1,1}^{(1)},0),&
		i=1,j=2:\,&(0,-1)=(Y_{1,2}^{(1)},0),\\
		i=2,j=1:\,&(0,0)=(Y_{2,1}^{(1)},0),&
		i=2,j=2:\,&(0,0)=(Y_{2,2}^{(1)},0).
	\end{align*}
	Here we need to solve a system of equation
	\begin{align*}
		g_{1,1,2}^{\ideal}(M)&=m_{1,2}=0,&g_{1,2,2}^{\ideal}(M)&=-1=0.
	\end{align*}
	Since  $-1=0$ is unsolvable, we get $c_{(0,1),(1,0)}^{\triangleleft}=0$. Together, we get \[a_{q}^{\ideal}(L_{2,2}(\Fq))=1.\]
	
	Summing all up, we get
	\[\zeta_{L_{2,2}(\Fq)}^{\ideal}(s)=1+t+t^2.\]    
\end{exm}

\section{Explicit computations of $\zeta_{L}^{*}(s)$ for solvable $\Fp$-Lie algebras of dimension $n\leq3$}\label{sec:dim3}
The method developed in Section \ref{sec:2} can be summarized as follows:
\begin{enumerate}
	\item For a given $n$-dimensional $\Fq$-Lie algebra $L(\Fq)$, write down $2^{n}$ disjoint diagonal cells $\mathcal{M}_{\underline{a},\underline{b}}$.
	\item For each diagonal cell with diagonal type $\underline{a},\underline{b}$, find a description for $\bold{U}_{\underline{a},\underline{b}}^{*}$ and compute $c_{\underline{a},\underline{b}}^{*}=\left|\overline{\bold{U}_{\underline{a},\underline{b}}^{*}}(\mathbb{F}_{q})\right|$.
	\item summing over all diagonal types, we obtain  $\zeta_{L(\Fq)}^{*}(s)=\sum_{\underline{a},\underline{b}}c_{\underline{a},\underline{b}}^{*}\cdot q^{-(\sum_{i=1}^{r}a_{i})s}$ as in \eqref{eq:zeta.fq}.
\end{enumerate}
In this section, using the classification given by de Graaf \cite{deG/05}, we record in detail the computation of subalgebra and ideal zeta functions of solvable $\Fq$-Lie algebras of dimension $n\leq 3$.

First, note that there is only one 1-dimensional $\Fq$-Lie algebra, namely 
\[L_{1,1}(\Fq)=\langle e_1\rangle_{\Fq},\]
and it is easy to see that 
\[\zeta_{L_{1,1}(\Fq)}^{\ideal}(s)=\zeta_{L_{1,1}(\Fq)}^{\leq}(s)=\zeta_{\Fq}(s)=1+t.\]

For $n=2$, there are exactly two solvable $\Fq$-Lie algebras of dimension 2 (up to isomorphism), namely
\begin{align*}
	L_{2,1}(\Fq)&=\langle e_1,e_2\rangle_{\Fq},&L_{2,2}(\Fq)&=\langle e_1,e_2|[e_1,e_2]=e_2\rangle_{\Fq}.
\end{align*}

For $L_{2,1}(\Fq)$, we have
\[\zeta_{L_{2,1}(\Fq)}^{\ideal}(s)=\zeta_{L_{2,1}(\Fq)}^{\leq}(s)=\zeta_{\Fq^2}(s)=1+(1+q)t+t^2.\]

For $L_{2,2}(\Fq)$, Example \ref{exm:L22.ideal} shows
\[\zeta_{L_{2,2}(\Fq)}^{\ideal}(s)=1+t+t^2.\]

For subalgebras, note that since $[\bsm_{i},\bsm_{i}]=0$ for all $i\in[n]$, any one-dimensional subspace of $L$ is automatically  a subalgebra of $L$. Therefore we always have $a_1^{\leq}(L)=a_{q^{n}}^{\leq}(L)=1$ and $a_{q^{n-1}}^{\leq}(L)=\binom{n}{n-1}_q$.

This trivially gives

\[\zeta_{L_{2,2}(\Fq)}^{\leq}(s)=\zeta_{L_{2,1}(\Fq)}^{\leq}(s)=\zeta_{\Fq^2}(s)=1+(1+q)t+t^2.\]

Note that while $L_{2,1}(\Fq)$ is nilpotent (in fact, abelian) and $L_{2,2}(\Fq)$ is not, they have the same subalgebra zeta function. This is not a coincidence. In fact, we can show more:

\begin{thm}\label{thm:sub.iso}
	For $n\geq 2$, let 
	\[L=\langle e_1,e_2,\ldots,e_n:\,\forall\,2\leq i\leq n,\,[e_1,e_i]=e_i\,\rangle_{\mcO}.\]
	Then we have
	\begin{align*}
		\zeta_{L(\Fq)}^{\leq}(s)=\zeta_{\Fq^{n}}(s).
	\end{align*}
\end{thm}
\begin{proof}
	Let $M\in\textnormal{DTr}_{n}(\Fq)$  represents a subspace in $L(\Fq)$. We have $M\leq L(\Fq)$ if and only if 
	\begin{align*}      	[\bsm_{i},\bsm_{j}]\in\langle\bsm_{1},\ldots,\bsm_{n}\rangle_{\Fq}\;\textrm{for all }1\leq i<j\leq n.
	\end{align*}
	By the definition of $L(\Fq)$, we only need to check
	\begin{align*}        [\bsm_{1},\bsm_{j}]\in\langle\bsm_{1},\ldots,\bsm_{n}\rangle_{\Fq}\;\textrm{for all }2\leq j\leq n.
	\end{align*}
	It is straightforward to verify that for all $2\leq j\leq n$, 
	\begin{align*}        
		[\bsm_{1},\bsm_{j}]=[m_1e_1,\bsm_{j}]=m_1\bsm_{j}\in\langle\bsm_{1},\ldots,\bsm_{n}\rangle_{\Fq}.
	\end{align*}
	Hence every subspace of $L(\Fq)$ becomes a subalgebra of $L(\Fq)$.
\end{proof}

\begin{rem}
	Theorem \ref{thm:sub.iso} shows that we  have infinitely many pairs of non-isomorphic $\Fq$-Lie algebras with the same subalgebra zeta function, where one of them is nilpotent and the other is not.   
\end{rem}

\subsection{Solvable $\Fq$-Lie algebras of dimension $n=3$, counting ideals}

Suppose $L$ is a 3-dimensional $\Fq$-Lie algebra. As discussed, any subspaces of $L$ of codimension $k$ can be represented by one of the following 8 diagonal cells: 
\begin{itemize}
	\item $M_{(0),(3)}$ if $k=0$,
	\item $M_{(1),(2)}$,  $M_{(0,1),(1,1)}$, or
	$M_{(0,1),(2,0)}$ if $k=1$,
	\item $M_{(2),(1)}$, $M_{(1,1),(1,0)}$, or
	$M_{(0,2),(1,0)}$ if $k=2$,
	\item $M_{(3),(0)}$ if $k=3$.
\end{itemize}
Note that for any 3-dimensional $\Fq$-Lie algebra $L$ we always have $a_{1}^{\ideal}(L)=a_{q^3}^{\ideal}(L)=1$. Hence it suffices to compute $a_{q}^{\ideal}(L)$ and $a_{q^2}^{\ideal}(L)$. In other words, one only needs to check $g_{i,j,k}^{\ideal}(M)$ for $M_{(1),(2)}$,  $M_{(0,1),(1,1)}$, $M_{(0,1),(2,0)}$, $M_{(2),(1)}$, $M_{(1,1),(1,0)}$, and $M_{(0,2),(1,0)}$.

\begin{thm}
	Let $L^{1}$ be the 3-dimensional abelian $\Fq$-Lie algebra. Then we have
	\[\zeta_{L^{1}(\Fq)}^{\ideal}(s)=1+(1+q+q^2)t+(1+q+q^2)t^2+t^3.\]
\end{thm}
\begin{proof}
	Note that $\zeta_{L^{1}(\Fq)}^{\ideal}(s)=\zeta_{\Fq^3}(s)$.
\end{proof}

\begin{thm}
	Let 
	\[L^2(\Fq)=\langle e_1,e_2,e_3|[e_3,e_1]=e_1,\,[e_3,e_2]=e_2\rangle_{\Fq}.\] 
	Then we have
	\[\zeta_{L^{2}(\Fq)}^{\triangleleft}(s)=1+t+(1+q)t^2+t^3.\]
\end{thm}
\begin{proof}
	For $L^2(\Fq)$ we have
	\begin{align*}
		C_1=\begin{pmatrix}
			0&0&0\\
			0&0&0\\
			1&0&0
		\end{pmatrix},\,C_2=
		\begin{pmatrix}
			0&0&0\\
			0&0&0\\
			0&1&0
		\end{pmatrix},\textrm{ and }
		C_3=\begin{pmatrix}
			-1&0&0\\
			0&-1&0\\
			0&0&0
		\end{pmatrix}.
	\end{align*}	
	We can explicitly calculate the rest of $\zeta_{L^2(\Fq)}^{\triangleleft}(s)$ using the method we just described.

	Let us count the number of ideals of index $q^2$. For $M_{(0,2),(1,0)}$, we have
	\begin{align*}
		i=1,j=1:\,&(m_{1,3},-m_{1,2}m_{1,3},-m_{1,3}^2)=(Y_{1,1}^{(1)},0,0),&
		i=1,j=2:\,&(0,m_{1,3},0)=(Y_{1,2}^{(1)},0,0),\\
		i=1,j=3:\,&(-1,0,m_{1,3})=(Y_{1,3}^{(1)},0,0),&
		i=2,j=1:\,&(0,0,0)=(Y_{2,1}^{(1)},0,0),\\
		i=2,j=2:\,&(0,0,0)=(Y_{2,2}^{(1)},0,0),&
		i=2,j=3:\,&(0,0,0)=(Y_{2,3}^{(1)},0,0),\\
		i=3,j=1:\,&(0,0,0)=(Y_{3,1}^{(1)},0,0),&
		i=3,j=2:\,&(0,0,0)=(Y_{3,2}^{(1)},0,0),\\
		i=3,j=3:\,&(0,0,0)=(Y_{3,3}^{(1)},0,0).
	\end{align*}
	Note that in this case the system of equations $g_{i,j,k}^{\ideal}(M)=0$ becomes
	\begin{align*}
		m_{1,2}m_{1,3}&=0,&m_{1,3}&=0,
	\end{align*}    
	showing  $\bold{U}_{(0,2),(1,0)}^{\ideal}= \mathbb{A}^{1}$. This gives $c_{(0,2),(1,0)}^{\triangleleft}=q.$

	For  $M_{(1,1),(1,0)}$, we have
	\begin{align*}
		i=1,j=1:\,&(0,0,0)=(0,Y_{1,1}^{(2)},0),&
		i=1,j=2:\,&(0,0,0)=(0,Y_{1,2}^{(2)},0),\\
		i=1,j=3:\,&(0,0,0)=(0,Y_{1,3}^{(2)},0),&
		i=2,j=1:\,&(-m_{2,3},0,0)=(0,Y_{2,1}^{(2)},0),\\
		i=2,j=2:\,&(0,-m_{2,3},m_{2,3}^2)=(0,Y_{2,2}^{(2)},0),&
		i=2,j=3:\,&(0,1,-m_{2,3})=(0,Y_{2,3}^{(2)},0),\\
		i=3,j=1:\,&(0,0,0)=(0,Y_{3,1}^{(2)},0),&
		i=3,j=2:\,&(0,0,0)=(0,Y_{3,2}^{(2)},0),\\
		i=3,j=3:\,&(0,0,0)=(0,Y_{3,3}^{(2)},0).
	\end{align*}
	In this case the system of equations we need to solve is the simple one $m_{2,3}=0$. Therefore we get $c_{(1,1),(1,0)}^{\ideal}=1$.
	
	The interesting case is the following. For  $M_{(2),(1)}$, we have
	\[i=3,j=1:(-1,0,0)=(0,0,Y_{3,1}^{(3)}),\]
	which is impossible to solve ($1\neq 0$). Hence none of elements in $\mathcal{M}_{(2),(1)}$ gives rise to an ideal of $L^2(\Fq)$. Thus we get $c_{(2)(1)}^{\ideal}=0$.
	
	Together, we have 
	\[a_{q^2}^{\ideal}(L^2)=1+q.\]
	
	One can count the number of ideals of index $q$ in the same way.	For $M_{(0,1),(2,0)}$, we have
	\begin{align*}
		i=1,j=1:\,&(m_{1,3},0,-m_{1,3}^2)=(Y_{1,1}^{(1)},Y_{1,1}^{(2)},0),&
		i=1,j=2:\,&(0,m_{1,3},-m_{1,3}m_{2,3})=(Y_{1,2}^{(1)},Y_{1,2}^{(2)},0),\\
		i=1,j=3:\,&(-1,0,m_{1,3})=(Y_{1,3}^{(1)},Y_{1,3}^{(2)},0),&
		i=2,j=1:\,&(m_{2,3},0,-m_{1,3}m_{2,3})=(Y_{2,1}^{(1)},Y_{2,1}^{(2)},0),\\
		i=2,j=2:\,&(0,m_{2,3},-m_{2,3}^2)=(Y_{2,2}^{(1)},Y_{2,2}^{(2)},0),&
		i=2,j=3:\,&(0,-1,m_{2,3})=(Y_{2,3}^{(1)},Y_{2,3}^{(2)},0),\\
		i=3,j=1:\,&(0,0,0)=(Y_{3,1}^{(1)},Y_{3,1}^{(2)},0),&
		i=3,j=2:\,&(0,0,0)=(Y_{3,2}^{(1)},Y_{3,2}^{(2)},0),\\
		i=3,j=3:\,&(0,0,0)=(Y_{3,3}^{(1)},Y_{3,3}^{(2)},0).
	\end{align*}
	
	We get to solve the system of equations
	\begin{align*}
		m_{1,3}&=0,&m_{2,3}&=0,
	\end{align*}
	giving $c_{(0,1)(2,0)}^{\ideal}=1$. 
	
	For  $M_{(0,1),(1,1)}$, we have
	\[		i=3,j=2:(0,-1,0)=(Y_{3,2}^{(1)},0,Y_{3,2}^{(3)}),\]
	which is impossible to solve.  Thus we get $c_{(0,1)(1,1)}^{\ideal}=0$.

	Finally, for  $M_{(1),(2)}$, we have
	\[i=3,j=1:(-1,0,0)=(0,Y_{3,1}^{(2)},Y_{3,1}^{(3)}),\]
	which is impossible to solve.  Thus we get $c_{(1)(2)}^{\ideal}=0$. 
	Together, we have 
	\[a_{q}^{\ideal}(L^2)=1.\]

	Summing all up, we get 
	\[\zeta_{L^2(\Fq)}^{\triangleleft}(s)=1+t+(1+q)t^{2}+t^{3}.\qedhere\] 
\end{proof}

For the sake of notation, for the rest of this article we omit the description of $C_{i}$ and only record \eqref{eq:ideal.final} and \eqref{eq:subalg.final} for non-redundant cases.

\begin{thm}\label{thm:sol.La} For $a\in\Fq$, let
	\[L_{a}^3(\Fq):=\langle e_{1},e_{2},e_{3}\,\mid\,[e_3,e_1]=e_2, [e_3,e_2]=ae_1+e_2\rangle_{\Fq}.\]
	Then
	
	\begin{equation*}\zeta_{L_a^3(\Fq)}^{\triangleleft}(s)=\begin{cases}
			1+t+|V_{3,a}(\Fq)|t^2+t^3&a\neq0\\
			1+(1+q)t+2t^2+t^3&a=0.
		\end{cases}
	\end{equation*}
	where 
	\[|V_{3,a}(\Fq)|=\left|\{x\in\Fq:ax^2-x-1=0 \}\right|.\]
\end{thm}	

\begin{proof}
	
	Let us count the number of ideals of index $q^2$. For $M_{(0,2),(1,0)}$, we get 
	\begin{align*}
		i=1,j=1:\,&(0,-m_{1,3},0)=(Y_{1,1}^{(1)},0,0),\\
		i=1,j=2:\,&(-am_{1,3},(-1+am_{1,2})m_{1,3},am_{1,3}^2)=(Y_{1,2}^{(1)},0,0),\\
		i=1,j=3:\,&(am_{1,2},1+m_{1,2}-am_{1,2}^2,-am_{1,2}m_{1,3})=(Y_{1,3}^{(1)},0,0),
	\end{align*}
	which gives the system of equations
	\begin{align*}
		m_{1,3}&=0,&1+m_{1,2}-am_{1,2}^2&=0.
	\end{align*}
	
	If $a\neq0$, then $m_{1,2}\in\Fq$ needs to be a solution of $1+x-ax^2$. If $a=0$ then $m_{1,2}=-1$. Therefore we have
	\[c_{(0,2),(1,0)}^{\ideal}=\begin{cases}
		|V_{3,a}(\Fq)|&a\neq0,\\
		1&a=0.
	\end{cases}\]

	For  $M_{(1,1),(1,0)}$, we have
	\begin{align*}
		i=2,j=1:\,&(0,-m_{2,3},m_{2,3}^2)=(0,Y_{2,1}^{(2)},0),&
		i=2,j=2:\,&(-am_{2,3},-m_{2,3},m_{2,3}^2)=(0,Y_{2,2}^{(2)},0),\\
		i=2,j=3:\,&(a,1,-m_{2,3})=(0,Y_{2,3}^{(2)},0).
	\end{align*}
	One needs to solve
	\begin{align*}
		m_{2,3}&=0,&a&=0,
	\end{align*}
	giving
	\[c_{(1,1),(1,0)}^{\ideal}=\begin{cases}
		0&a\neq0,\\
		1&a=0.
	\end{cases}\]
	
	For  $M_{(2),(1)}$, we have
	\[i=3,j=1:(0,-1,0)=(0,0,Y_{3,1}^{(3)}),\]
	which is impossible to solve. Hence $c_{(2)(1)}^{\ideal}=0$.
	
	Together, we have
	\[a_{q^2}^{\ideal}(L_a^3)=\begin{cases}
		|V_{3,a}(\Fq)|&a\neq0,\\
		2&a=0.
	\end{cases}\]
	
	Let us also count the number of ideals of index $q$.	For $M_{(0,1),(2,0)}$, we get 
	\begin{align*}
		i=1,j=1:\,&(0,-m_{1,3},m_{1,3}m_{2,3})=(Y_{1,1}^{(1)},Y_{1,1}^{(2)},0),\\
		i=1,j=2:\,&(-am_{1,3},m_{1,3},m_{1,3}(am_{1,3}+m_{2,3}))=(Y_{1,2}^{(1)},Y_{1,2}^{(2)},0),\\
		i=1,j=3:\,&(0,1,-m_{2,3})=(Y_{1,3}^{(1)},Y_{1,3}^{(2)},0),\\
		i=2,j=1:\,&(0,-m_{2,3},m_{2,3}^2)=(Y_{2,1}^{(1)},Y_{2,1}^{(2)},0),\\
		i=2,j=2:\,&(-am_{2,3},-m_{2,3},m_{2,3}(am_{1,3}+m_{2,3}))=(Y_{2,2}^{(1)},Y_{2,2}^{(2)},0),\\
		i=2,j=3:\,&(a,1,-am_{1,3}-m_{2,3})=(Y_{2,3}^{(1)},Y_{2,3}^{(2)},0).
	\end{align*}
	The system of equations can be reduced into
	\begin{align*}
		m_{2,3}&=0,&am_{1,3}&=0,
	\end{align*}
	giving
	\[c_{(0,1),(2,0)}^{\ideal}=\begin{cases}
		1&a\neq0,\\
		q&a=0.
	\end{cases}\]
	
	For  $M_{(0,1),(1,1)}$, one of the conditions we get is 
	\[i=3,j=1:(0,-1,0)=(Y_{3,1}^{(1)},0,Y_{3,1}^{(3)}),\]
	giving  $c_{(0,1)(1,1)}^{\ideal}=0$. 
	
	Finally, for  $M_{(1),(2)}$, we have
	\begin{align*}
		i=2, j=3:\,&(a,1,0)=(0,Y_{2,3}^{(2)},Y_{2,3}^{(3)})&i=3,j=2\,&(-a,-1,0)=(0,Y_{3,2}^{(2)},Y_{3,2}^{(3)}).
	\end{align*}
	This requires $a=0$, giving
	\[c_{(1)(2)}^{\ideal}=\begin{cases}
		0&a\neq0,\\
		1&a=0.
	\end{cases}\]
	Summing all up, we get 
	\begin{equation*}\zeta_{L_a^3(\Fq)}^{\triangleleft}(s)=\begin{cases}
			1+t+|V_{3,a}(\Fq)|t^2+t^3&a\neq0\\
			1+(1+q)t+2t^2+t^3&a=0.\qedhere
		\end{cases}
	\end{equation*}
\end{proof}	
\begin{rem}
	For $n\in\N$,    let $\textrm{tr}_n$ denote the set of $n\times n$ upper-triangular matrices with the usual Lie bracket $[x,y]=xy-yx$ for $x,y\in\textrm{tr}_{n}$. Then $\textrm{tr}_2(\Fq)\cong L_{0}^3(\Fq)$, and our result for $\zeta_{L_0^3(\Fq)}^{\triangleleft}(s)$ indeed matches $\zeta_{\textrm{tr}_{2}(\Fq)}^{\triangleleft}(s)$ in \cite[Theorem 5.2]{Lee25}.
\end{rem}

\begin{thm}\label{thm:sol.L4a} For $a\in\Fq$, let
	\[L_{a}^4(\Fq):=\langle e_{1},e_{2},e_{3}\,\mid\,[e_3,e_1]=e_2, [e_3,e_2]=ae_1\rangle_{\Fq}.\]
	Then
	
	\begin{equation*}\zeta_{L_a^4(\Fq)}^{\triangleleft}(s)=\begin{cases}
			1+t+|V_{4,a}(\Fq)|t^2+t^3&a\neq0\\
			1+(1+q)t+t^2+t^3&a=0.
		\end{cases}
	\end{equation*}
	where 
	\[|V_{4,a}(\Fq)|=\left|\{x\in\Fq:ax^2-1=0 \}\right|.\]
\end{thm}

\begin{proof}
	Let us count the number of ideals of index $q^2$. For $M_{(0,2),(1,0)}$, we get 
	\begin{align*}
		i=1,j=1:\,&(0,-m_{1,3},0)=(Y_{1,1}^{(1)},0,0),\\
		i=1,j=2:\,&(-am_{1,3},am_{1,2}m_{1,3},am_{1,3}^2)=(Y_{1,2}^{(1)},0,0),\\
		i=1,j=3:\,&(am_{1,2},1-am_{1,2}^2,-am_{1,2}m_{1,3})=(Y_{1,3}^{(1)},0,0),
	\end{align*}
	which gives the system of equations
	\begin{align*}
		m_{1,3}&=0,&1-am_{1,2}^2&=0.
	\end{align*}
	
	If $a\neq0$, then $m_{1,2}\in\Fq$ needs to be a solution of $1-ax^2$. If $a=0$ then we get $1=0$. Therefore we have
	\[c_{(0,2),(1,0)}^{\ideal}=\begin{cases}
		|V_{4,a}(\Fq)|&a\neq0,\\
		0&a=0.
	\end{cases}\]

	For  $M_{(1,1),(1,0)}$, we have
	\begin{align*}
		i=2,j=1:\,&(0,-m_{2,3},m_{2,3}^2)=(0,Y_{2,1}^{(2)},0),&
		i=2,j=2:\,&(-am_{2,3},0,0)=(0,Y_{2,2}^{(2)},0),\\
		i=2,j=3:\,&(a,0,0)=(0,Y_{2,3}^{(2)},0).
	\end{align*}
	One needs to solve
	\begin{align*}
		m_{2,3}&=0,&a&=0,
	\end{align*}
	giving
	\[c_{(1,1),(1,0)}^{\ideal}=\begin{cases}
		0&a\neq0,\\
		1&a=0.
	\end{cases}\]
	
	For  $M_{(2),(1)}$, we have
	\[i=3,j=1:(0,-1,0)=(0,0,Y_{3,1}^{(3)}),\]
	which is impossible to solve. Hence $c_{(2)(1)}^{\ideal}=0$.
	
	Together, we have
	\[a_{q^2}^{\ideal}(L_a^4)=\begin{cases}
		|V_{4,a}(\Fq)|&a\neq0,\\
		1&a=0.
	\end{cases}\]
	
	Let us also count the number of ideals of index $q$. For $M_{(0,1),(2,0)}$, we get 
	\begin{align*}
		i=1,j=1:\,&(0,-m_{1,3},m_{1,3}m_{2,3})=(Y_{1,1}^{(1)},Y_{1,1}^{(2)},0),&
		i=1,j=2:\,&(-am_{1,3},0,am_{1,3}^2)=(Y_{1,2}^{(1)},Y_{1,2}^{(2)},0),\\
		i=1,j=3:\,&(0,1,-m_{2,3})=(Y_{1,3}^{(1)},Y_{1,3}^{(2)},0),&
		i=2,j=1:\,&(0,-m_{2,3},m_{2,3}^2)=(Y_{2,1}^{(1)},Y_{2,1}^{(2)},0),\\
		i=2,j=2:\,&(-am_{2,3},0,am_{1,3}m_{2,3})=(Y_{2,2}^{(1)},Y_{2,2}^{(2)},0),&
		i=2,j=3:\,&(a,0,-am_{1,3})=(Y_{2,3}^{(1)},Y_{2,3}^{(2)},0).
	\end{align*}
	The system of equations can be reduced into
	\begin{align*}
		m_{2,3}&=0,&am_{1,3}&=0,
	\end{align*}
	giving
	\[c_{(0,1),(2,1)}^{\ideal}=\begin{cases}
		1&a\neq0,\\
		q&a=0.
	\end{cases}\]	
	
	For  $M_{(0,1),(1,1)}$, one of the conditions we get is 
	\[i=3,j=1:(0,-1,0)=(Y_{3,1}^{(1)},0,Y_{3,1}^{(3)}),\]
	giving  $c_{(0,1)(1,1)}^{\ideal}=0$. 
	
	For  $M_{(1),(2)}$, we have
	\begin{align*}
		i=2, j=3:\,&(a,0,0)=(0,Y_{2,3}^{(2)},Y_{2,3}^{(3)})&i=3,j=2\,&(-a,0,0)=(0,Y_{3,2}^{(2)},Y_{3,2}^{(3)}).
	\end{align*}
	This requires $a=0$, giving
	\[c_{(1)(2)}^{\ideal}=\begin{cases}
		0&a\neq0,\\
		1&a=0.
	\end{cases}\]
	
	Summing all up, we get \begin{equation*}\zeta_{L_a^4(\Fq)}^{\triangleleft}(s)=\begin{cases}
			1+t+|V_{4,a}(\Fq)|t^2+t^3&a\neq0\\
			1+(1+q)t+t^2+t^3&a=0.\qedhere
		\end{cases}
\end{equation*}\end{proof}	
\begin{rem}
	Note that  $L_0^4(\Fq)$ is isomorphic to the Heisenberg Lie algebra $H(\Fq)$.      Our result for $\zeta_{L_0^4(\Fq)}^{\triangleleft}(s)$ indeed matches $\zeta_{H(\Fq)}^{\triangleleft}(s)$ in \cite[Example 2.9]{Lee25}.
\end{rem}
\subsection{Solvable $\Fq$-Lie algebras of dimension $n=3$, counting subalgebras}
For subalgebras, we know $a_1^{\leq}(L)=a_{q^{n}}^{\leq}(L)=1$ and $a_{q^{2}}^{\leq}(L)=\binom{3}{2}_q=1+q+q^2$. Therefore one only needs to check $g_{i,j,k}^{\leq}(M)$ for  $M_{(2),(1)}$, $M_{(1,1),(1,0)}$, and $M_{(0,2),(1,0)}$.

\begin{thm}
	Let $L^{1}$ be the 3-dimensional abelian $\Fq$-Lie algebra. Then we have
	\[\zeta_{L^{1}(\Fq)}^{\leq}(s)=1+(1+q+q^2)t+(1+q+q^2)t^2+t^3.\]
\end{thm}
\begin{proof}
	Note that $\zeta_{L^{1}(\Fq)}^{\leq}(s)=\zeta_{\Fq^3}(s)$
\end{proof}

\begin{thm}
	Let 
	\[L^2(\Fq)=\langle e_1,e_2,e_3|[e_3,e_1]=e_1,\,[e_3,e_2]=e_2\rangle_{\Fq}.\] 
	Then we have
	\[\zeta_{L^{2}(\Fq)}^{\leq}(s)=\zeta_{L^{1}(\Fq)}^{\leq}(s)=1+(1+q+q^2)t+(1+q+q^2)t^2+t^3.\]
\end{thm}
\begin{proof}   
	By Theorem \ref{thm:sub.iso}.
\end{proof}

\begin{thm} For $a\in\Fq$, let
	\[L_{a}^3(\Fq):=\langle e_{1},e_{2},e_{3}\,\mid\,[e_3,e_1]=e_2, [e_3,e_2]=ae_1+e_2\rangle_{\Fq}.\]
	Then
	
	\begin{equation*}\zeta_{L_a^3(\Fq)}^{\leq}(s)=\begin{cases}
			1+(1+|V_{3,a}(\Fq)|q)t+(1+q+q^2)t^2+t^3&a\neq0\\
			1+(1+2q)t+(1+q+q^2)t^2+t^3&a=0.
		\end{cases}
	\end{equation*}
	where 
	\[|V_{3,a}(\Fq)|=\left|\{x\in\Fq:ax^2-x-1=0 \}\right|.\]
\end{thm}	

\begin{proof}
	Let us compute $a_{q}^{\leq}(L_a^3(\Fq))$.	For $M_{(0,1),(2,0)}$, we get 
	\begin{align*}
		i=1,j=2:\,&(-am_{1,3},-m_{1,3}+m_{2,3},am_{1,3}^2+m_{1,3}m_{2,3}-m_{2,3}^2)=(Y_{1,2}^{(1)},Y_{1,2}^{(2)},0),
	\end{align*}
	giving
	\begin{align*}
		am_{1,3}^2+m_{1,3}m_{2,3}-m_{2,3}^2=0
	\end{align*}
	Suppose $a\neq0$. If $m_{2,3}=0$, then $m_{1,3}=0$ as well and we get 1 solution. If $m_{2,3}\neq0$. then substituting $m_{1.3}'=-m_{1,3}/m_{2,3}$ gives
	\begin{align*}
		am_{1,3}^2+m_{1,3}m_{2,3}-m_{2,3}^2=\left(a\left(m_{1.3}'\right)^2-m_{1.3}'-1\right)m_{2,3}^2=0.
	\end{align*} This gives $|V_{3,a}(\Fq)|(q-1)$ other solutions. 
	
	Suppose $a=0$. Then we need to solve $m_{2,3}(m_{1,3}-m_{2,3})=0$. If $m_{2,3}=0$ one can have $1$ different $m_{1,3}$. If $m_{2,3}\neq0$ then one needs $m_{1,3}=m_{2,3}$. Hence we get $2q-1$ solutions. Therefore we have
	\[c_{(0,1),(2,0)}^{\leq}=\begin{cases}
		|V_{3,a}(\Fq)|(q-1)+1&a\neq0,\\
		2q-1&a=0.
	\end{cases}\]
	
	For  $M_{(0,1),(1,1)}$, one of the conditions we get is 
	\[i=1,j=3:(am_{1,2},1+m_{1,2}-am_{1,2}^2,0)=(Y_{3,1}^{(1)},0,Y_{3,1}^{(3)}),\]
	giving  $c_{(0,1)(1,1)}^{\leq}=|V_{3,a}(\Fq)|$. 
	
	Finally, for  $M_{(1),(2)}$, we have
	\begin{align*}
		i=2,j=3:\,&(a,1,0)=(0,Y_{2,3}^{(2)},Y_{2,3}^{(3)}).
	\end{align*}
	This requires $a=0$, giving
	\[c_{(1)(2)}^{\leq}=\begin{cases}
		0&a\neq0,\\
		1&a=0.
	\end{cases}\]
	Together, we have
	\[a_{q}^{\leq}(L_a^3)=\begin{cases}
		q|V_{3,a}(\Fq)|+1&a\neq0\\
		|V_{3,0}(\Fq)|+2q=1+2q&a= 0.
	\end{cases}\]
	Summing all up, we get 
	\begin{equation*}\zeta_{L_a^3(\Fq)}^{\leq}(s)=\begin{cases}
			1+(1+|V_{3,a}(\Fq)|q)t+(1+q+q^2)t^2+t^3&a\neq0\\
			1+(1+2q)t+(1+q+q^2)t^2+t^3&a=0.
		\end{cases}\qedhere
	\end{equation*}
\end{proof}

\begin{thm}\label{thm:sol.L4a} For $a\in\Fq$, let
	\[L_{a}^4(\Fq):=\langle e_{1},e_{2},e_{3}\,\mid\,[e_3,e_1]=e_2, [e_3,e_2]=ae_1\rangle_{\Fq}.\]
	Then
	
	\begin{equation*}\zeta_{L_a^4(\Fq)}^{\leq}(s)=\begin{cases}
			1+(1+|V_{4,a}(\Fq)|q)t+t^2+t^3&a\neq0\\
			1+(1+q)t+(1+q+q^2)t^2+t^3&a=0.
		\end{cases}
	\end{equation*}
	where 
	\[|V_{4,a}(\Fq)|=\left|\{x\in\Fq:ax^2-1=0 \}\right|.\]
\end{thm}	

\begin{proof}
	Let us compute $a_{q}^{\leq}(L_a^4(\Fq))$. For $M_{(0,1),(2,0)}$, we get 
	\begin{align*}
		i=1,j=2:\,&(-am_{1,3},m_{2,3},am_{1,3}^2-m_{2,3}^2)=(Y_{1,2}^{(1)},Y_{1,2}^{(2)},0),
	\end{align*}giving
	\begin{align*}
		am_{1,3}^2-m_{2,3}^2=0
	\end{align*}
	suppose $a\neq0$. If $m_{2,3}=0$, then $m_{1,3}=0$ as well and we get 1 solution. If $m_{2,3}\neq0$. then substituting $m_{1.3}'=m_{1,3}/m_{2,3}$ gives
	\begin{align*}
		am_{1,3}^2-m_{2,3}^2=\left(a\left(m_{1.3}'\right)^2-1\right)m_{2,3}^2=0.
	\end{align*} This gives $|V_{4,a}(\Fq)|(q-1)$ other solutions.
	
	Suppose $a=0$. Then we have $m_{2,3}=0$ and $q$ possible values of $m_{1,3}$.   Therefore we have
	\[c_{(0,1),(2,0)}^{\leq}=\begin{cases}
		|V_{4,a}(\Fq)|(q-1)+1&a\neq0,\\
		q&a=0.
	\end{cases}\]

	For  $M_{(0,1),(1,1)}$, one of the conditions we get is 
	\[i=1,j=3:(am_{1,2},1-am_{1,2}^2,0)=(Y_{3,1}^{(1)},0,Y_{3,1}^{(3)}),\]
	giving  $1-am_{1,2}^2=0$. 
	If $a\neq0$, then we have $|V_{4,a}(\Fq)|$ solutions. If $a=0$, then the condition becomes unsolvable. Therefore we have
	\[c_{(0,1),(1,1)}^{\leq}=\begin{cases}
		|V_{4,a}(\Fq)|&a\neq0,\\
		0&a=0.
	\end{cases}\]

	For  $M_{(1),(2)}$, we have
	\begin{align*}
		i=2, j=3:\,&(a,0,0)=(0,Y_{2,3}^{(2)},Y_{2,3}^{(3)}),
	\end{align*}
	giving
	\[c_{(1)(2)}^{\leq}=\begin{cases}
		0&a\neq0,\\
		1&a=0.
	\end{cases}\]
	Summing all up, we get \begin{equation*}\zeta_{L_a^4(\Fq)}^{\leq}(s)=\begin{cases}
			1+(1+|V_{4,a}(\Fq)|q)t+t^2+t^3&a\neq0\\
			1+(1+q)t+(1+q+q^2)t^2+t^3&a=0.
		\end{cases}\qedhere
	\end{equation*}
\end{proof}	
\begin{rem}
	Note that our result for $\zeta_{L_0^4(\Fq)}^{\leq}(s)$ matches $\zeta_{H(\Fq)}^{\leq}(s)$ in \cite[Example 2.9]{Lee25}.
\end{rem}

\section{Explicit computations of $\zeta_{L}^{*}(s)$ for solvable $\Fp$-Lie algebras of dimension 4}\label{sec:dim4}
\subsection{Solvable $\Fq$-Lie algebras of dimension 4, counting ideals}
Suppose $L$ is a $4$-dimensional $\Fq$-Lie algebra. Any subspaces of $L$ of codimension $k$ can be represented by one of the following 16 diagonal cells:
\begin{itemize}
	\item $M_{(0),(4)}$ if $k=0$,
	\item $M_{(0,1),(3,0)}$,  $M_{(0,1),(2,1)}$,
	$M_{(0,1),(1,2)}$, or $M_{(1),(3)}$ if $k=1$,
	\item $M_{(0,2),(2,0)}$, $M_{(0,1,1),(1,1,0)}$, 
	$M_{(0,2),(1,1)}$, $M_{(1,1),(0,2)}$, $M_{(1,1),(1,1)}$, or $M_{(2),(2)}$ if $k=2$,
	\item $M_{(0,3),(1,0)}$, $M_{(1,2),(1,0)}$, $M_{(2,1),(1,0)}$, $M_{(3),(1)}$ if $k=3$,
	\item $M_{(4),(0)}$ if $k=4$.
\end{itemize}
Like $3$-dimensional cases, for any 4-dimensional $\Fq$-Lie algebra $L$ we always have $a_{1}^{\ideal}(L)=a_{q^4}^{\ideal}(L)=1$. Hence we only focus on $a_{q}^{\ideal}(L)$, $a_{q^2}^{\ideal}(L)$, and $a_{q^3}^{\ideal}(L)$.

\begin{thm}
	Let $M^{1}$ be the 4-dimensional abelian $\Fq$-Lie algebra. Then we have
	\[\zeta_{M^{1}(\Fq)}^{\ideal}(s)=1+\binom{4}{1}_qt+\binom{4}{2}_qt^2+\binom{4}{3}_qt^3+t^4.\]
\end{thm}
\begin{proof}
	Note that $\zeta_{M^{1}(\Fq)}^{\ideal}(s)=\zeta_{\Fq^4}(s)$.
\end{proof}
\begin{thm}
	Let 
	\[M^2:=\langle e_1,e_2,e_3,e_4:[e_4,e_1]=e_1,[e_4,e_2]=e_2,[e_4,e_3]=e_3\rangle_{\Fq}.\]
	We have
	\[\zeta_{M^2(\Fq)}^{\ideal}(s)=1+t+(1+q+q^2)t^{2}+(1+q+q^2)t^{3}+t^{4}.\]
\end{thm}

\begin{proof}
	Let us count the number of ideals of index $q^3$. For $M_{(0,3),(1,0)}$, we get
	\begin{align*}
		i=1,j=1:\,&(-m_{1,4},m_{1,2}m_{1,4},m_{1,3}m_{1,4},m_{1,4}^2)=(Y_{1,1}^{(1)},0,0,0),\\
		i=1,j=2:\,&(0,-m_{1,4},0,0)=(Y_{1,2}^{(1)},0,0,0),\\
		i=1,j=3:\,&(0,0,-m_{1,4},0)=(Y_{1,3}^{(1)},0,0,0),\\
		i=1,j=4:\,&(0,0,0,-m_{1,4})=(Y_{1,4}^{(1)},0,0,0),
	\end{align*}
	giving $c_{(0,3),(1,0)}^{\ideal}=q^2$.
	
	For $M_{(1,2),(1,0)}$, we get
	\begin{align*}
		i=2,j=1:\,&(-m_{2,4},0,0,0)=(0,Y_{2,1}^{(2)},0,0),\\
		i=2,j=2:\,&(0,-m_{2,4},m_{2,3}m_{2,4},m_{2,4}^2)=(0,Y_{2,2}^{(2)},0,0),\\
		i=2,j=3:\,&(0,0,-m_{2,4},0)=(0,Y_{2,3}^{(2)},0,0),\\
		i=2,j=4:\,&(0,1,0,-m_{2,4})=(0,Y_{2,4}^{(2)},0,0),
	\end{align*}
	giving $c_{(1,2),(1,0)}^{\ideal}=q$.
	
	For $M_{(2,1),(1,0)}$, we get
	\begin{align*}
		i=3,j=1:\,&(-m_{3,4},0,0,0)=(0,0,Y_{3,1}^{(3)},0),\\
		i=3,j=2:\,&(0,-m_{3,4},0,0)=(0,0,Y_{3,2}^{(3)},0),\\
		i=3,j=3:\,&(0,0,-m_{3,4},m_{3,4}^2)=(0,0,Y_{3,3}^{(3)},0),\\
		i=3,j=4:\,&(0,0,1,-m_{3,4})=(0,0,Y_{3,4}^{(3)},0),
	\end{align*}
	giving $c_{(2,1),(1,0)}^{\ideal}=1$. 
	
	For $M_{(3),(1)}$ we get 
	\begin{align*}
		i=4,j=1:\,(-1,0,0,0)=(0,0,0,Y_{4,1}^{(4)}),
	\end{align*}
	giving $c_{(3)(1)}^{\ideal}=0$. Together, we have 
	\[a_{q^3}^{\ideal}(M^2)=q^2+q+1.\]
	
	Let us count the number of ideals of index $q^2$. For $M_{(0,2),(2,0)}$, we get
	\begin{align*}
		i=1,j=1:\,&(-m_{1,4},0,m_{1,3}m_{1,4},m_{1,4}^2)=(Y_{1,1}^{(1)},Y_{1,1}^{(2)},0,0),\\
		i=1,j=2:\,&(0,-m_{1,4},m_{1,4}m_{2,3},m_{1,4}m_{2,4})=(Y_{1,2}^{(1)},Y_{1,2}^{(2)},0,0),\\
		i=1,j=3:\,&(0,0,-m_{1,4},0)=(Y_{1,3}^{(1)},Y_{1,3}^{(2)},0,0),\\
		i=1,j=4:\,&(1,0,0,-m_{1,4})=(Y_{1,4}^{(1)},Y_{1,4}^{(2)},0,0),\\
		i=2,j=1:\,&(-m_{2,4},0,m_{1,3}m_{2,4},m_{1,4}m_{2,4})=(Y_{2,1}^{(1)},Y_{2,1}^{(2)},0,0),\\
		i=2,j=2:\,&(0,-m_{2,4},m_{2,3}m_{2,4},m_{2,4}^2)=(Y_{2,2}^{(1)},Y_{2,2}^{(2)},0,0),\\
		i=2,j=3:\,&(0,0,-m_{2,4},0)=(Y_{2,3}^{(1)},Y_{2,3}^{(2)},0,0),\\
		i=2,j=4:\,&(0,1,0,-m_{2,4})=(Y_{2,4}^{(1)},Y_{2,4}^{(2)},0,0),
	\end{align*}
	giving $c_{(0,2),(2,0)}^{\ideal}=q^2$.
	
	For $M_{(0,1,1),(1,1,0)}$, we get
	\begin{align*}
		i=1,j=1:\,&(-m_{1,4},m_{1,2}m_{1,4},0,m_{1,4}^2)=(Y_{1,1}^{(1)},0,Y_{1,1}^{(3)},0),\\
		i=1,j=2:\,&(0,-m_{1,4},0,0)=(Y_{1,2}^{(1)},0,Y_{1,2}^{(3)},0),\\
		i=1,j=3:\,&(0,0,-m_{1,4},m_{1,4}m_{3,4})=(Y_{1,3}^{(1)},0,Y_{1,3}^{(3)},0),\\
		i=1,j=4:\,&(1,0,0,-m_{1,4})=(Y_{1,4}^{(1)},0,Y_{1,4}^{(3)},0),\\
		i=3,j=1:\,&(-m_{3,4},m_{1,2}m_{3,4},0,m_{1,4}m_{3,4})=(Y_{3,1}^{(1)},0,Y_{3,1}^{(3)},0),\\
		i=3,j=2:\,&(0,-m_{3,4},0,0)=(Y_{3,2}^{(1)},0,Y_{3,2}^{(3)},0),\\
		i=3,j=3:\,&(0,0,-m_{3,4},m_{3,4}^2)=(Y_{3,3}^{(1)},0,Y_{3,3}^{(3)},0),\\
		i=3,j=4:\,&(0,0,1,-m_{3,4})=(Y_{3,4}^{(1)},0,Y_{3,4}^{(3)},0),
	\end{align*}
	giving $c_{(0,1,1),(1,1,0)}^{\ideal}=q$.
	
	For $M_{(0,2),(1,1)}$, we get
	\begin{align*}
		i=4,j=2:\,(0,-1,0,0)=(Y_{4,2}^{(1)},0,0,Y_{4,2}^{(4)}),
	\end{align*}
	giving $c_{(0,2),(1,1)}^{\ideal}=0$. 
	
	For $M_{(1,1),(2,0)}$, we get
	\begin{align*}
		i=2,j=1:\,&(-m_{2,4},0,0,0)=(0,Y_{2,1}^{(2)},Y_{2,1}^{(3)},0),\\
		i=2,j=2:\,&(0,-m_{2,4},0,m_{2,4}^2)=(0,Y_{2,2}^{(2)},Y_{2,2}^{(3)},0),\\
		i=2,j=3:\,&(0,0,-m_{2,4},m_{2,4}m_{3,4})=(0,Y_{2,3}^{(2)},Y_{2,3}^{(3)},0),\\
		i=2,j=4:\,&(0,1,0,-m_{2,4})=(0,Y_{2,4}^{(2)},Y_{2,4}^{(3)},0),\\ 
		i=3,j=1:\,&(-m_{3,4},0,0,0)=(0,Y_{3,1}^{(2)},Y_{3,1}^{(3)},0),\\
		i=3,j=2:\,&(0,-m_{3,4},0,m_{2,4}m_{3,4})=(0,Y_{3,2}^{(2)},Y_{3,2}^{(3)},0),\\
		i=3,j=3:\,&(0,0,-m_{3,4},m_{3,4}^2)=(0,Y_{3,3}^{(2)},Y_{3,3}^{(3)},0),\\
		i=3,j=4:\,&(0,0,1,-m_{3,4})=(0,Y_{3,4}^{(2)},Y_{3,4}^{(3)},0),
	\end{align*}
	giving $c_{(1,1),(2,0)}^{\ideal}=1$.
	
	For $M_{(1,1),(1,1)}$, we get
	\begin{align*}
		i=4,j=1:\,(-1,0,0,0)=(0,Y_{4,1}^{(1)},0,Y_{4,1}^{(4)}),
	\end{align*}
	giving $c_{(1,1),(1,1)}^{\ideal}=0$. 
	
	For $M_{(2),(2)}$, we get
	\begin{align*}
		i=4,j=1:\,(-1,0,0,0)=(0,0,Y_{4,1}^{(3)},Y_{4,1}^{(4)}),
	\end{align*}
	giving $c_{(2),(2)}^{\ideal}=0$. Together, we have 
	\[a_{q^2}^{\ideal}(M^2)=q^2+q+1.\]
	
	Finally, let us count the number of ideals of index $q$. For $M_{(0,1),(3,0)}$, we get
	\begin{align*}
		i=1,j=1:\,&(-m_{1,4},0,0,m_{1,4}^2)=(Y_{1,1}^{(1)},Y_{1,1}^{(2)},Y_{1,1}^{(3)},0),\\
		i=1,j=2:\,&(0,-m_{1,4},0,m_{1,4}m_{2,4})=(Y_{1,2}^{(1)},Y_{1,2}^{(2)},Y_{1,3}^{(3)},0),\\
		i=1,j=3:\,&(0,0,-m_{1,4},m_{1,4}m_{3,4})=(Y_{1,3}^{(1)},Y_{1,3}^{(2)},Y_{1,3}^{(3)},0),\\
		i=1,j=4:\,&(1,0,0,-m_{1,4})=(Y_{1,4}^{(1)},Y_{1,4}^{(2)},Y_{1,4}^{(3)},0),\\
		i=2,j=1:\,&(-m_{2,4},0,0,m_{1,4}m_{2,4})=(Y_{2,1}^{(1)},Y_{2,1}^{(2)},Y_{2,1}^{(3)},0),\\
		i=2,j=2:\,&(0,-m_{2,4},0,m_{2,4}^2)=(Y_{2,2}^{(1)},Y_{2,2}^{(2)},Y_{2,2}^{(3)},0),\\
		i=2,j=3:\,&(0,0,-m_{2,4},m_{2,4}m_{3,4})=(Y_{2,3}^{(1)},Y_{2,3}^{(2)},Y_{2,3}^{(3)},0),\\
		i=2,j=4:\,&(0,1,0,-m_{2,4})=(Y_{2,4}^{(1)},Y_{2,4}^{(2)},Y_{2,4}^{(3)},0),\\
		i=3,j=1:\,&(-m_{3,4},0,0,m_{1,4}m_{3,4})=(Y_{3,1}^{(1)},Y_{3,1}^{(2)},Y_{3,1}^{(3)},0),\\
		i=3,j=2:\,&(0,-m_{3,4},0,m_{2,4}m_{3,4})=(Y_{3,2}^{(1)},Y_{3,2}^{(2)},Y_{3,2}^{(3)},0),\\
		i=3,j=3:\,&(0,0,-m_{3,4},m_{3,4}^2)=(Y_{3,3}^{(1)},Y_{3,3}^{(2)},Y_{3,3}^{(3)},0),\\
		i=3,j=4:\,&(0,0,1,-m_{3,4})=(Y_{3,4}^{(1)},Y_{3,4}^{(2)},Y_{3,4}^{(3)},0),
	\end{align*}
	giving $c_{(0,1),(3,0)}^{\ideal}=1$.
	
	For $M_{(0,1),(2,1)}$, we get
	\begin{align*}
		i=4,j=3:\, (0,0,-1,0)=(Y_{4,3}^{(1)},Y_{4,3}^{(2)},0,Y_{4,3}^{(4)}),
	\end{align*}
	giving $c_{(0,1),(2,1)}^{\ideal}=0$. 
	
	For $M_{(0,1),(1,2)}$, we get
	\begin{align*}
		i=4,j=2:\, (0,-1,0,0)=(Y_{4,2}^{(1)},0,Y_{4,2}^{(3)},Y_{4,2}^{(4)}),
	\end{align*}
	giving $c_{(0,1),(1,2)}^{\ideal}=0$.
	
	for $M_{(1),(3)}$,  we get
	\begin{align*}
		i=4,j=1:\, (-1,0,0,0)=(0,Y_{4,1}^{(2)},Y_{4,1}^{(3)},Y_{4,1}^{(4)}),
	\end{align*}
	giving $c_{(1),(3)}^{\ideal}=0$. Together, we have 
	\[a_{q}^{\ideal}(M^2)=1.\]
	
	Summing all up, we get 	\[\zeta_{M^2(\Fq)}^{\triangleleft}(s)=1+t+(1+q+q^2)t^{2}+(1+q+q^2)t^{3}+t^{4}.\qedhere\] 
\end{proof}

\begin{thm}
	For $a\in\Fq$, let 
	\[M_a^3:=\langle e_1,e_2,e_3,e_4:[e_4,e_1]=e_1,[e_4,e_2]=e_3,[e_4,e_3]=-ae_2+(a+1)e_3\rangle_{\Fq}.\]
	We have
	\begin{equation*}
		\zeta_{M_a^3(\Fq)}^{\triangleleft}(s)=\begin{cases}
			1+(q+1)t+(q+2)t^2+(q+2)t^3+t^4&a=0,\\
			1+t+(q+1)t^2+(q+1)t^3+t^4&a=1,\\
			1+t+(q+2)t^2+(q+2)t^3+t^4&a\neq0,1.\\
		\end{cases} 
	\end{equation*}
\end{thm}

\begin{proof}
	Let us count the number of ideals of index $q^3$. For $M_{(0,3),(1,0)}$, we get
	\begin{align*}
		i=1,j=1:\,&(-m_{1,4},m_{1,2}m_{1,4},m_{1,3}m_{1,4},m_{1,4}^2)=(Y_{1,1}^{(1)},0,0,0),\\
		i=1,j=2:\,&(0,0,-m_{1,4},0)=(Y_{1,2}^{(1)},0,0,0),\\
		i=1,j=3:\,&(0,am_{1,4},-(a+1)m_{1,4},0)=(Y_{1,3}^{(1)},0,0,0),\\
		i=1,j=4:\,&(1,-m_{1,2}-am_{1,3},m_{1,2}+am_{1,3},-m_{1,4})=(Y_{1,4}^{(1)},0,0,0),
	\end{align*}
	which gives the system of equations
	\begin{align*}
		m_{1,4}&=0,&m_{1,2}+am_{1,3}^2&=0.
	\end{align*}
	Therefore we get $c_{(0,3),(1,0)}^{\ideal}=q$.
	
	For $M_{(1,2),(1,0)}$, we get
	\begin{align*}
		i=2,j=1:\,&(-m_{2,4},0,0,0)=(0,Y_{2,1}^{(2)},0,0),\\
		i=2,j=2:\,&(0,0,-m_{2,4},0)=(0,Y_{2,2}^{(2)},0,0),\\
		i=2,j=3:\,&(0,am_{2,4},-(1+a+am_{2,3})m_{2,4},-am_{2,4}^2)=(0,Y_{2,3}^{(2)},0,0),\\
		i=2,j=4:\,&(0,-am_{2,3},(1+m_{2,3})(1+am_{2,3}),am_{2,3}m_{2,4})=(0,Y_{2,4}^{(2)},0,0),
	\end{align*}
	which gives the system of equations
	\begin{align*}
		m_{2,4}&=0, &(1+m_{2,3})(1+am_{2,3})&=0.
	\end{align*}
	If $a\neq0,1$, then $m_{2,3}\in\Fq$ must be either $-1$ or $-1/a$, giving 2 distinct solution. If $a=0$ or $1$, then we must have $m_{2,3}=-1$.  Therefore we have
	\[c_{(1,2),(1,0)}^{\ideal}=\begin{cases}
		2&a\neq0,1,\\
		1&a=0,1.
	\end{cases}\]
	
	For $M_{(2,1),(1,0)}$, we get
	\begin{align*}
		i=3,j=1:\,&(-m_{3,4},0,0,0)=(0,0,Y_{3,1}^{(3)},0),\\
		i=3,j=2:\,&(0,0,-m_{3,4},m_{3,4}^2)=(0,0,Y_{3,2}^{(3)},0),\\
		i=3,j=3:\,&(0,am_{3,4},-(1+a)m_{3,4},(1+a)m_{3,4}^2)=(0,0,Y_{3,3}^{(3)},0),\\
		i=3,j=4:\,&(0,-a,1+a,-(1+a)m_{3,4})=(0,0,Y_{3,4}^{(3)},0),
	\end{align*}
	which gives the system of equations
	\begin{align*}
		m_{3,4}&=0,&a&=0.
	\end{align*}
	Therefore we have
	\[c_{(2,1),(1,0)}^{\ideal}=\begin{cases}
		0&a\neq0,\\
		1&a=0.
	\end{cases}\]

	For $M_{(3),(1)}$ we get 
	\begin{align*}
		i=4,j=1:\,(-1,0,0,0)=(0,0,0,Y_{4,1}^{(4)}),
	\end{align*}
	giving $c_{(3)(1)}^{\ideal}=0$. Together, we have
	\[a_{q^3}^{\ideal}(M_a^3)=\begin{cases}
		q+2&a\neq1,\\
		q+1&a=1.
	\end{cases}\]

	Let us count the number of ideals of index $q^2$. For $M_{(0,2),(2,0)}$, we get
	\begin{align*}
		i=1,j=1:\,&(-m_{1,4},0,m_{1,3}m_{1,4},m_{1,4}^2)=(Y_{1,1}^{(1)},Y_{1,1}^{(2)},0,0),\\
		i=1,j=2:\,&(0,0,-m_{1,4},0)=(Y_{1,2}^{(1)},Y_{1,2}^{(2)},0,0),\\
		i=1,j=3:\,&(0,am_{1,4},-m_{1,4}(1+a+am_{2,3}),-am_{1,4}m_{2,4})=(Y_{1,3}^{(1)},Y_{1,3}^{(2)},0,0),\\
		i=1,j=4:\,&(1,-am_{1,3},am_{1,3}(1+m_{2,3}),-m_{1,4}+am_{1,3}m_{2,4})=(Y_{1,4}^{(1)},Y_{1,4}^{(2)},0,0),\\
		i=2,j=1:\,&(-m_{2,4},0,m_{1,3}m_{2,4},m_{1,4}m_{2,4})=(Y_{2,1}^{(1)},Y_{2,1}^{(2)},0,0),\\
		i=2,j=2:\,&(0,0,-m_{2,4},0)=(Y_{2,2}^{(1)},Y_{2,2}^{(2)},0,0),\\
		i=2,j=3:\,&(0,am_{2,4},-(1+a+am_{2,3})m_{2,4},-am_{2,4}^2)=(Y_{2,3}^{(1)},Y_{2,3}^{(2)},0,0),\\
		i=2,j=4:\,&(0,-am_{2,3},(1+m_{2,3})(1+am_{2,3}),am_{2,3}m_{2,4})=(Y_{2,4}^{(1)},Y_{2,4}^{(2)},0,0),
	\end{align*}
	which gives the system of equations
	\begin{align*}
		m_{1,4}&=0, &m_{2,4}&=0,\\
		(1+m_{2,3})(1+am_{2,3})&=0,&am_{1,3}(1+m_{2,3})&=0.
	\end{align*}
	If $a\neq0,1$, then $m_{2,3}\in\Fq$ must be either $-1$ or $-1/a$. If $m_{2,3}=-1$, then one can have $q$ different $m_{1,3}\in\Fq$. If $m_{2,3}=-1/a$, then $m_{1,3}$=0.
	
	If $a=0,1$ then we must have $m_{2,3}=-1$, giving us $q$ free choices of $m_{1,3}$. Therefore we have
	\[c_{(0,2),(2,0)}^{\ideal}=\begin{cases}
		1+q&a\neq0,1,\\
		q&a=0,1.
	\end{cases}\]
	
	For $M_{(0,1,1),(1,1,0)}$, we get
	\begin{align*}
		i=1,j=1:\,&(-m_{1,4},m_{1,2}m_{1,4},0,m_{1,4}^2)=(Y_{1,1}^{(1)},0,Y_{1,1}^{(3)},0),\\
		i=1,j=2:\,&(0,0,-m_{1,4},m_{1,4}m_{3,4})=(Y_{1,2}^{(1)},0,Y_{1,2}^{(3)},0),\\
		i=1,j=3:\,&(0,am_{1,4},-(1+a)m_{1,4},(1+a)m_{1,4}m_{3,4})=(Y_{1,3}^{(1)},0,Y_{1,3}^{(3)},0),\\
		i=1,j=4:\,&(1,-m_{1,2},-m_{1,2},-m_{1,4}-m_{1,2}m_{3,4})=(Y_{1,4}^{(1)},0,Y_{1,4}^{(3)},0),\\
		i=3,j=1:\,&(-m_{3,4},m_{1,2}m_{3,4},0,m_{1,4}m_{3,4})=(Y_{3,1}^{(1)},0,Y_{3,1}^{(3)},0),\\
		i=3,j=2:\,&(0,0,-m_{3,4},m_{3,4}^2)=(Y_{3,2}^{(1)},0,Y_{3,2}^{(3)},0),\\
		i=3,j=3:\,&(0,am_{3,4},-(1+a)m_{3,4},(1+a)m_{3,4}^2)=(Y_{3,3}^{(1)},0,Y_{3,3}^{(3)},0),\\
		i=3,j=4:\,&(0,-a,1+a,-(1+a)m_{3,4})=(Y_{3,4}^{(1)},0,Y_{3,4}^{(3)},0),
	\end{align*}
	which gives the system of equations
	\begin{align*}
		m_{1,2}&=0, &m_{1,4}&=0,\\
		m_{3,4}&=0,&a&=0.
	\end{align*}
	Therefore we have
	\[c_{(0,1,1),(1,1,0)}^{\ideal}=\begin{cases}
		0&a\neq0,\\
		1&a=0.
	\end{cases}\]
	
	For $M_{(0,2),(1,1)}$, we get
	\begin{align*}
		i=4,j=2:\,(0,0,-1,0)=(Y_{4,2}^{(1)},0,0,Y_{4,2}^{(4)}),
	\end{align*}
	giving $c_{(0,2),(1,1)}^{\ideal}=0$. 
	
	For $M_{(1,1),(2,0)}$, we get
	\begin{align*}
		i=2,j=1:\,&(-m_{2,4},0,0,0)=(0,Y_{2,1}^{(2)},Y_{2,1}^{(3)},0),\\
		i=2,j=2:\,&(0,0,-m_{2.4},m_{2,4}m_{3,4})=(0,Y_{2,2}^{(2)},Y_{2,2}^{(3)},0),\\
		i=2,j=3:\,&(0,am_{2,4},-(1+a)m_{2,4},m_{2,4}(m_{3,4}+a(-m_{2,4}+m_{3,4})))=(0,Y_{2,3}^{(2)},Y_{2,3}^{(3)},0),\\
		i=2,j=4:\,&(0,0,1,-m_{3,4})=(0,Y_{2,4}^{(2)},Y_{2,4}^{(3)},0),\\ 
		i=3,j=1:\,&(-m_{3,4},0,0,0)=(0,Y_{3,1}^{(2)},Y_{3,1}^{(3)},0),\\
		i=3,j=2:\,&(0,0,-m_{3,4},m_{3,4}^2)=(0,Y_{3,2}^{(2)},Y_{3,2}^{(3)},0),\\
		i=3,j=3:\,&(0,am_{3,4},-(1+a)m_{3,4},m_{3,4}(m_{3,4}+a(-m_{2,4}+m_{3,4})))=(0,Y_{3,3}^{(2)},Y_{3,3}^{(3)},0),\\
		i=3,j=4:\,&(0,-a,1+a,am_{2,4}-(1+a)m_{3,4})=(0,Y_{3,4}^{(2)},Y_{3,4}^{(3)},0),
	\end{align*}
	which gives the system of equations
	\begin{align*}
		m_{2,4}&=0, &m_{3,4}&=0.
	\end{align*}
	Therefore we have
	\[c_{(1,1),(2,0)}^{\ideal}=1.\]
	
	For $M_{(1,1),(1,1)}$, we get
	\begin{align*}
		i=4,j=1:\,(-1,0,0,0)=(0,Y_{4,1}^{(1)},0,Y_{4,1}^{(4)}),
	\end{align*}
	giving $c_{(1,1),(1,1)}^{\ideal}=0$. 
	
	For $M_{(2),(2)}$, we get
	\begin{align*}
		i=4,j=1:\,(-1,0,0,0)=(0,0,Y_{4,1}^{(3)},Y_{4,1}^{(4)}),
	\end{align*}
	giving $c_{(2),(2)}^{\ideal}=0$. Together, we have 
	\[a_{q^2}^{\ideal}(M_a^3)=\begin{cases}
		q+2&a\neq1,\\
		q+1&a=1.
	\end{cases}\]

	Finally, let us count the number of ideals of index $q$. For $M_{(0,1),(3,0)}$, we get
	\begin{align*}
		i=1,j=1:\,&(-m_{1,4},0,0,m_{1,4}^2)=(Y_{1,1}^{(1)},Y_{1,1}^{(2)},Y_{1,1}^{(3)},0),\\
		i=1,j=2:\,&(0,0,-m_{1,4},m_{1,4}m_{3,4})=(Y_{1,2}^{(1)},Y_{1,2}^{(2)},Y_{1,3}^{(3)},0),\\
		i=1,j=3:\,&(0,am_{1,4},-(1+a)m_{1,4},m_{1,4}(m_{3,4}+a(m_{3,4}-m_{2,4})))=(Y_{1,3}^{(1)},Y_{1,3}^{(2)},Y_{1,3}^{(3)},0),\\
		i=1,j=4:\,&(1,0,0,-m_{1,4})=(Y_{1,4}^{(1)},Y_{1,4}^{(2)},Y_{1,4}^{(3)},0),\\
		i=2,j=1:\,&(-m_{2,4},0,0,m_{1,4}m_{2,4})=(Y_{2,1}^{(1)},Y_{2,1}^{(2)},Y_{2,1}^{(3)},0),\\
		i=2,j=2:\,&(0,0,-m_{2,4},m_{2,4}m_{3,4})=(Y_{2,2}^{(1)},Y_{2,2}^{(2)},Y_{2,2}^{(3)},0),\\
		i=2,j=3:\,&(0,am_{2,4},-(1+a)m_{2,4},m_{2,4}(m_{3,4}+a(m_{3,4}-m_{2,4})))=(Y_{2,3}^{(1)},Y_{2,3}^{(2)},Y_{2,3}^{(3)},0),\\
		i=2,j=4:\,&(0,0,1,-m_{3,4})=(Y_{2,4}^{(1)},Y_{2,4}^{(2)},Y_{2,4}^{(3)},0),\\
		i=3,j=1:\,&(-m_{3,4},0,0,m_{1,4}m_{3,4})=(Y_{3,1}^{(1)},Y_{3,1}^{(2)},Y_{3,1}^{(3)},0),\\
		i=3,j=2:\,&(0,0,-m_{3,4},m_{3,4}^2)=(Y_{3,2}^{(1)},Y_{3,2}^{(2)},Y_{3,2}^{(3)},0),\\
		i=3,j=3:\,&(0,am_{3,4},-(1+a)m_{3,4},m_{3,4}(m_{3,4}+a(m_{3,4}-m_{2,4})))=(Y_{3,3}^{(1)},Y_{3,3}^{(2)},Y_{3,3}^{(3)},0),\\
		i=3,j=4:\,&(0,-a,1+a,am_{2,4}-(1+a)m_{3,4})=(Y_{3,4}^{(1)},Y_{3,4}^{(2)},Y_{3,4}^{(3)},0),
	\end{align*}
	which gives the system of equations
	\begin{align*}
		m_{1,4}&=0, &m_{3,4}&=0,&am_{2,4}&=0.
	\end{align*}
	If $a\neq0$, then $m_{2,4}$ must be zero. If $a=0$, then we have $q$ possible values of $m_{2,4}$. Therefore we have
	\[c_{(0,1),(3,0)}^{\ideal}=\begin{cases}
		1&a\neq0,\\
		q&a=0.
	\end{cases}\]
	
	For $M_{(0,1),(2,1)}$, we get
	\begin{align*}
		i=4,j=2:\, (0,0,-1,0)=(Y_{4,2}^{(1)},Y_{4,2}^{(2)},0,Y_{4,2}^{(4)}),
	\end{align*}
	giving $c_{(0,1),(2,1)}^{\ideal}=0$. 
	
	For $M_{(0,1),(1,2)}$, we get
	\begin{align*}
		i=1, j=4:\,&(1,-m_{1,2},m_{1,2},0)=(Y_{1,4}^{(1)},0,Y_{1,4}^{(3)},Y_{1,4}^{(4)}),\\
		i=3, j=4:\,&(0,-a,1+a,0)=(Y_{3,4}^{(1)},0,Y_{3,4}^{(3)},Y_{3,4}^{(4)}),\\
		i=4,j=1:\,&(-1,m_{1,2},0,0)=(Y_{4,1}^{(1)},0,Y_{4,1}^{(3)},Y_{4,1}^{(4)})\\
		i=4,j=3:\, &(0,a,-1-a,0)=(Y_{4,3}^{(1)},0,Y_{4,3}^{(3)},Y_{4,3}^{(4)}),
	\end{align*}which gives the system of equations
	\begin{align*}
		a&=0, &m_{1,2}&=0.
	\end{align*}
	Therefore we have
	\[c_{(0,1),(1,2)}^{\ideal}=\begin{cases}
		0&a\neq0,\\
		1&a=0.
	\end{cases}\]
	
	for $M_{(1),(3)}$,  we get
	\begin{align*}
		i=4,j=1:\, (-1,0,0,0)=(0,Y_{4,1}^{(2)},Y_{4,1}^{(3)},Y_{4,1}^{(4)}),
	\end{align*}
	giving $c_{(1),(3)}=0$. Together, we have 
	\[a_{q}^{\ideal}(M_a^3)=\begin{cases}
		1&a\neq0,\\
		q+1&a=0.
	\end{cases}\]
	
	Summing all up, we get 	
	\begin{equation*}
		\zeta_{M_a^3(\Fq)}^{\triangleleft}(s)=\begin{cases}
			1+(q+1)t+(q+2)t^2+(q+2)t^3+t^4&a=0,\\
			1+t+(q+1)t^2+(q+1)t^3+t^4&a=1,\\
			1+t+(q+2)t^2+(q+2)t^3+t^4&a\neq0,1.\qedhere\\
		\end{cases} 
	\end{equation*}
\end{proof}

\begin{thm}
	Let 
	\[M^4:=\langle e_1,e_2,e_3,e_4:[e_4,e_2]=e_3,[e_4,e_3]=e_3\rangle_{\Fq}.\]
	We have
	\[\zeta_{M^4(\Fq)}^{\triangleleft}(s)=1+(q^2+q+1)t+(q^2+q+2)t^{2}+(q+2)t^{3}+t^{4}.\] 
\end{thm}

\begin{proof}
	Let us count the number of ideals of index $q^3$. For $M_{(0,3),(1,0)}$, we get
	\begin{align*}
		i=1,j=2:\,&(0,0,-m_{1,4},0)=(Y_{1,2}^{(1)},0,0,0),\\
		i=1,j=3:\,&(0,0,-m_{1,4},0)=(Y_{1,3}^{(1)},0,0,0),\\
		i=1,j=4:\,&(0,0,m_{1,2}+m_{1,3},0)=(Y_{1,4}^{(1)},0,0,0),
	\end{align*}which gives the system of equations
	\begin{align*}
		m_{1,4}&=0, &m_{1,2}+m_{1,3}&=0.
	\end{align*}
	Therefore we have
	\[c_{(0,3),(1,0)}^{\ideal}=q.\]
	
	For $M_{(1,2),(1,0)}$, we get
	\begin{align*}
		i=2,j=2:\,&(0,0,-m_{2,4},0)=(0,Y_{2,2}^{(2)},0,0),\\
		i=2,j=3:\,&(0,0,-m_{2,4},0)=(0,Y_{2,3}^{(2)},0,0),\\
		i=2,j=4:\,&(0,0,1+m_{2,3},0)=(0,Y_{2,4}^{(2)},0,0),
	\end{align*}which gives the system of equations
	\begin{align*}
		m_{2,4}&=0, &1+m_{2,3}&=0.
	\end{align*}
	Therefore we have $c_{(1,2),(1,0)}^{\ideal}=1$.
	
	For $M_{(2,1),(1,0)}$, we get
	\begin{align*}
		i=3,j=2:\,&(0,0,-m_{3,4},m_{3,4}^2)=(0,0,Y_{3,2}^{(3)},0),\\
		i=3,j=3:\,&(0,0,-m_{3,4},m_{3,4}^2)=(0,0,Y_{3,3}^{(3)},0),\\
		i=3,j=4:\,&(0,0,1,-m_{3,4})=(0,0,Y_{3,4}^{(3)},0),
	\end{align*}
	giving $c_{(2,1),(1,0)}^{\ideal}=1$. 
	
	For $M_{(3),(1)}$ we get 
	\begin{align*}
		i=4,j=2:\,(0,0,-1,0)=(0,0,0,Y_{4,2}^{(4)}),
	\end{align*}
	giving $c_{(3)(1)}^{\ideal}=0$. Together, we have 
	\[a_{q^3}^{\ideal}(M^4)=q+2.\]
	
	Let us count the number of ideals of index $q^2$. For $M_{(0,2),(2,0)}$, we get
	\begin{align*}
		i=1,j=2:\,&(0,0,-m_{1,4},0)=(Y_{1,2}^{(1)},Y_{1,2}^{(2)},0,0),\\
		i=1,j=3:\,&(0,0,-m_{1,4},0)=(Y_{1,3}^{(1)},Y_{1,3}^{(2)},0,0),\\
		i=1,j=4:\,&(0,0,m_{1,3},0)=(Y_{1,4}^{(1)},Y_{1,4}^{(2)},0,0),\\
		i=2,j=2:\,&(0,0,-m_{2,4},0)=(Y_{2,2}^{(1)},Y_{2,2}^{(2)},0,0),\\
		i=2,j=3:\,&(0,0,-m_{2,4},0)=(Y_{2,3}^{(1)},Y_{2,3}^{(2)},0,0),\\
		i=2,j=4:\,&(0,0,1+m_{2,3},0)=(Y_{2,4}^{(1)},Y_{2,4}^{(2)},0,0),
	\end{align*}which gives the system of equations
	\begin{align*}
		m_{1,3}&=0, &m_{1,4}&=0,\\
		1+m_{2,3}&=0,&m_{2,4}&=0.
	\end{align*}
	Therefore we have $c_{(0,2),(2,0)}^{\ideal}=1$.
	
	For $M_{(0,1,1),(1,1,0)}$, we get
	\begin{align*}
		i=1,j=2:\,&(0,0,-m_{1,4},m_{1,4}m_{3,4})=(Y_{1,2}^{(1)},0,Y_{1,2}^{(3)},0),\\
		i=1,j=3:\,&(0,0,-m_{1,4},m_{1,4}m_{3,4})=(Y_{1,3}^{(1)},0,Y_{1,3}^{(3)},0),\\
		i=1,j=4:\,&(0,0,m_{1,2},-m_{1,2}m_{3,4})=(Y_{1,4}^{(1)},0,Y_{1,4}^{(3)},0),\\
		i=3,j=2:\,&(0,0,-m_{3,4},m_{3,4}^2)=(Y_{3,2}^{(1)},0,Y_{3,2}^{(3)},0),\\
		i=3,j=3:\,&(0,0,-m_{3,4},m_{3,4}^2)=(Y_{3,3}^{(1)},0,Y_{3,3}^{(3)},0),\\
		i=3,j=4:\,&(0,0,1,-m_{3,4})=(Y_{3,4}^{(1)},0,Y_{3,4}^{(3)},0),
	\end{align*}which forces $m_{3,4}=0$ and allows $m_{1,2},m_{1,4}\in\Fq$.
	Therefore we get $c_{(0,1,1),(1,1,0)}^{\ideal}=q^2$.
	
	For $M_{(0,2),(1,1)}$, we get
	\begin{align*}
		i=4,j=2:\,(0,0,-1,0)=(Y_{4,2}^{(1)},0,0,Y_{4,2}^{(4)}),
	\end{align*}
	giving $c_{(0,2),(1,1)}^{\ideal}=0$. 
	
	For $M_{(1,1),(2,0)}$, we get
	\begin{align*}
		i=2,j=2:\,&(0,0,-m_{2,4},m_{2,4}m_{3,4})=(0,Y_{2,2}^{(2)},Y_{2,2}^{(3)},0),\\
		i=2,j=3:\,&(0,0,-m_{2,4},m_{2,4}m_{3,4})=(0,Y_{2,3}^{(2)},Y_{2,3}^{(3)},0),\\
		i=2,j=4:\,&(0,0,1,-m_{3,4})=(0,Y_{2,4}^{(2)},Y_{2,4}^{(3)},0),\\ 
		i=3,j=1:\,&(-m_{3,4},0,0,0)=(0,Y_{3,1}^{(2)},Y_{3,1}^{(3)},0),\\
		i=3,j=2:\,&(0,0,-m_{3,4},m_{3,4}^2)=(0,Y_{3,2}^{(2)},Y_{3,2}^{(3)},0),\\
		i=3,j=3:\,&(0,0,-m_{3,4},m_{3,4}^2)=(0,Y_{3,3}^{(2)},Y_{3,3}^{(3)},0),\\
		i=3,j=4:\,&(0,0,1,-m_{3,4})=(0,Y_{3,4}^{(2)},Y_{3,4}^{(3)},0),
	\end{align*}which forces $m_{3,4}=0$ and allows $m_{2,4}\in\Fq$.
	Therefore we get $c_{(1,1),(2,0)}^{\ideal}=q$.
	
	For $M_{(1,1),(1,1)}$, we get
	\begin{align*}
		i=4,j=2:\,(0,0,-1,0)=(0,Y_{4,2}^{(1)},0,Y_{4,2}^{(4)}),
	\end{align*}
	giving $c_{(1,1),(1,1)}^{\ideal}=0$. 
	
	For $M_{(2),(2)}$, we get no equations to satisfy, giving $c_{(2),(2)}^{\ideal}=1$. 
	Together, we have 
	\[a_{q^2}^{\ideal}(M^4)=q^2+q+2.\]
	
	Finally, let us count the number of ideals of index $q$. For $M_{(0,1),(3,0)}$, we get
	\begin{align*}
		i=1,j=2:\,&(0,0,-m_{1,4},m_{1,4}m_{3,4})=(Y_{1,2}^{(1)},Y_{1,2}^{(2)},Y_{1,2}^{(3)},0),\\
		i=1,j=3:\,&(0,0,-m_{1,4},m_{1,4}m_{3,4})=(Y_{1,3}^{(1)},Y_{1,3}^{(2)},Y_{1,3}^{(3)},0),\\
		i=2,j=2:\,&(0,0,-m_{2,4},m_{2,4}m_{3,4})=(Y_{2,2}^{(1)},Y_{2,2}^{(2)},Y_{2,2}^{(3)},0),\\
		i=2,j=3:\,&(0,0,-m_{2,4},m_{2,4}m_{3,4})=(Y_{2,3}^{(1)},Y_{2,3}^{(2)},Y_{2,3}^{(3)},0),\\
		i=2,j=4:\,&(0,0,1,-m_{3,4})=(Y_{2,4}^{(1)},Y_{2,4}^{(2)},Y_{2,4}^{(3)},0),\\
		i=3,j=2:\,&(0,0,-m_{3,4},m_{3,4}^2)=(Y_{3,2}^{(1)},Y_{3,2}^{(2)},Y_{3,2}^{(3)},0),\\
		i=3,j=3:\,&(0,0,-m_{3,4},m_{3,4}^2)=(Y_{3,3}^{(1)},Y_{3,3}^{(2)},Y_{3,3}^{(3)},0),\\
		i=3,j=4:\,&(0,0,1,-m_{3,4})=(Y_{3,4}^{(1)},Y_{3,4}^{(2)},Y_{3,4}^{(3)},0),
	\end{align*}which forces $m_{3,4}=0$ and allows $m_{1,4},m_{2,4}\in\Fq$.
	Therefore we get $c_{(0,1),(3,0)}^{\ideal}=q^2$.
	
	For $M_{(0,1),(2,1)}$, we get
	\begin{align*}
		i=4,j=3:\, (0,0,-1,0)=(Y_{4,3}^{(1)},Y_{4,3}^{(2)},0,Y_{4,3}^{(4)}),
	\end{align*}
	giving $c_{(0,1),(2,1)}^{\ideal}=0$. 
	
	For $M_{(0,1),(1,2)}$ and $M_{(1),(3)}$ we get no equations to satisfy, giving $c_{(0,1),(1,2)}^{\ideal}=q$ and $c_{(1),(3)}^{\ideal}=1$.
	Together, we have 
	\[a_{q}^{\ideal}(M^4)=q^2+q+1.\]
	
	Summing all up, we get 	\[\zeta_{M^4(\Fq)}^{\triangleleft}(s)=1+(q^2+q+1)t+(q^2+q+2)t^{2}+(q+2)t^{3}+t^{4}.\qedhere\] 
\end{proof}

\begin{thm}
	Let 
	\[M^5:=\langle e_1,e_2,e_3,e_4:[e_4,e_2]=e_3\rangle_{\Fq}.\]
	We have
	\[\zeta_{M^5(\Fq)}^{\triangleleft}(s)=1+(q^2+q+1)t+(q^2+q+1)t^{2}+(q+1)t^{3}+t^{4}.\] 
\end{thm}

\begin{proof}
	Let us count the number of ideals of index $q^3$. For $M_{(0,3),(1,0)}$, we get
	\begin{align*}
		i=1,j=2:\,&(0,0,-m_{1,4},0)=(Y_{1,2}^{(1)},0,0,0),\\
		i=1,j=4:\,&(0,0,m_{1,2},0)=(Y_{1,4}^{(1)},0,0,0),
	\end{align*}which forces $m_{1,2}=m_{1,4}=0$ and allows $m_{1,3}\in\Fq$.
	Therefore we have
	\[c_{(0,3),(1,0)}^{\ideal}=q.\]
	
	For $M_{(1,2),(1,0)}$, we get
	\begin{align*}
		i=2,j=4:\,&(0,0,1,0)=(0,Y_{2,4}^{(2)},0,0),
	\end{align*} giving $c_{(1,2),(1,0)}^{\ideal}=0$.
	
	For $M_{(2,1),(1,0)}$, we get
	\begin{align*}
		i=3,j=2:\,&(0,0,-m_{3,4},m_{3,4}^2)=(0,0,Y_{3,2}^{(3)},0),
	\end{align*}
	giving $c_{(2,1),(1,0)}^{\ideal}=1$. 
	
	For $M_{(3),(1)}$ we get 
	\begin{align*}
		i=4,j=2:\,(0,0,-1,0)=(0,0,0,Y_{4,2}^{(4)}),
	\end{align*}
	giving $c_{(3)(1)}^{\ideal}=0$. Together, we have 
	\[a_{q^3}^{\ideal}(M^4)=q+1.\]
	
	Let us count the number of ideals of index $q^2$. For $M_{(0,2),(2,0)}$, we get
	\begin{align*}
		i=2,j=4:\,&(0,0,1,0)=(Y_{2,4}^{(1)},Y_{2,4}^{(2)},0,0),
	\end{align*}giving $c_{(0,2),(2,0)}^{\ideal}=0$.
	
	For $M_{(0,1,1),(1,1,0)}$, we get
	\begin{align*}
		i=1,j=2:\,&(0,0,-m_{1,4},m_{1,4}m_{3,4})=(Y_{1,2}^{(1)},0,Y_{1,2}^{(3)},0),\\
		i=1,j=4:\,&(0,0,m_{1,2},-m_{1,2}m_{3,4})=(Y_{1,4}^{(1)},0,Y_{1,4}^{(3)},0),\\
		i=3,j=2:\,&(0,0,-m_{3,4},m_{3,4}^2)=(Y_{3,2}^{(1)},0,Y_{3,2}^{(3)},0),
	\end{align*}which forces $m_{3,4}=0$ and allows $m_{1,2},m_{1,4}\in\Fq$.
	Therefore we get $c_{(0,1,1),(1,1,0)}^{\ideal}=q^2$.
	
	For $M_{(0,2),(1,1)}$, we get
	\begin{align*}
		i=4,j=2:\,(0,0,-1,0)=(Y_{4,2}^{(1)},0,0,Y_{4,2}^{(4)}),
	\end{align*}
	giving $c_{(0,2),(1,1)}^{\ideal}=0$. 
	
	For $M_{(1,1),(2,0)}$, we get
	\begin{align*}
		i=2,j=2:\,&(0,0,-m_{2,4},m_{2,4}m_{3,4})=(0,Y_{2,2}^{(2)},Y_{2,2}^{(3)},0),\\
		i=2,j=4:\,&(0,0,1,-m_{3,4})=(0,Y_{2,4}^{(2)},Y_{2,4}^{(3)},0),\\
		i=3,j=2:\,&(0,0,-m_{3,4},m_{3,4}^2)=(0,Y_{3,2}^{(2)},Y_{3,2}^{(3)},0),
	\end{align*}which forces $m_{3,4}=0$ and allows $m_{2,4}\in\Fq$.
	Therefore we get $c_{(1,1),(2,0)}^{\ideal}=q$.
	
	For $M_{(1,1),(1,1)}$, we get
	\begin{align*}
		i=4,j=2:\,(0,0,-1,0)=(0,Y_{4,2}^{(1)},0,Y_{4,2}^{(4)}),
	\end{align*}
	giving $c_{(1,1),(1,1)}^{\ideal}=0$. 
	
	For $M_{(2),(2)}$, we get no equations to satisfy, giving $c_{(2),(2)}^{\ideal}=1$. Together, we have 
	\[a_{q^2}^{\ideal}(M^4)=q^2+q+1.\]
	
	Finally, let us count the number of ideals of index $q$. For $M_{(0,1),(3,0)}$, we get
	\begin{align*}
		i=1,j=2:\,&(0,0,-m_{1,4},m_{1,4}m_{3,4})=(Y_{1,2}^{(1)},Y_{1,2}^{(2)},Y_{1,2}^{(3)},0),\\
		i=2,j=2:\,&(0,0,-m_{2,4},m_{2,4}m_{3,4})=(Y_{2,2}^{(1)},Y_{2,2}^{(2)},Y_{2,2}^{(3)},0),\\
		i=2,j=4:\,&(0,0,1,-m_{3,4})=(Y_{2,4}^{(1)},Y_{2,4}^{(2)},Y_{2,4}^{(3)},0),\\
		i=3,j=2:\,&(0,0,-m_{3,4},m_{3,4}^2)=(Y_{3,2}^{(1)},Y_{3,2}^{(2)},Y_{3,2}^{(3)},0),
	\end{align*}which forces $m_{3,4}=0$ and allows $m_{1,4},m_{2,4}\in\Fq$.
	Therefore we get $c_{(0,1),(3,0)}^{\ideal}=q^2$.
	
	For $M_{(0,1),(2,1)}$, we get
	\begin{align*}
		i=2,j=4:\, (0,0,1,0)=(Y_{2,4}^{(1)},Y_{2,4}^{(2)},0,Y_{2,4}^{(4)}),
	\end{align*}
	giving $c_{(0,1),(2,1)}^{\ideal}=0$. 
	
	For $M_{(0,1),(1,2)}$ and $M_{(1),(3)}$ we get no equations to satisfy, giving $c_{(0,1),(1,2)}^{\ideal}=q$ and $c_{(1),(3)}^{\ideal}=1$.
	Together, we have 
	\[a_{q}^{\ideal}(M^4)=q^2+q+1.\]
	
	Summing all up, we get 	\[\zeta_{M^5(\Fq)}^{\triangleleft}(s)=1+(q^2+q+1)t+(q^2+q+1)t^{2}+(q+1)t^{3}+t^{4}.\qedhere\] 
\end{proof}

\begin{thm}
	For $a,b\in\Fq$, let 
	\[M_{a,b}^6:=\langle e_1,e_2,e_3,e_4:[e_4,e_1]=e_2,[e_4,e_2]=e_3,[e_4,e_3]=-ae_1+be_2+e_3\rangle_{\Fq}.\]
	We have
	\begin{equation*}
		\zeta_{M_{a,b}^6(\Fq)}^{\triangleleft}(s)=\begin{cases}
			1+t+|V_{6,a,b}^{(2)}(\Fq)|t^2+|V_{6,a,b}^{(1)}(\Fq)|t^3+t^4&a\neq0,\\
			1+(q+1)t+(|V_{6,0,b}^{(2)}(\Fq)|+1)t^2+(|V_{3,b}(\Fq)|+1)t^3+t^4&a=0.
		\end{cases} 
	\end{equation*}
	where
	\begin{align*}
		|V_{6,a,b}^{(1)}(\Fq)|&=\left|\{x\in\Fq: -a^2x^3+ax^2+bx+1=0\}\right|,\\		
		|V_{6,a,b}^{(2)}(\Fq)|&=\left|\{x\in\Fq: ax^3-bx^2+x+1=0\}\right|.
	\end{align*}
	
\end{thm}

\begin{proof}
	Let us count the number of ideals of index $q^3$. For $M_{(0,3),(1,0)}$, we get
	\begin{align*}
		i=1,j=1:\,&(0,-m_{1,4},0,0)=(Y_{1,1}^{(1)},0,0,0),\\
		i=1,j=2:\,&(0,0,-m_{1,4},0)=(Y_{1,2}^{(1)},0,0,0),\\
		i=1,j=3:\,&(-am_{1,4},(am_{1,2}-b)m_{1,4},(am_{1,3}-1)m_{1,4},am_{1,4}^2)=(Y_{1,3}^{(1)},0,0,0),\\
		i=1,j=4:\,&(am_{1,3},1+bm_{1,3}-am_{1,2}m_{1,3},m_{1,2}+m_{1,3}-am_{1,3}^2,-am_{1,3}m_{1,4})=(Y_{1,4}^{(1)},0,0,0),
	\end{align*}
	which gives the system of equations
	\begin{align*}
		m_{1,4}&=0,&1+bm_{1,3}-am_{1,2}m_{1,3}&=0,& m_{1,2}+m_{1,3}-am_{1,3}^2&=0.
	\end{align*}
	
	If $a=b=0$, then we get $1=0$, making the equation unsolvable. 
	
	If $a=0$ and $b\neq0$, we get $m_{1,2}=-m_{1,3}=1/b$.
	
	If $a,b\neq0$, then $m_{1,2}=-m_{1,3}+am_{1,3}^2$ and $m_{1,3}$ must be a solution of $1+bx+ax^2-a^2x^3$.
	
	Therefore, we have
	\[c_{(0,3),(1,0)}^{\ideal}=\begin{cases}
		|V_{6,a,b}^{(1)}(\Fq)|&a,b\neq0\\
		1&a=0,b\neq0,\\
		0&a=b=0.
	\end{cases}\]

	For $M_{(1,2),(1,0)}$, we get
	\begin{align*}
		i=2,j=1:\,&(0,-m_{2,4},m_{2,3}m_{2,4},m_{2,4}^2)=(0,Y_{2,1}^{(2)},0,0),\\
		i=2,j=2:\,&(0,0,-m_{2,4},0)=(0,Y_{2,2}^{(2)},0,0),\\
		i=2,j=3:\,&(-am_{2,4},-bm_{2,4},(bm_{2,3}-1)m_{2,4},bm_{2,4}^2)=(0,Y_{2,3}^{(2)},0,0),\\
		i=2,j=4:\,&(am_{2,3},bm_{2.3},1+m_{2,3}-bm_{2.3}^2,-bm_{2,3}m_{2,4})=(0,Y_{2,4}^{(2)},0,0),
	\end{align*}
	which gives the system of equations
	\begin{align*}
		m_{2,4}&=0, &am_{2,3}&=0,&1+m_{2,3}-bm_{2,3}^2&=0.
	\end{align*}
	If $a\neq0$, then $m_{2,3}$, forcing $1=0$.
	
	If $a=0$, then $m_{2,3}=0$ must be a solution of $1+x-bx^2$.
	
	Therefore, we have
	\[c_{(1,2),(1,0)}^{\ideal}=\begin{cases}
		0&a\neq0\\
		|V_{3,b}(\Fq)|&a=0.
	\end{cases}\]

	For $M_{(2,1),(1,0)}$, we get
	\begin{align*}
		i=3,j=1:\,&(0,-m_{3,4},0,0)=(0,0,Y_{3,1}^{(3)},0),\\
		i=3,j=2:\,&(0,0,-m_{3,4},m_{3,4}^2)=(0,0,Y_{3,2}^{(3)},0),\\
		i=3,j=3:\,&(-am_{3,4},-bm_{3,4},-m_{3,4},m_{3,4}^2)=(0,0,Y_{3,3}^{(3)},0),\\
		i=3,j=4:\,&(a,b,1,-m_{3,4})=(0,0,Y_{3,4}^{(3)},0),
	\end{align*}
	which gives the system of equations
	\begin{align*}
		a&=0,&b&=0,&m_{3,4}&=0.
	\end{align*}
	Therefore we have
	\[c_{(2,1),(1,0)}^{\ideal}=\begin{cases}
		1&a=b=0,\\
		0&otherwise.
	\end{cases}\]

	For $M_{(3),(1)}$ we get 
	\begin{align*}
		i=4,j=1:\,(0,-1,0,0)=(0,0,0,Y_{4,1}^{(4)}),
	\end{align*}
	giving $c_{(3)(1)}^{\ideal}=0$. Together, we have
	\[a_{q^3}^{\ideal}(M_{a,b}^6)=\begin{cases}
		|V_{6,a,b}^{(1)}(\Fq)|&a\neq0,\\
		|V_{3,b}(\Fq)|+1&a=0.
	\end{cases}\]

	Let us count the number of ideals of index $q^2$. For $M_{(0,2),(2,0)}$, we get
	\begin{align*}
		i=1,j=1:\,&(0,-m_{1,4},m_{1,4}m_{2,3},m_{1,4}m_{2,4})=(Y_{1,1}^{(1)},Y_{1,1}^{(2)},0,0),\\
		i=1,j=2:\,&(0,0,-m_{1,4},0)=(Y_{1,2}^{(1)},Y_{1,2}^{(2)},0,0),\\
		i=1,j=3:\,&(-am_{1,4},-bm_{1,4},m_{1,4}(am_{1,3}+bm_{2,3}-1),m_{1,4}(am_{1,4}+bm_{2,4}))=(Y_{1,3}^{(1)},Y_{1,3}^{(2)},0,0),\\
		i=1,j=4:\,&(am_{1,3},1+bm_{1,3},m_{1,3}-am_{1,3}^2-(1+bm_{1,3})m_{2,3},-am_{1,3}m_{1,4}-(1+bm_{1,3})m_{2,4})=(Y_{1,4}^{(1)},Y_{1,4}^{(2)},0,0),\\
		i=2,j=1:\,&(0,-m_{2,4},m_{2,3}m_{2,4},m_{2,4}^2)=(Y_{2,1}^{(1)},Y_{2,1}^{(2)},0,0),\\
		i=2,j=2:\,&(0,0,-m_{2,4},0)=(Y_{2,2}^{(1)},Y_{2,2}^{(2)},0,0),\\
		i=2,j=3:\,&(-am_{2,4},-bm_{2,4},(am_{1,3}+bm_{2,3}-1)m_{2,4},m_{2,4}(am_{1,4}+bm_{2,4}))=(Y_{2,3}^{(1)},Y_{2,3}^{(2)},0,0),\\
		i=2,j=4:\,&(am_{2,3},bm_{2,3},1+(1-am_{1,3})m_{2,3}-bm_{2,3}^2,-m_{2,3}(am_{1,4}+bm_{2,4}))=(Y_{2,4}^{(1)},Y_{2,4}^{(2)},0,0),
	\end{align*}
	which gives the system of equations
	\begin{align*}
		m_{1,4}&=0, &m_{2,4}&=0,\\
		m_{1,3}-am_{1,3}^2-(1+bm_{1,3})m_{2,3}&=0,&1+(1-am_{1,3})m_{2,3}-bm_{2,3}^2&=0.
	\end{align*}
	One can see that $(m_{1,3},m_{2,3})\in\Fq^2$ must be a solution satisfying $x-y-ax^2-bxy=1+y-axy-by^2=0$. Using Groebner basis, one can show that this is equivalent to  $m_{1,3}=-m_{2,3}^2$ and $1+m_{2,3}-bm_{2,3}^2+am_{2,3}^3=0$, giving 
	\[c_{(0,2),(2,0)}^{\ideal}=|V_{6,a,b}^{(2)}(\Fq)|\]
	
	For $M_{(0,1,1),(1,1,0)}$, we get
	\begin{align*}
		i=1,j=4:\,(0,1,m_{1,2},-m_{1,2}m_{3,4})=(Y_{1,4}^{(1)},0,Y_{1,4}^{(3)},0),
	\end{align*}
	giving $c_{(0,1,1),(1,1,0)}^{\ideal}=0$.
	
	For $M_{(0,2),(1,1)}$, we get
	\begin{align*}
		i=4,j=1:\,(0,-1,0,0)=(Y_{4,1}^{(1)},0,0,Y_{4,1}^{(4)}),
	\end{align*}
	giving $c_{(0,2),(1,1)}^{\ideal}=0$. 
	
	For $M_{(1,1),(2,0)}$, we get
	\begin{align*}
		i=2,j=1:\,&(0,-m_{2,4},0,m_{2,4}^2)=(0,Y_{2,1}^{(2)},Y_{2,1}^{(3)},0),\\
		i=2,j=2:\,&(0,0,-m_{2.4},m_{2,4}m_{3,4})=(0,Y_{2,2}^{(2)},Y_{2,2}^{(3)},0),\\
		i=2,j=3:\,&(-am_{2,4},-bm_{2,4},-m_{2,4},m_{2,4}(bm_{2,4}+m_{3,4}))=(0,Y_{2,3}^{(2)},Y_{2,3}^{(3)},0),\\
		i=2,j=4:\,&(0,0,1,-m_{3,4})=(0,Y_{2,4}^{(2)},Y_{2,4}^{(3)},0),\\ 
		i=3,j=1:\,&(0,-m_{3,4},0,m_{2,4}m_{3,4})=(0,Y_{3,1}^{(2)},Y_{3,1}^{(3)},0),\\
		i=3,j=2:\,&(0,0,-m_{3,4},m_{3,4}^2)=(0,Y_{3,2}^{(2)},Y_{3,2}^{(3)},0),\\
		i=3,j=3:\,&(-am_{3,4},-bm_{3,4},-m_{3,4},m_{3,4}(bm_{2,4}+m_{3,4}))=(0,Y_{3,3}^{(2)},Y_{3,3}^{(3)},0),\\
		i=3,j=4:\,&(a,b,1,-bm_{2,4}-m_{3,4})=(0,Y_{3,4}^{(2)},Y_{3,4}^{(3)},0),
	\end{align*}
	which gives the system of equations
	\begin{align*}
		m_{2,4}&=0, &m_{3,4}&=0,&a&=0.      
	\end{align*}
	Therefore we have
	\[c_{(1,1),(2,0)}^{\ideal}=\begin{cases}
		1&a=0,\\
		0&a\neq0.
	\end{cases}\]

	For $M_{(1,1),(1,1)}$, we get
	\begin{align*}
		i=4,j=2:\,(0,0,-1,0)=(0,Y_{4,2}^{(1)},0,Y_{4,2}^{(4)}),
	\end{align*}
	giving $c_{(1,1),(1,1)}^{\ideal}=0$. 
	
	For $M_{(2),(2)}$, we get
	\begin{align*}
		i=4,j=1:\,(0,-1,0,0)=(0,0,Y_{4,1}^{(3)},Y_{4,1}^{(4)}),
	\end{align*}
	giving $c_{(2),(2)}^{\ideal}=0$. Together, we have 
	\[a_{q^2}^{\ideal}(M_a^3)=\begin{cases}
		|V_{6,0,b}^{(2)}(\Fq)|+1&a=0,\\
		|V_{6,a,b}^{(2)}(\Fq)|&a\neq0.
	\end{cases}\]

	Finally, let us count the number of ideals of index $q$. For $M_{(0,1),(3,0)}$, we get
	\begin{align*}
		i=1,j=1:\,&(0,-m_{1,4},0,m_{1,4}m_{2,4})=(Y_{1,1}^{(1)},Y_{1,1}^{(2)},Y_{1,1}^{(3)},0),\\
		i=1,j=2:\,&(0,0,-m_{1,4},m_{1,4}m_{3,4})=(Y_{1,2}^{(1)},Y_{1,2}^{(2)},Y_{1,3}^{(3)},0),\\
		i=1,j=3:\,&(-am_{1,4},-bm_{1,4},-m_{1,4},m_{1,4}(am_{1,4}+bm_{2,4}+m_{3,4}))=(Y_{1,3}^{(1)},Y_{1,3}^{(2)},Y_{1,3}^{(3)},0),\\
		i=1,j=4:\,&(0,1,0,-m_{2,4})=(Y_{1,4}^{(1)},Y_{1,4}^{(2)},Y_{1,4}^{(3)},0),\\
		i=2,j=1:\,&(0,-m_{2,4},0,m_{2,4}^2)=(Y_{2,1}^{(1)},Y_{2,1}^{(2)},Y_{2,1}^{(3)},0),\\
		i=2,j=2:\,&(0,0,-m_{2,4},m_{2,4}m_{3,4})=(Y_{2,2}^{(1)},Y_{2,2}^{(2)},Y_{2,2}^{(3)},0),\\
		i=2,j=3:\,&(-am_{2,4},-bm_{2,4},-m_{2,4},m_{2,4}(am_{1,4}+bm_{2,4}+m_{3,4}))=(Y_{2,3}^{(1)},Y_{2,3}^{(2)},Y_{2,3}^{(3)},0),\\
		i=2,j=4:\,&(0,0,1,-m_{3,4})=(Y_{2,4}^{(1)},Y_{2,4}^{(2)},Y_{2,4}^{(3)},0),\\
		i=3,j=1:\,&(0,-m_{3,4},0,m_{2,4}m_{3,4})=(Y_{3,1}^{(1)},Y_{3,1}^{(2)},Y_{3,1}^{(3)},0),\\
		i=3,j=2:\,&(0,0,-m_{3,4},m_{3,4}^2)=(Y_{3,2}^{(1)},Y_{3,2}^{(2)},Y_{3,2}^{(3)},0),\\
		i=3,j=3:\,&(-am_{3,4},-bm_{3,4},-m_{3,4},m_{3,4}(am_{1,4}+bm_{2,4}+m_{3,4}))=(Y_{3,3}^{(1)},Y_{3,3}^{(2)},Y_{3,3}^{(3)},0),\\
		i=3,j=4:\,&(a,b,1,-am_{1,4}-bm_{2,4}-m_{3,4})=(Y_{3,4}^{(1)},Y_{3,4}^{(2)},Y_{3,4}^{(3)},0),
	\end{align*}
	which gives the system of equations
	\begin{align*}
		am_{1,4}&=0, &m_{2,4}&=0,&m_{3,4}&=0.
	\end{align*}
	If $a\neq0$, then $m_{1,4}$ must be zero. If $a=0$, then we have $q$ possible values of $m_{1,4}$. Therefore we have
	\[c_{(0,1),(3,0)}^{\ideal}=\begin{cases}
		1&a\neq0,\\
		q&a=0.
	\end{cases}\]
	
	For $M_{(0,1),(2,1)}$, we get
	\begin{align*}
		i=4,j=2:\, (0,0,-1,0)=(Y_{4,2}^{(1)},Y_{4,2}^{(2)},0,Y_{4,2}^{(4)}),
	\end{align*}
	giving $c_{(0,1),(2,1)}^{\ideal}=0$. 
	
	For $M_{(0,1),(1,2)}$, we get
	\begin{align*}
		i=1, j=4:\,&(0,1,m_{1,2},0)=(Y_{1,4}^{(1)},0,Y_{1,4}^{(3)},Y_{1,4}^{(4)}),
	\end{align*}
	giving $c_{(0,1),(1,2)}^{\ideal}=0$.
	
	for $M_{(1),(3)}$,  we get
	\begin{align*}
		i=3,j=4:\,&(a,b,1,0)=(0,Y_{3,4}^{(2)},Y_{3,4}^{(3)},Y_{3,4}^{(4)}),\\
		i=4, j=3:\,&(-a,-b,-1,0)=(0,Y_{4,3}^{(2)},Y_{4,3}^{(3)},Y_{4,3}^{(4)}),
	\end{align*}
	giving  \[c_{(1),(3)}^{\ideal}=\begin{cases}
		0&a\neq0,\\
		1&a=0.
	\end{cases}\] Together, we have 
	\[a_{q}^{\ideal}(M_{a,b}^6)=\begin{cases}
		1&a\neq0,\\
		q+1&a=0.
	\end{cases}\]
	
	Summing all up, we get 	
	\begin{equation*}
		\zeta_{M_{a,b}^6(\Fq)}^{\triangleleft}(s)=\begin{cases}
			1+t+|V_{6,a,b}^{(2)}(\Fq)|t^2+|V_{6,a,b}^{(1)}(\Fq)|t^3+t^4&a\neq0,\\
			1+(q+1)t+(|V_{6,0,b}^{(2)}(\Fq)|+1)t^2+(|V_{3,b}(\Fq)|+1)t^3+t^4&a=0.\qedhere
		\end{cases} 
	\end{equation*}
\end{proof}

\begin{thm}\label{thm:sol.Ma}
	For $a,b\in\Fq$, let 
	\[M_{a,b}^7:=\langle e_1,e_2,e_3,e_4:[e_4,e_1]=e_2,[e_4,e_2]=e_3,[e_4,e_3]=-ae_1+be_2\rangle_{\Fq}.\]
	We have
	\begin{equation*}
		\zeta_{M_{a,b}^7(\Fq)}^{\triangleleft}(s)=\begin{cases}
			1+t+|V_{7,a,b}^{(2)}(\Fq)|t^2+|V_{7,a,b}^{(1)}(\Fq)|t^3+t^4&a\neq0,\\
			1+(q+1)t+(|V_{7,0,b}^{(2)}(\Fq)|+1)t^2+(|V_{4,b}(\Fq)|+1)t^3+t^4&a=0.
		\end{cases} 
	\end{equation*}
	where
	\begin{align*}
		|V_{7,a,b}^{(1)}(\Fq)|&=\left|\{x\in\Fq:-a^2x^3+bx+1=0 \}\right|,\\
		|V_{7,a,b}^{(2)}(\Fq)|&=\left|\{x\in\Fq:ax^3-bx^2+1=0 \}\right|.
	\end{align*}
	
\end{thm}

\begin{proof}
	Let us count the number of ideals of index $q^3$. For $M_{(0,3),(1,0)}$, we get
	\begin{align*}
		i=1,j=1:\,&(0,-m_{1,4},0,0)=(Y_{1,1}^{(1)},0,0,0),\\
		i=1,j=2:\,&(0,0,-m_{1,4},0)=(Y_{1,2}^{(1)},0,0,0),\\
		i=1,j=3:\,&(-am_{1,4},(am_{1,2}-b)m_{1,4},am_{1,3}m_{1,4},am_{1,4}^2)=(Y_{1,3}^{(1)},0,0,0),\\
		i=1,j=4:\,&(am_{1,3},1+bm_{1,3}-am_{1,2}m_{1,3},m_{1,2}-am_{1,3}^2,-am_{1,3}m_{1,4})=(Y_{1,4}^{(1)},0,0,0),
	\end{align*}
	which gives the system of equations
	\begin{align*}
		m_{1,4}&=0,&1+bm_{1,3}-am_{1,2}m_{1,3}&=0,& m_{1,2}-am_{1,3}^2&=0.
	\end{align*}
	
	If $a=b=0$, then we get $1=0$, making the equation unsolvable. 
	
	If $a=0$ and $b\neq0$, we get $m_{1,2}=0$ and $m_{1,3}=-1/b$.
	
	If $a,b\neq0$, then $m_{1,2}=am_{1,3}^2$ and $m_{1,3}$ must be a solution of $1+bx-a^2x^3$.
	
	Therefore, we have
	\[c_{(0,3),(1,0)}^{\ideal}=\begin{cases}
		|V_{7,a,b}^{(1)}(\Fq)|&a,b\neq0\\
		1&a=0,b\neq0,\\
		0&a=b=0.
	\end{cases}\]

	For $M_{(1,2),(1,0)}$, we get
	\begin{align*}
		i=2,j=1:\,&(0,-m_{2,4},m_{2,3}m_{2,4},m_{2,4}^2)=(0,Y_{2,1}^{(2)},0,0),\\
		i=2,j=2:\,&(0,0,-m_{2,4},0)=(0,Y_{2,2}^{(2)},0,0),\\
		i=2,j=3:\,&(-am_{2,4},-bm_{2,4},bm_{2,3}m_{2,4},bm_{2,4}^2)=(0,Y_{2,3}^{(2)},0,0),\\
		i=2,j=4:\,&(am_{2,3},bm_{2.3},1-bm_{2.3}^2,-bm_{2,3}m_{2,4})=(0,Y_{2,4}^{(2)},0,0),
	\end{align*}
	which gives the system of equations
	\begin{align*}
		m_{2,4}&=0, &am_{2,3}&=0,&1-bm_{2,3}^2&=0.
	\end{align*}
	If $a\neq0$, then $m_{2,3}=0$, forcing $1=0$.
	
	If $a=0$, then $m_{2,3}$ must be a solution of $1-bx^2$.
	
	Therefore, we have
	\[c_{(1,2),(1,0)}^{\ideal}=\begin{cases}
		0&a\neq0\\
		|V_{4,b}(\Fq)|&a=0.
	\end{cases}\]

	For $M_{(2,1),(1,0)}$, we get
	\begin{align*}
		i=3,j=1:\,&(0,-m_{3,4},0,0)=(0,0,Y_{3,1}^{(3)},0),\\
		i=3,j=2:\,&(0,0,-m_{3,4},m_{3,4}^2)=(0,0,Y_{3,2}^{(3)},0),\\
		i=3,j=3:\,&(-am_{3,4},-bm_{3,4},0,0)=(0,0,Y_{3,3}^{(3)},0),\\
		i=3,j=4:\,&(a,b,0,0)=(0,0,Y_{3,4}^{(3)},0),
	\end{align*}
	which gives the system of equations
	\begin{align*}
		a&=0,&b&=0,&m_{3,4}&=0.
	\end{align*}
	Therefore we have
	\[c_{(2,1),(1,0)}^{\ideal}=\begin{cases}
		1&a=b=0,\\
		0&otherwise.
	\end{cases}\]

	For $M_{(3),(1)}$ we get 
	\begin{align*}
		i=4,j=1:\,(0,-1,0,0)=(0,0,0,Y_{4,1}^{(4)}),
	\end{align*}
	giving $c_{(3)(1)}^{\ideal}=0$. Together, we have
	\[a_{q^3}^{\ideal}(M_{a,b}^6)=\begin{cases}
		|V_{7,a,b}^{(1)}(\Fq)|&a\neq0,\\
		|V_{4,b}(\Fq)|+1&a=0.
	\end{cases}\]

	Let us count the number of ideals of index $q^2$. For $M_{(0,2),(2,0)}$, we get
	\begin{align*}
		i=1,j=1:\,&(0,-m_{1,4},m_{1,4}m_{2,3},m_{1,4}m_{2,4})=(Y_{1,1}^{(1)},Y_{1,1}^{(2)},0,0),\\
		i=1,j=2:\,&(0,0,-m_{1,4},0)=(Y_{1,2}^{(1)},Y_{1,2}^{(2)},0,0),\\
		i=1,j=3:\,&(-am_{1,4},-bm_{1,4},m_{1,4}(am_{1,3}+bm_{2,3}),m_{1,4}(am_{1,4}+bm_{2,4}))=(Y_{1,3}^{(1)},Y_{1,3}^{(2)},0,0),\\
		i=1,j=4:\,&(am_{1,3},1+bm_{1,3},-am_{1,3}^2-(1+bm_{1,3})m_{2,3},-am_{1,3}m_{1,4}-(1+bm_{1,3})m_{2,4})=(Y_{1,4}^{(1)},Y_{1,4}^{(2)},0,0),\\
		i=2,j=1:\,&(0,-m_{2,4},m_{2,3}m_{2,4},m_{2,4}^2)=(Y_{2,1}^{(1)},Y_{2,1}^{(2)},0,0),\\
		i=2,j=2:\,&(0,0,-m_{2,4},0)=(Y_{2,2}^{(1)},Y_{2,2}^{(2)},0,0),\\
		i=2,j=3:\,&(-am_{2,4},-bm_{2,4},(am_{1,3}+bm_{2,3})m_{2,4},m_{2,4}(am_{1,4}+bm_{2,4}))=(Y_{2,3}^{(1)},Y_{2,3}^{(2)},0,0),\\
		i=2,j=4:\,&(am_{2,3},bm_{2,3},1-am_{1,3}m_{2,3}-bm_{2,3}^2,-m_{2,3}(am_{1,4}+bm_{2,4}))=(Y_{2,4}^{(1)},Y_{2,4}^{(2)},0,0),
	\end{align*}
	which gives the system of equations
	\begin{align*}
		m_{1,4}&=0, &m_{2,4}&=0,\\
		-am_{1,3}^2-(1+bm_{1,3})m_{2,3}&=0,&1-am_{1,3}m_{2,3}-bm_{2,3}^2&=0.
	\end{align*}
	One can see that $(m_{1,3},m_{2,3})\in\Fq^2$ must be a solution satisfying $y+ax^2+bxy=1-axy-by^2=0$. Using Groebner basis, this is equivalent to $m_{1,3}=-m_{2,3}^2$ and $1-bm_{2,3}^2+am_{2,3}^3=0$, giving 
	\[c_{(0,2),(2,0)}^{\ideal}=|V_{7,a,b}^{(2)}(\Fq)|\]
	
	For $M_{(0,1,1),(1,1,0)}$, we get
	\begin{align*}
		i=1,j=4:\,(0,1,m_{1,2},-m_{1,2}m_{3,4})=(Y_{1,4}^{(1)},0,Y_{1,4}^{(3)},0),
	\end{align*}
	giving $c_{(0,1,1),(1,1,0)}^{\ideal}=0$.
	
	For $M_{(0,2),(1,1)}$, we get
	\begin{align*}
		i=4,j=1:\,(0,-1,0,0)=(Y_{4,1}^{(1)},0,0,Y_{4,1}^{(4)}),
	\end{align*}
	giving $c_{(0,2),(1,1)}^{\ideal}=0$. 
	
	For $M_{(1,1),(2,0)}$, we get
	\begin{align*}
		i=2,j=1:\,&(0,-m_{2,4},0,m_{2,4}^2)=(0,Y_{2,1}^{(2)},Y_{2,1}^{(3)},0),\\
		i=2,j=2:\,&(0,0,-m_{2.4},m_{2,4}m_{3,4})=(0,Y_{2,2}^{(2)},Y_{2,2}^{(3)},0),\\
		i=2,j=3:\,&(-am_{2,4},-bm_{2,4},0,bm_{2,4}^2)=(0,Y_{2,3}^{(2)},Y_{2,3}^{(3)},0),\\
		i=2,j=4:\,&(0,0,1,-m_{3,4})=(0,Y_{2,4}^{(2)},Y_{2,4}^{(3)},0),\\ 
		i=3,j=1:\,&(0,-m_{3,4},0,m_{2,4}m_{3,4})=(0,Y_{3,1}^{(2)},Y_{3,1}^{(3)},0),\\
		i=3,j=2:\,&(0,0,-m_{3,4},m_{3,4}^2)=(0,Y_{3,2}^{(2)},Y_{3,2}^{(3)},0),\\
		i=3,j=3:\,&(-am_{3,4},-bm_{3,4},0,bm_{2,4}m_{3,4})=(0,Y_{3,3}^{(2)},Y_{3,3}^{(3)},0),\\
		i=3,j=4:\,&(a,b,0,-bm_{2,4})=(0,Y_{3,4}^{(2)},Y_{3,4}^{(3)},0),
	\end{align*}
	which gives the system of equations
	\begin{align*}
		m_{2,4}&=0, &m_{3,4}&=0,&a&=0.      
	\end{align*}
	Therefore we have
	\[c_{(1,1),(2,0)}^{\ideal}=\begin{cases}
		0&a\neq0,\\
		1&a=0.     
	\end{cases}\]

	For $M_{(1,1),(1,1)}$, we get
	\begin{align*}
		i=4,j=2:\,(0,0,-1,0)=(0,Y_{4,2}^{(1)},0,Y_{4,2}^{(4)}),
	\end{align*}
	giving $c_{(1,1),(1,1)}^{\ideal}=0$. 
	
	For $M_{(2),(2)}$, we get
	\begin{align*}
		i=4,j=1:\,(0,-1,0,0)=(0,0,Y_{4,1}^{(3)},Y_{4,1}^{(4)}),
	\end{align*}
	giving $c_{(2),(2)}^{\ideal}=0$. Together, we have 
	\[a_{q^2}^{\ideal}(M_{a,b}^7)=\begin{cases}
		|V_{7,0,b}^{(2)}(\Fq)|+1&a=0,\\
		|V_{7,a,b}^{(2)}(\Fq)|&a\neq0.
	\end{cases}\]

	Finally, let us count the number of ideals of index $q$. For $M_{(0,1),(3,0)}$, we get
	\begin{align*}
		i=1,j=1:\,&(0,-m_{1,4},0,m_{1,4}m_{2,4})=(Y_{1,1}^{(1)},Y_{1,1}^{(2)},Y_{1,1}^{(3)},0),\\
		i=1,j=2:\,&(0,0,-m_{1,4},m_{1,4}m_{3,4})=(Y_{1,2}^{(1)},Y_{1,2}^{(2)},Y_{1,3}^{(3)},0),\\
		i=1,j=3:\,&(-am_{1,4},-bm_{1,4},0,m_{1,4}(am_{1,4}+bm_{2,4}))=(Y_{1,3}^{(1)},Y_{1,3}^{(2)},Y_{1,3}^{(3)},0),\\
		i=1,j=4:\,&(0,1,0,-m_{2,4})=(Y_{1,4}^{(1)},Y_{1,4}^{(2)},Y_{1,4}^{(3)},0),\\
		i=2,j=1:\,&(0,-m_{2,4},0,m_{2,4}^2)=(Y_{2,1}^{(1)},Y_{2,1}^{(2)},Y_{2,1}^{(3)},0),\\
		i=2,j=2:\,&(0,0,-m_{2,4},m_{2,4}m_{3,4})=(Y_{2,2}^{(1)},Y_{2,2}^{(2)},Y_{2,2}^{(3)},0),\\
		i=2,j=3:\,&(-am_{2,4},-bm_{2,4},0,m_{2,4}(am_{1,4}+bm_{2,4}))=(Y_{2,3}^{(1)},Y_{2,3}^{(2)},Y_{2,3}^{(3)},0),\\
		i=2,j=4:\,&(0,0,1,-m_{3,4})=(Y_{2,4}^{(1)},Y_{2,4}^{(2)},Y_{2,4}^{(3)},0),\\
		i=3,j=1:\,&(0,-m_{3,4},0,m_{2,4}m_{3,4})=(Y_{3,1}^{(1)},Y_{3,1}^{(2)},Y_{3,1}^{(3)},0),\\
		i=3,j=2:\,&(0,0,-m_{3,4},m_{3,4}^2)=(Y_{3,2}^{(1)},Y_{3,2}^{(2)},Y_{3,2}^{(3)},0),\\
		i=3,j=3:\,&(-am_{3,4},-bm_{3,4},0,m_{3,4}(am_{1,4}+bm_{2,4}))=(Y_{3,3}^{(1)},Y_{3,3}^{(2)},Y_{3,3}^{(3)},0),\\
		i=3,j=4:\,&(a,b,0,-am_{1,4}-bm_{2,4})=(Y_{3,4}^{(1)},Y_{3,4}^{(2)},Y_{3,4}^{(3)},0),
	\end{align*}
	which gives the system of equations
	\begin{align*}
		am_{1,4}&=0, &m_{2,4}&=0,&m_{3,4}&=0.
	\end{align*}
	If $a\neq0$, then $m_{1,4}$ must be zero. If $a=0$, then we have $q$ possible values of $m_{1,4}$. Therefore we have
	\[c_{(0,1),(3,0)}^{\ideal}=\begin{cases}
		1&a\neq0,\\
		q&a=0.
	\end{cases}\]
	
	For $M_{(0,1),(2,1)}$, we get
	\begin{align*}
		i=4,j=2:\, (0,0,-1,0)=(Y_{4,2}^{(1)},Y_{4,2}^{(2)},0,Y_{4,2}^{(4)}),
	\end{align*}
	giving $c_{(0,1),(2,1)}^{\ideal}=0$. 
	
	For $M_{(0,1),(1,2)}$, we get
	\begin{align*}
		i=1, j=4:\,&(0,1,m_{1,2},0)=(Y_{1,4}^{(1)},0,Y_{1,4}^{(3)},Y_{1,4}^{(4)}),
	\end{align*}
	giving $c_{(0,1),(1,2)}^{\ideal}=0$.
	
	for $M_{(1),(3)}$,  we get
	\begin{align*}
		i=3,j=4:\,&(a,b,0,0)=(0,Y_{3,4}^{(2)},Y_{3,4}^{(3)},Y_{3,4}^{(4)}),\\
		i=4, j=3:\,&(-a,-b,0,0)=(0,Y_{4,3}^{(2)},Y_{4,3}^{(3)},Y_{4,3}^{(4)}),
	\end{align*}
	giving  \[c_{(1),(3)}^{\ideal}=\begin{cases}
		0&a\neq0,\\
		1&a=0.
	\end{cases}\] Together, we have 
	\[a_{q}^{\ideal}(M_{a,b}^7)=\begin{cases}
		1&a\neq0,\\
		q+1&a=0.
	\end{cases}\]
	
	Summing all up, we get 	
	\begin{equation*}
		\zeta_{M_{a,b}^7(\Fq)}^{\triangleleft}(s)=\begin{cases}
			1+t+|V_{7,a,b}^{(2)}(\Fq)|t^2+|V_{7,a,b}^{(1)}(\Fq)|t^3+t^4&a\neq0,\\
			1+(q+1)t+(|V_{7,0,b}^{(2)}(\Fq)|+1)t^2+(|V_{4,b}(\Fq)|+1)t^3+t^4&a=0.\qedhere
		\end{cases} 
	\end{equation*}
\end{proof}	
\begin{rem}
	For $c\geq2$, let $M_{c}$ denote a nilpotent class-$c$ $\mcO$-Lie algebra of maximal class, of the form
	\[M_{c}=\langle e_{1},e_{2},\ldots,e_{c+1}\mid\forall\; 2\leq i\leq 
	c,[e_{1},e_{i}]=e_{i+1}\rangle_{\mcO}.\]  
	Then $M_{3}(\Fq)\cong M_{0,0}^{7}(\Fq)$, and  we have $\zeta_{M_{0,0}^7(\Fq)}^{\triangleleft}(s)=\zeta_{M_3(\Fq)}^{\triangleleft}(s)$ in \cite[Theorem 3.3]{Lee25}.
\end{rem}

\begin{thm}
	Let 
	\[M^8:=\langle e_1,e_2,e_3,e_4:[e_1,e_2]=e_2,[e_3,e_4]=e_4\rangle_{\Fq}.\]
	We have
	\[\zeta_{M^8(\Fq)}^{\triangleleft}(s)=1+(1+q)t+3t^{2}+2t^{3}+t^{4}.\]
\end{thm}

\begin{proof}
	Let us count the number of ideals of index $q^3$. For $M_{(0,3),(1,0)}$, we get
	\begin{align*}
		i=1, j=2:\,(0,-1,0,0)=(Y_{1,2}^{(1)},0,0,0),
	\end{align*}
	giving $c_{(0,3),(1,0)}^{\ideal}=0$.
	
	For $M_{(1,2),(1,0)}$, we get
	\begin{align*}
		i=2,j=1:\,&(0,1,-m_{2,3},-m_{2,4})=(0,Y_{2,1}^{(2)},0,0),\\
		i=2,j=3:\,&(0,0,0,m_{2,4})=(0,Y_{2,3}^{(2)},0,0),\\
		i=2,j=4:\,&(0,0,0,-m_{2,3})=(0,Y_{2,4}^{(2)},0,0),
	\end{align*}
	giving $c_{(1,2),(1,0)}^{\ideal}=1$.
	
	For $M_{(2,1),(1,0)}$, we get
	\begin{align*}
		i=3,j=4:\,&(0,0,0,-1)=(0,0,Y_{3,4}^{(3)},0),
	\end{align*}
	giving $c_{(2,1),(1,0)}^{\ideal}=0$. 
	
	For $M_{(3),(1)}$ we get no equations to solve,
	giving $c_{(3)(1)}^{\ideal}=1$. Together, we have 
	\[a_{q^3}^{\ideal}(M^8)=2.\]
	
	Let us count the number of ideals of index $q^2$. For $M_{(0,2),(2,0)}$, we get
	\begin{align*}
		i=1,j=2:\,&(0,-1,m_{2,3},m_{2,4})=(Y_{1,2}^{(1)},Y_{1,2}^{(2)},0,0),\\
		i=1,j=3:\,&(0,0,0,m_{1,4})=(Y_{1,3}^{(1)},Y_{1,3}^{(2)},0,0),\\
		i=1,j=4:\,&(0,0,0,-m_{1,3})=(Y_{1,4}^{(1)},Y_{1,4}^{(2)},0,0),\\
		i=2,j=1:\,&(0,1,-m_{2,3},-m_{2,4})=(Y_{2,1}^{(1)},Y_{2,1}^{(2)},0,0),\\
		i=2,j=3:\,&(0,0,0,m_{2,4})=(Y_{2,3}^{(1)},Y_{2,3}^{(2)},0,0),\\
		i=2,j=4:\,&(0,0,0,-m_{2,3})=(Y_{2,4}^{(1)},Y_{2,4}^{(2)},0,0),
	\end{align*}
	forcing $m_{1,3}=m_{1,4}=m_{2,3}=m_{2,4}=0$. Therefore we get $c_{(0,2),(2,0)}=1$.
	
	For $M_{(0,1,1),(1,1,0)}$, we get
	\begin{align*}
		i=1,j=2:\,(0,-1,0,0)=(Y_{1,2}^{(1)},0,Y_{1,2}^{(3)},0)
	\end{align*}
	giving $c_{(0,1,1),(1,1,0)}^{\ideal}=0$.
	
	For $M_{(0,2),(1,1)}$, we get
	\begin{align*}
		i=1,j=2:\,(0,-1,0,0)=(Y_{1,2}^{(1)},0,0,Y_{1,2}^{(4)}),
	\end{align*}
	giving $c_{(0,2),(1,1)}^{\ideal}=0$. 
	
	For $M_{(1,1),(2,0)}$, we get
	\begin{align*}
		i=3,j=4:\,(0,0,0,-1)=(0,Y_{3,4}^{(2)},Y_{3,4}^{(3)},0),
	\end{align*}
	giving $c_{(1,1),(2,0)}^{\ideal}=0$.
	
	For $M_{(1,1),(1,1)}$, we get
	\begin{align*}
		i=2,j=1:\,(0,1,-m_{2,3},0)=(0,Y_{2,1}^{(2)},0,Y_{2,1}^{(4)}),
	\end{align*}
	giving $c_{(1,1),(1,1)}^{\ideal}=1$. 
	
	For $M_{(2),(2)}$, we get no equations to solve, giving $c_{(2),(2)}^{\ideal}=1$. Together, we have 
	\[a_{q^2}^{\ideal}(M^8)=3.\]
	
	Finally, let us count the number of ideals of index $q$. For $M_{(0,1),(3,0)}$, we get
	\begin{align*}
		i=3,j=4:\,(0,0,0,-1)=(Y_{3,4}^{(1)},Y_{3,4}^{(2)},Y_{3,4}^{(3)},0),
	\end{align*}
	giving $c_{(0,1),(3,0)}^{\ideal}=0$.
	
	For $M_{(0,1),(2,1)}$, we get
	\begin{align*}
		i=1,j=2:\, &(0,-1,m_{2,3},0)=(Y_{1,2}^{(1)},Y_{1,2}^{(2)},0,Y_{1,2}^{(4)}),\\
		i=2,j=1:\, &(0,1,-m_{2,3},0)=(Y_{2,1}^{(1)},Y_{2,1}^{(2)},0,Y_{2,1}^{(4)}),
	\end{align*}
	giving $m_{2,3}=0$ and $m_{1,3}\in\Fq$. Therefore we get $c_{(0,1),(2,1)}^{\ideal}=q$. 
	
	For $M_{(0,1),(1,2)}$, we get
	\begin{align*}
		i=1,j=2:\, (0,-1,0,0)=(Y_{1,2}^{(1)},0,Y_{1,2}^{(3)},Y_{1,2}^{(4)}),
	\end{align*}
	giving $c_{(0,1),(1,2)}^{\ideal}=0$.
	
	for $M_{(1),(3)}$,  we get no equations to solve, 
	giving $c_{(1),(3)}^{\ideal}=1$. Together, we have 
	\[a_{q}^{\ideal}(M^8)=1+q.\]
	
	Summing all up, we get 	\[\zeta_{M^8(\Fq)}^{\triangleleft}(s)=1+(1+q)t+3t^{2}+2t^{3}+t^{4}.\qedhere\] 
\end{proof}

\begin{thm}\label{thm:sol.Ma9.ideal}
	Let $a\in\Fq$ such that $x^2-x-a$ has no roots in $\Fq$, and let 
	\[M_a^9:=\langle e_1,e_2,e_3,e_4:[e_4,e_1]=e_1+ae_2,[e_4,e_2]=e_1,[e_3,e_1]=e_1,[e_3,e_2]=e_2\rangle_{\Fq}.\]
	We have
	\begin{equation*}
		\zeta_{M_a^9(\Fq)}^{\triangleleft}(s)=1+(q+1)t+t^2+t^4.
	\end{equation*}
\end{thm}

\begin{proof}
	Let us count the number of ideals of index $q^3$. For $M_{(0,3),(1,0)}$, we get
	\begin{align*}
		i=1,j=1:\,&(-m_{1,3}-m_{1,4},-am_{1,4}+m_{1,2}(m_{1,3}+m_{1,4}),m_{1,3}(m_{1,3}+m_{1,4}),m_{1,4}(m_{1,3}+m_{1,4}))=(Y_{1,1}^{(1)},0,0,0),\\
		i=1,j=2:\,&(-m_{1,4},-m_{1,3}+m_{1,2}m_{1,4},m_{1,3}m_{1,4},m_{1,4}^2)=(Y_{1,2}^{(1)},0,0,0),\\
		i=1,j=3:\,&(1,0,-m_{1,3},-m_{1,4})=(Y_{1,3}^{(1)},0,0,0),\\
		i=1,j=4:\,&(1+m_{1,2},a-m_{1,2}-m_{1,2}^2,-(1+m_{1,2})m_{1,3},-(1+m_{1,2})m_{1,4})=(Y_{1,4}^{(1)},0,0,0),
	\end{align*}
	which gives the system of equations
	\begin{align*}
		a-m_{1,2}-m_{1,2}^2&=0,&m_{1,3}&=0,&m_{1,4}&=0.
	\end{align*}
	However, by the very definition of $M_{a}^{9}$ there is no $m_{1,2}\in\Fq$ satisfying $a-m_{1,2}-m_{1,2}^2=0$. 
	Therefore we get $c_{(0,3),(1,0)}^{\ideal}=0$.

	For $M_{(1,2),(1,0)}$, we get
	\begin{align*}
		i=2,j=4:\,&(1,0,0,0)=(0,Y_{2,4}^{(2)},0,0),
	\end{align*}
	giving $c_{(1,2),(1,0)}^{\ideal}=0$.
	
	For $M_{(2,1),(1,0)}$, we get
	\begin{align*}
		i=3,j=2:\,&(-m_{3,4},-1,0,0)=(0,0,Y_{3,2}^{(3)},0),\\
	\end{align*}
	giving $c_{(2,1),(1,0)}^{\ideal}=0$.
	
	For $M_{(3),(1)}$ we get 
	\begin{align*}
		i=4,j=1:\,(-1,-a,0,0)=(0,0,0,Y_{4,1}^{(4)}),
	\end{align*}
	giving $c_{(3)(1)}^{\ideal}=0$. Together, we have
	\[a_{q^3}^{\ideal}(M_a^3)=0.\]

	Let us count the number of ideals of index $q^2$. For $M_{(0,2),(2,0)}$, we get
	\begin{align*}
		i=1,j=1:\,&(-m_{1,3}-m_{1,4},-am_{1,4},m_{1,3}^2+m_{1,3}m_{1,4}+am_{1,4}m_{2,3},m_{1,4}(m_{1,3}+m_{1,4}+am_{2,4}))=(Y_{1,1}^{(1)},Y_{1,1}^{(2)},0,0),\\
		i=1,j=2:\,&(-m_{1,4},-m_{1,3},m_{1,3}(m_{1,4}+m_{2,3}),m_{1,4}^2+m_{1,3}m_{2,4})=(Y_{1,2}^{(1)},Y_{1,2}^{(2)},0,0),\\
		i=1,j=3:\,&(1,0,-m_{1,3},-m_{1,4})=(Y_{1,3}^{(1)},Y_{1,3}^{(2)},0,0),\\
		i=1,j=4:\,&(1,a,-m_{1,3}-am_{2,3},-m_{1,4}-am_{2,4})=(Y_{1,4}^{(1)},Y_{1,4}^{(2)},0,0),\\
		i=2,j=1:\,&(-m_{2,3}-m_{2,4},-am_{2,4},am_{2,3}m_{2,4}+m_{1,3}(m_{2,3}+m_{2,4}),am_{2,4}^2+m_{1,4}(m_{2,3}+m_{2,4}))=(Y_{2,1}^{(1)},Y_{2,1}^{(2)},0,0),\\
		i=2,j=2:\,&(-m_{2,4},-m_{2,3},m_{2,3}^2+m_{1,3}m_{2,4},m_{2,4}(m_{1,4}+m_{2,3}))=(Y_{2,2}^{(1)},Y_{2,2}^{(2)},0,0),\\
		i=2,j=3:\,&(0,1,-m_{2,3},-m_{2,4})=(Y_{2,3}^{(1)},Y_{2,3}^{(2)},0,0),\\
		i=2,j=4:\,&(1,0,-m_{1,3},-m_{1,4})=(Y_{2,4}^{(1)},Y_{2,4}^{(2)},0,0),
	\end{align*}
	which forces $m_{1,3}=m_{1,4}=m_{2,3}=m_{2,4}=0$. Therefore we have $c_{(0,2),(2,0)}^{\ideal}=1$.
	
	For $M_{(0,1,1),(1,1,0)}$, we get
	\begin{align*}
		i=1,j=1:\,&(-m_{1,4},(m_{1,2}-a)m_{1,4},0,m_{1,4}^2)=(Y_{1,1}^{(1)},0,Y_{1,1}^{(3)},0),\\
		i=1,j=2:\,&(-m_{1,4},m_{1,2}m_{1,4},0,m_{1,4}^2)=(Y_{1,2}^{(1)},0,Y_{1,2}^{(3)},0),\\
		i=1,j=3:\,&(1,0,0,-m_{1,4})=(Y_{1,3}^{(1)},0,Y_{1,3}^{(3)},0),\\
		i=1,j=4:\,&(1+m_{1,2},a-m_{1,2}-m_{1,2}^2,0,-(1+m_{1,2})m_{1,4})=(Y_{1,4}^{(1)},0,Y_{1,4}^{(3)},0),\\
		i=3,j=1:\,&(-1-m_{3,4},-am_{3,4}+m_{1,2}+m_{1,2}m_{3,4},0,m_{1,4}(1+m_{3,4}))=(Y_{3,1}^{(1)},0,Y_{3,1}^{(3)},0),\\
		i=3,j=2:\,&(-m_{3,4},-1+m_{1,2}m_{3,4},0,m_{1,4}m_{3,4})=(Y_{3,2}^{(1)},0,Y_{3,2}^{(3)},0),
	\end{align*}
	which gives the system of equations
	\begin{align*}
		m_{1,4}&=0, &a-m_{1,2}-m_{1,2}^2&=0,\\
		m_{1,2}m_{3,4}-1&=0,&-am_{3,4}+m_{1,2}+m_{1,2}m_{3,4}&=0.
	\end{align*}
	Again, there is no $m_{1,2}\in\Fq$ satisfying $a-m_{1,2}-m_{1,2}^2=0$. Therefore we have
	\[c_{(0,1,1),(1,1,0)}^{\ideal}=0.\]

	For $M_{(0,2),(1,1)}$, we get
	\begin{align*}
		i=4,j=1:\,&(-1,-a+m_{1,2},m_{1,3},0)=(Y_{4,1}^{(1)},0,0,Y_{4,1}^{(4)}),\\
		i=4,j=2:\,&(-1,m_{1,2},m_{1,3},0)=(Y_{4,2}^{(1)},0,0,Y_{4,2}^{(4)}),
	\end{align*}which gives the system of equations
	\begin{align*}
		m_{1,2}&=0, &m_{1,3}&=0,&a&=0.
	\end{align*}
	As $a\neq0$, we get $c_{(0,2),(1,1)}^{\ideal}=0$. 
	
	For $M_{(1,1),(2,0)}$, we get
	\begin{align*}
		i=2,j=4:\,&(1,0,0,0)=(0,Y_{2,4}^{(2)},Y_{2,4}^{(3)},0)
	\end{align*}
	giving $c_{(1,1),(2,0)}^{\ideal}=0$.
	
	For $M_{(1,1),(1,1)}$, we get
	\begin{align*}
		i=2,j=4:\,(1,0,0,0)=(0,Y_{2,4}^{(2)},0,Y_{2,4}^{(4)}),
	\end{align*}
	giving $c_{(1,1),(1,1)}^{\ideal}=0$. 
	
	For $M_{(2),(2)}$, we get
	\begin{align*}
		i=3,j=1:\,(-1,0,0,0)=(0,0,Y_{3,1}^{(3)},Y_{3,1}^{(4)}),
	\end{align*}
	giving $c_{(2),(2)}^{\ideal}=0$. Together, we have 
	\[a_{q^2}^{\ideal}(M_a^9)=1.\]

	Finally, let us count the number of ideals of index $q$. For $M_{(0,1),(3,0)}$, we get
	\begin{align*}
		i=1,j=1:\,&(-m_{1,4},-am_{1,4},0,m_{1,4}(m_{1,4}+am_{2,4}))=(Y_{1,1}^{(1)},Y_{1,1}^{(2)},Y_{1,1}^{(3)},0),\\
		i=1,j=2:\,&(-m_{1,4},0,0,m_{1,4}^2)=(Y_{1,2}^{(1)},Y_{1,2}^{(2)},Y_{1,3}^{(3)},0),\\
		i=1,j=3:\,&(1,0,0,-m_{1,4})=(Y_{1,3}^{(1)},Y_{1,3}^{(2)},Y_{1,3}^{(3)},0),\\
		i=1,j=4:\,&(1,a,0,-m_{1,4}-am_{2,4})=(Y_{1,4}^{(1)},Y_{1,4}^{(2)},Y_{1,4}^{(3)},0),\\
		i=2,j=1:\,&(-m_{2,4},-am_{2,4},0,m_{2,4}(m_{1,4}+am_{2,4}))=(Y_{2,1}^{(1)},Y_{2,1}^{(2)},Y_{2,1}^{(3)},0),\\
		i=2,j=2:\,&(-m_{2,4},0,0,m_{1,4}m_{2,4})=(Y_{2,2}^{(1)},Y_{2,2}^{(2)},Y_{2,2}^{(3)},0),\\
		i=2,j=3:\,&(0,1,0,-m_{2,4})=(Y_{2,3}^{(1)},Y_{2,3}^{(2)},Y_{2,3}^{(3)},0),\\
		i=2,j=4:\,&(1,0,0,-m_{1,4})=(Y_{2,4}^{(1)},Y_{2,4}^{(2)},Y_{2,4}^{(3)},0),\\
		i=3,j=1:\,&(-1-m_{3,4},-am_{3,4},0,am_{2,4}m_{3,4}+m_{1,4}(1+m_{3,4}))=(Y_{3,1}^{(1)},Y_{3,1}^{(2)},Y_{3,1}^{(3)},0),\\
		i=3,j=2:\,&(-m_{3,4},-1,0,m_{2,4}+m_{1,4}m_{3,4})=(Y_{3,2}^{(1)},Y_{3,2}^{(2)},Y_{3,2}^{(3)},0),
	\end{align*}
	which gives the system of equations
	\begin{align*}
		m_{1,4}&=0, &m_{2,4}&=0
	\end{align*}
	and $q$ free choices of $m_{3,4}\in\Fq$. Therefore we have
	$c_{(0,1),(3,0)}^{\ideal}=q$.

	For $M_{(0,1),(2,1)}$, we get
	\begin{align*}
		i=1,j=1:\,&(-m_{1,3},0,m_{1,3}^2,0)=(Y_{1,1}^{(1)},Y_{1,1}^{(2)},0,Y_{1,1}^{(4)})\\
		i=1,j=2:\,&(0,-m_{1,3},m_{1,3}m_{2,3},0)=(Y_{1,2}^{(1)},Y_{1,2}^{(2)},0,Y_{1,2}^{(4)})\\
		i=1,j=3:\,&(1,0,-m_{1,3},0)=(Y_{1,3}^{(1)},Y_{1,3}^{(2)},0,Y_{1,3}^{(4)})\\
		i=1,j=4:\,&(1,a,-m_{1,3}-am_{2,3},0)=(Y_{1,4}^{(1)},Y_{1,4}^{(2)},0,Y_{1,4}^{(4)})\\
		i=2,j=1:\, &(-m_{2,3},0,m_{1,3}m_{2,3},0)=(Y_{2,1}^{(1)},Y_{2,1}^{(2)},0,Y_{2,1}^{(4)})\\
		i=2,j=2:\, &(0,-m_{2,3},m_{2,3}^2,0)=(Y_{2,2}^{(1)},Y_{2,2}^{(2)},0,Y_{2,2}^{(4)})\\
		i=2, j=3:\, &(0,1,-m_{2,3},0)=(Y_{2,3}^{(1)},Y_{2,3}^{(2)},0,Y_{2,3}^{(4)})\\
		i=2, j=4:\, &(1,0,-m_{1,3},0)=(Y_{2,4}^{(1)},Y_{2,4}^{(2)},0,Y_{2,4}^{(4)})\\
		i=4,j=1:\, &(-1, -a, m_{1,3}+am_{2,3},0)=(Y_{4,1}^{(1)},Y_{4,1}^{(2)},0,Y_{4,1}^{(4)})\\
		i=4, j=2:\, &(-1, 0,m_{1,3},0)=(Y_{4,2}^{(1)},Y_{4,2}^{(2)},0,Y_{4,2}^{(4)}),
	\end{align*}
	which gives the system of equations
	\begin{align*}
		m_{1,3}&=0, &m_{2,3}&=0,
	\end{align*}
	giving $c_{(0,1),(2,1)}^{\ideal}=1$. 
	
	For $M_{(0,1),(1,2)}$, we get
	\begin{align*}
		i=3, j=2:\,(0,-1,0,0)=(Y_{3,2}^{(1)},0,Y_{3,2}^{(3)},Y_{3,2}^{(4)}),
	\end{align*}
	giving $c_{(0,1),(1,2)}^{\ideal}=0$.
	
	for $M_{(1),(3)}$,  we get
	\begin{align*}
		i=2,j=4:\, (1,0,0,0)=(0,Y_{2,4}^{(2)},Y_{2,4}^{(3)},Y_{2,4}^{(4)}),
	\end{align*}
	giving $c_{(1),(3)}^{\ideal}=0$. Together, we have 
	\[a_{q}^{\ideal}(M_a^9)=q+1.\]
	
	Summing all up, we get 	
	\begin{equation*}
		\zeta_{M_a^9(\Fq)}^{\triangleleft}(s)=1+(q+1)t+t^2+t^4\qedhere
	\end{equation*}
\end{proof}	

\begin{rem}	In \cite{deG/05}, de Graaf showed that over an arbitrary field $\boldsymbol{F}$, $M_{a}^9$ only defines a new Lie algebra if $T^2-T-a$ has no roots in $\boldsymbol{F}$, and otherwise $M_{a}^9$ is isomorphic to 
	\[M^8=\langle x_{1},x_{2},x_{3},x_{4}\,\mid\,[x_1,x_2]=x_2, [x_3,x_4]=x_4\rangle.\] 
	
	In our setting this is exactly when $a-m_{1,2}-m_{1,2}^2=0$. It is very interesting that this defining condition on the parameter $a$ can also be observed in the zeta function. The classification of Lie algebras over finite fields is an important problem, and studying their zeta functions might provide useful information. 
\end{rem}

\begin{thm}
	Let 
	\[M^{12}:=\langle e_1,e_2,e_3,e_4:[e_4,e_1]=e_1,[e_4,e_2]=2e_2,[e_4,e_3]=e_3, [e_3,e_1]=e_2\rangle_{\Fq}.\]
	We have
	\[\zeta_{M^{12}(\Fq)}^{\triangleleft}(s)=1+t+(1+q+q^2)t^{2}+(1+q+q^2)t^{3}+t^{4}.\]
\end{thm}

\begin{proof}
	Let us count the number of ideals of index $q^3$. For $M_{(0,3),(1,0)}$, we get
	\begin{align*}
		i=1,j=3:\,&(0,1,-m_{1,4},0)=(Y_{1,3}^{(1)},0,0,0),
	\end{align*}
	giving $c_{(0,3),(1,0)}^{\ideal}=0$.
	
	For $M_{(1,2),(1,0)}$, we get
	\begin{align*}
		i=2,j=1:\,&(-m_{2,4},-m_{2,3},m_{2,3}^2,m_{2,3}m_{2,4})=(0,Y_{2,1}^{(2)},0,0),\\
		i=2,j=2:\,&(0,-2m_{2,4},2m_{2,3}m_{2,4},2m_{2,4}^2)=(0,Y_{2,2}^{(2)},0,0),\\
		i=2,j=3:\,&(0,0,-m_{2,4},0)=(0,Y_{2,3}^{(2)},0,0),\\
		i=2,j=4:\,&(0,2,-m_{2,3},-2m_{2,4})=(0,Y_{2,4}^{(2)},0,0),
	\end{align*}
	which gives the system of equations
	\begin{align*}
		m_{2,3}&=0, &m_{2,4}&=0,
	\end{align*}
	giving $c_{(1,2),(1,0)}^{\ideal}=1$.
	
	For $M_{(2,1),(1,0)}$, we get
	\begin{align*}
		i=3,j=1:\,&(-m_{3,4},-1,0,0)=(0,0,Y_{3,1}^{(3)},0),
	\end{align*}
	giving $c_{(2,1),(1,0)}^{\ideal}=0$. 
	
	For $M_{(3),(1)}$ we get 
	\begin{align*}
		i=4,j=1:\,(-1,0,0,0)=(0,0,0,Y_{4,1}^{(4)}),
	\end{align*}
	giving $c_{(3)(1)}^{\ideal}=0$. Together, we have 
	\[a_{q^3}^{\ideal}(M^{12})=1.\]
	
	Let us count the number of ideals of index $q^2$. For $M_{(0,2),(2,0)}$, we get
	\begin{align*}
		i=1,j=1:\,&(-m_{1,4},-m_{1,3},m_{1,3}(m_{1,4}+m_{2,3}),m_{1,4}^2+m_{1,3}m_{2,4})=(Y_{1,1}^{(1)},Y_{1,1}^{(2)},0,0),\\
		i=1,j=2:\,&(0,-2m_{1,4},2m_{1,4}m_{2,3},2m_{1,4}m_{2,4})=(Y_{1,2}^{(1)},Y_{1,2}^{(2)},0,0),\\
		i=1,j=3:\,&(0,1,-m_{1,4}-m_{2,3},-m_{2,4})=(Y_{1,3}^{(1)},Y_{1,3}^{(2)},0,0),\\
		i=1,j=4:\,&(1,0,0,-m_{1,4})=(Y_{1,4}^{(1)},Y_{1,4}^{(2)},0,0),\\
		i=2,j=1:\,&(-m_{2,4},-m_{2,3},m_{2,3}^2+m_{1,3}m_{2,4},m_{2,4}(m_{1,4}+m_{2,3}))=(Y_{2,1}^{(1)},Y_{2,1}^{(2)},0,0),\\
		i=2,j=2:\,&(0,-2m_{2,4},2m_{2,3}m_{2,4},2m_{2,4}^2)=(Y_{2,2}^{(1)},Y_{2,2}^{(2)},0,0),\\
		i=2,j=3:\,&(0,0,-m_{2,4},0)=(Y_{2,3}^{(1)},Y_{2,3}^{(2)},0,0),\\
		i=2,j=4:\,&(0,2,-m_{2,3},-2m_{2,4})=(Y_{2,4}^{(1)},Y_{2,4}^{(2)},0,0),
	\end{align*}which gives the system of equations
	\begin{align*}
		m_{1,4}&=0,&m_{2,3}&=0, &m_{2,4}&=0.
	\end{align*}
	We have $m_{1,4}=m_{2,3}=m_{2,4}=0$ and $q$ free choices for $m_{1,3}\in\Fq$. Therefore we get
	$c_{(0,2),(2,0)}^{\ideal}=q$.
	
	For $M_{(0,1,1),(1,1,0)}$, we get
	\begin{align*}
		i=1,j=3:\,&(0,1,-m_{1,4},m_{1,4}m_{3,4})=(Y_{1,3}^{(1)},0,Y_{1,3}^{(3)},0),\\
	\end{align*}
	giving $c_{(0,1,1),(1,1,0)}^{\ideal}=0$.
	
	For $M_{(0,2),(1,1)}$, we get
	\begin{align*}
		i=1,j=3:\,(0,1,0,0)=(Y_{1,3}^{(1)},0,0,Y_{1,3}^{(4)}),
	\end{align*}
	giving $c_{(0,2),(1,1)}^{\ideal}=0$. 
	
	For $M_{(1,1),(2,0)}$, we get
	\begin{align*}
		i=2,j=1:\,&(-m_{2,4},0,0,0)=(0,Y_{2,1}^{(2)},Y_{2,1}^{(3)},0),\\
		i=2,j=2:\,&(0,-2m_{2,4},0,2m_{2,4}^2)=(0,Y_{2,2}^{(2)},Y_{2,2}^{(3)},0),\\
		i=2,j=3:\,&(0,0,-m_{2,4},m_{2,4}m_{3,4})=(0,Y_{2,3}^{(2)},Y_{2,3}^{(3)},0),\\
		i=2,j=4:\,&(0,2,0,-2m_{2,4})=(0,Y_{2,4}^{(2)},Y_{2,4}^{(3)},0),\\ 
		i=3,j=1:\,&(-m_{3,4},-1,0,m_{2,4})=(0,Y_{3,1}^{(2)},Y_{3,1}^{(3)},0),\\
		i=3,j=2:\,&(0,-2m_{3,4},0,2m_{2,4}m_{3,4})=(0,Y_{3,2}^{(2)},Y_{3,2}^{(3)},0),\\
		i=3,j=3:\,&(0,0,-m_{3,4},m_{3,4}^2)=(0,Y_{3,3}^{(2)},Y_{3,3}^{(3)},0),\\
		i=3,j=4:\,&(0,0,1,-m_{3,4})=(0,Y_{3,4}^{(2)},Y_{3,4}^{(3)},0),
	\end{align*}
	which gives the system of equations
	\begin{align*}
		m_{2,4}&=0, &m_{3,4}&=0,
	\end{align*}
	giving $c_{(1,1),(2,0)}^{\ideal}=1$.
	
	For $M_{(1,1),(1,1)}$, we get
	\begin{align*}
		i=4,j=1:\,(-1,0,0,0)=(0,Y_{4,1}^{(1)},0,Y_{4,1}^{(4)}),
	\end{align*}
	giving $c_{(1,1),(1,1)}^{\ideal}=0$. 
	
	For $M_{(2),(2)}$, we get
	\begin{align*}
		i=3,j=1:\,(0,-1,0,0)=(0,0,Y_{3,1}^{(3)},Y_{3,1}^{(4)}),
	\end{align*}
	giving $c_{(2),(2)}^{\ideal}=0$. Together, we have 
	\[a_{q^2}^{\ideal}(M^{12})=q+1.\]
	
	Finally, let us count the number of ideals of index $q$. For $M_{(0,1),(3,0)}$, we get
	\begin{align*}
		i=1,j=1:\,&(-m_{1,4},0,0,m_{1,4}^2)=(Y_{1,1}^{(1)},Y_{1,1}^{(2)},Y_{1,1}^{(3)},0),\\
		i=1,j=2:\,&(0,-2m_{1,4},0,2m_{1,4}m_{2,4})=(Y_{1,2}^{(1)},Y_{1,2}^{(2)},Y_{1,3}^{(3)},0),\\
		i=1,j=3:\,&(0,1,-m_{1,4},-m_{2,4}+m_{1,4}m_{3,4})=(Y_{1,3}^{(1)},Y_{1,3}^{(2)},Y_{1,3}^{(3)},0),\\
		i=1,j=4:\,&(1,0,0,-m_{1,4})=(Y_{1,4}^{(1)},Y_{1,4}^{(2)},Y_{1,4}^{(3)},0),\\
		i=2,j=1:\,&(-m_{2,4},0,0,m_{1,4}m_{2,4})=(Y_{2,1}^{(1)},Y_{2,1}^{(2)},Y_{2,1}^{(3)},0),\\
		i=2,j=2:\,&(0,-2m_{2,4},0,2m_{2,4}^2)=(Y_{2,2}^{(1)},Y_{2,2}^{(2)},Y_{2,2}^{(3)},0),\\
		i=2,j=3:\,&(0,0,-m_{2,4},m_{2,4}m_{3,4})=(Y_{2,3}^{(1)},Y_{2,3}^{(2)},Y_{2,3}^{(3)},0),\\
		i=2,j=4:\,&(0,2,0,-2m_{2,4})=(Y_{2,4}^{(1)},Y_{2,4}^{(2)},Y_{2,4}^{(3)},0),\\
		i=3,j=1:\,&(-m_{3,4},-1,0,m_{2,4}+m_{1,4}m_{3,4})=(Y_{3,1}^{(1)},Y_{3,1}^{(2)},Y_{3,1}^{(3)},0),\\
		i=3,j=2:\,&(0,-2m_{3,4},0,2m_{2,4}m_{3,4})=(Y_{3,2}^{(1)},Y_{3,2}^{(2)},Y_{3,2}^{(3)},0),\\
		i=3,j=3:\,&(0,0,-m_{3,4},m_{3,4}^2)=(Y_{3,3}^{(1)},Y_{3,3}^{(2)},Y_{3,3}^{(3)},0),\\
		i=3,j=4:\,&(0,0,1,-m_{3,4})=(Y_{3,4}^{(1)},Y_{3,4}^{(2)},Y_{3,4}^{(3)},0),
	\end{align*}
	which gives the system of equations
	\begin{align*}
		m_{1,4}&=0,&m_{2,4}&=0, &m_{3,4}&=0,
	\end{align*}
	giving $c_{(0,1),(3,0)}^{\ideal}=1$.
	
	For $M_{(0,1),(2,1)}$, we get
	\begin{align*}
		i=4,j=3:\, (0,0,-1,0)=(Y_{4,3}^{(1)},Y_{4,3}^{(2)},0,Y_{4,3}^{(4)}),
	\end{align*}
	giving $c_{(0,1),(2,1)}^{\ideal}=0$. 
	
	For $M_{(0,1),(1,2)}$, we get
	\begin{align*}
		i=3,j=1:\, (0,-1,0,0)=(Y_{3,1}^{(1)},0,Y_{3,1}^{(3)},Y_{3,1}^{(4)}),
	\end{align*}
	giving $c_{(0,1),(1,2)}^{\ideal}=0$.
	
	for $M_{(1),(3)}$,  we get
	\begin{align*}
		i=4,j=1:\, (-1,0,0,0)=(0,Y_{4,1}^{(2)},Y_{4,1}^{(3)},Y_{4,1}^{(4)}),
	\end{align*}
	giving $c_{(1),(3)}^{\ideal}=0$. Together, we have 
	\[a_{q}^{\ideal}(M^{12})=1.\]
	
	Summing all up, we get 	\[\zeta_{M^{12}(\Fq)}^{\triangleleft}(s)=1+t+(1+q)t^{2}+t^{3}+t^{4}.\qedhere\] 
\end{proof}

\begin{thm}
	Let 
	\[M_a^{13}:=\langle e_1,e_2,e_3,e_4:[e_4,e_1]=e_1+ae_3,[e_4,e_2]=e_2,[e_4,e_3]=e_1, [e_3,e_1]=e_2\rangle_{\Fq}.\]
	We have
	\[\zeta_{M_{a}^{13}(\Fq)}^{\triangleleft}(s)=\begin{cases}
		1+t+|V_{13,a}(\Fq)|t^{2}+t^{3}+t^{4}&a\neq0\\
		1+(q+1)t+2t^2+t^3+t^4&a=0,
	\end{cases}\] 
	where
	\begin{align*}
		|V_{13,a}(\Fq)|&=\left|\{x\in\Fq:x^2+x-a=0\}\right|
	\end{align*}
\end{thm}

\begin{proof}
	Let us count the number of ideals of index $q^3$. For $M_{(0,3),(1,0)}$, we get
	\begin{align*}
		i=1,j=2:\,&(0,-m_{1,4},0,0)=(Y_{1,2}^{(1)},0,0,0),\\
		i=1,j=3:\,&(-m_{1,4},1+m_{1,2}m_{1,4},m_{1,3}m_{1,4},m_{1,4}^2)=(Y_{1,3}^{(1)},0,0,0).
	\end{align*}which gives the system of equations
	\begin{align*}
		m_{1,4}&=0,&1+m_{1,2}m_{1,4}=1&=0.
	\end{align*}
	Therefore we have $c_{(0,3),(1,0)}^{\ideal}=0$.
	
	For $M_{(1,2),(1,0)}$, we get
	\begin{align*}
		i=2,j=1:\,&(-m_{2,4},-m_{2,3},m_{2,3}^2-am_{2,4},m_{2,3}m_{2,4})=(0,Y_{2,1}^{(2)},0,0),\\
		i=2,j=2:\,&(0,-m_{2,4},m_{2,3}m_{2,4},m_{2,4}^2)=(0,Y_{2,2}^{(2)},0,0),\\
		i=2,j=3:\,&(-m_{2,4},0,0,0)=(0,Y_{2,3}^{(2)},0,0),\\
		i=2,j=4:\,&(m_{2,3},1,-m_{2,3},-m_{2,4})=(0,Y_{2,4}^{(2)},0,0),
	\end{align*}
	which gives the system of equations
	\begin{align*}
		m_{2,3}&=0, &m_{2,4}&=0,
	\end{align*}
	giving $c_{(1,2),(1,0)}^{\ideal}=1$.
	
	For $M_{(2,1),(1,0)}$, we get
	\begin{align*}
		i=3,j=1:\,&(-m_{3,4},-1,-am_{3,4},am_{3,4}^2)=(0,0,Y_{3,1}^{(3)},0),
	\end{align*}
	giving $c_{(2,1),(1,0)}^{\ideal}=0$. 
	
	For $M_{(3),(1)}$ we get 
	\begin{align*}
		i=4,j=1:\,(-1,0,-a,0)=(0,0,0,Y_{4,1}^{(4)}),
	\end{align*}
	giving $c_{(3)(1)}^{\ideal}=0$. Together, we have 
	\[a_{q^3}^{\ideal}(M_a^{13})=1.\]

	Let us count the number of ideals of index $q^2$. For $M_{(0,2),(2,0)}$, we get
	\begin{align*}
		i=1,j=1:\,&(-m_{1,4},-m_{1,3},-am_{1,4}+m_{1,3}(m_{1,4}+m_{2,3}),m_{1,4}^2+m_{1,3}m_{2,4})=(Y_{1,1}^{(1)},Y_{1,1}^{(2)},0,0),\\
		i=1,j=2:\,&(0,-m_{1,4},m_{1,4}m_{2,3},m_{1,4}m_{2,4})=(Y_{1,2}^{(1)},Y_{1,2}^{(2)},0,0),\\
		i=1,j=3:\,&(-m_{1,4},1,m_{1,3}m_{1,4}-m_{2,3},m_{1,4}^2-m_{2,4})=(Y_{1,3}^{(1)},Y_{1,3}^{(2)},0,0),\\
		i=1,j=4:\,&(1+m_{1,3},0,a-m_{1,3}-m_{1,3}^2,-m_{1,4}(1+m_{1,3}))=(Y_{1,4}^{(1)},Y_{1,4}^{(2)},0,0),\\
		i=2,j=1:\,&(-m_{2,4},-m_{2,3},m_{2,3}^2+(m_{1,3}-a)m_{2,4},m_{2,4}(m_{1,4}+m_{2,3}))=(Y_{2,1}^{(1)},Y_{2,1}^{(2)},0,0),\\
		i=2,j=2:\,&(0,-m_{2,4},m_{2,3}m_{2,4},m_{2,4}^2)=(Y_{2,2}^{(1)},Y_{2,2}^{(2)},0,0),\\
		i=2,j=3:\,&(-m_{2,4},0,m_{1,3}m_{2,4},m_{1,4}m_{2,4})=(Y_{2,3}^{(1)},Y_{2,3}^{(2)},0,0),\\
		i=2,j=4:\,&(m_{2,3},1,-m_{2,3}(1+m_{1,3}),-m_{1,4}m_{2,3}-m_{2,4})=(Y_{2,4}^{(1)},Y_{2,4}^{(2)},0,0),
	\end{align*}which gives the system of equations
	\begin{align*}
		m_{1,4}&=0,&m_{2,3}&=0,\\
		m_{2,4}&=0,&a-m_{1,3}-m_{1,3}^2&=0.
	\end{align*}
	Therefore we get
	$c_{(0,2),(2,0)}^{\ideal}=|V_{13,a}(\Fq)|$.
	
	For $M_{(0,1,1),(1,1,0)}$, we get
	\begin{align*}   i=1,j=3:\,&(-m_{1,4},1+m_{1,2}m_{1,4},0,m_{1,4}^2)=(Y_{1,3}^{(1)},0,Y_{1,3}^{(3)},0),
	\end{align*}
	giving $c_{(0,1,1),(1,1,0)}^{\ideal}=0$. 
	
	For $M_{(0,2),(1,1)}$, we get
	\begin{align*}
		i=1,j=3:\,(0,1,0,0)=(Y_{1,3}^{(1)},0,0,Y_{1,3}^{(4)}),
	\end{align*}
	giving $c_{(0,2),(1,1)}^{\ideal}=0$. 
	
	For $M_{(1,1),(2,0)}$, we get
	\begin{align*}
		i=3,j=4:\,&(1,0,0,0)=(0,Y_{3,4}^{(2)},Y_{3,4}^{(3)},0),
	\end{align*}
	giving $c_{(1,1),(2,0)}^{\ideal}=0$.
	
	For $M_{(1,1),(1,1)}$, we get
	\begin{align*}
		i=4,j=1:\,(-1,0,-a,0)=(0,Y_{4,1}^{(1)},0,Y_{4,1}^{(4)}),
	\end{align*}
	giving $c_{(1,1),(1,1)}^{\ideal}=0$. 
	
	For $M_{(2),(2)}$, we get
	\begin{align*}
		i=3,j=1:\,(0,-1,0,0)=(0,0,Y_{3,1}^{(3)},Y_{3,1}^{(4)}),
	\end{align*}
	giving $c_{(2),(2)}^{\ideal}=0$. Together, we have 
	\[a_{q^2}^{\ideal}(M_a^{13})=|V_{13,a}(\Fq)|.\]
	
	Finally, let us count the number of ideals of index $q$. For $M_{(0,1),(3,0)}$, we get
	\begin{align*}
		i=1,j=1:\,&(-m_{1,4},0,-am_{1,4},m_{1,4}(m_{1,4}+am_{3,4}))=(Y_{1,1}^{(1)},Y_{1,1}^{(2)},Y_{1,1}^{(3)},0),\\
		i=1,j=2:\,&(0,-m_{1,4},0,m_{1,4}m_{2,4})=(Y_{1,2}^{(1)},Y_{1,2}^{(2)},Y_{1,3}^{(3)},0),\\
		i=1,j=3:\,&(-m_{1,4},1,0,m_{1,4}^2-m_{2,4})=(Y_{1,3}^{(1)},Y_{1,3}^{(2)},Y_{1,3}^{(3)},0),\\
		i=1,j=4:\,&(1,0,a,-m_{1,4}-am_{3,4})=(Y_{1,4}^{(1)},Y_{1,4}^{(2)},Y_{1,4}^{(3)},0),\\
		i=2,j=1:\,&(-m_{2,4},0,-am_{2,4},m_{2,4}(m_{1,4}+am_{3,4}))=(Y_{2,1}^{(1)},Y_{2,1}^{(2)},Y_{2,1}^{(3)},0),\\
		i=2,j=2:\,&(0,-m_{2,4},0,m_{2,4}^2)=(Y_{2,2}^{(1)},Y_{2,2}^{(2)},Y_{2,2}^{(3)},0),\\
		i=2,j=3:\,&(-m_{2,4},0,0,m_{1,4}m_{2,4})=(Y_{2,3}^{(1)},Y_{2,3}^{(2)},Y_{2,3}^{(3)},0),\\
		i=2,j=4:\,&(0,1,0,-m_{2,4})=(Y_{2,4}^{(1)},Y_{2,4}^{(2)},Y_{2,4}^{(3)},0),\\
		i=3,j=1:\,&(-m_{3,4},-1,-am_{3,4},m_{2,4}+m_{3,4}(m_{1,4}+am_{3,4}))=(Y_{3,1}^{(1)},Y_{3,1}^{(2)},Y_{3,1}^{(3)},0),\\
		i=3,j=2:\,&(0,-m_{3,4},0,m_{2,4}m_{3,4})=(Y_{3,2}^{(1)},Y_{3,2}^{(2)},Y_{3,2}^{(3)},0),\\
		i=3,j=3:\,&(-m_{3,4},0,0,m_{1,4}m_{3,4})=(Y_{3,3}^{(1)},Y_{3,3}^{(2)},Y_{3,3}^{(3)},0),\\
		i=3,j=4:\,&(1,0,0,-m_{1,4})=(Y_{3,4}^{(1)},Y_{3,4}^{(2)},Y_{3,4}^{(3)},0),
	\end{align*}
	which gives the system of equations
	\begin{align*}
		m_{1,4}&=0,&m_{2,4}&=0, &am_{3,4}&=0.
	\end{align*}
	If $a\neq0$, then we must have $m_{3,4}=0$. 
	
	If $a=0$, then we have $q$ possible choices for $m_{3,4}\in\Fq$. 
	
	Therefore we have $c_{(0,1),(3,0)}^{\ideal}=\begin{cases}
		1&a\neq0.\\
		q&a=0.
	\end{cases}$
	
	For $M_{(0,1),(2,1)}$, we get
	\begin{align*}
		i=1,j=1:\, &(0,-m_{1,3},m_{1,3}m_{2,3},0)=(Y_{1,1}^{(1)},Y_{1,1}^{(2)},0,Y_{1,1}^{(4)}),\\
		i=1, j=3:\, &(0,1,-m_{2,3},0)=(Y_{1,3}^{(1)},Y_{1,3}^{(2)},0,Y_{1,3}^{(4)}),\\
		i=1, j=4:\, &(1+m_{1,3},0,a-m_{1,3}-m_{1,3}^2,0)=(Y_{1,4}^{(1)},Y_{1,4}^{(2)},0,Y_{1,4}^{(4)}),\\
		i=2, j=1:\, &(0,-m_{2,3},m_{2,3}^2,0)=(Y_{2,1}^{(1)},Y_{2,1}^{(2)},0,Y_{2,1}^{(4)}),\\
		i=2, j=4:\, &(m_{2,3},1,-m_{2,3}(1+m_{1,3}),0)=(Y_{2,4}^{(1)},Y_{2,4}^{(2)},0,Y_{2,4}^{(4)}),\\
		i=4, j=1:\,&(-1,0,-a+m_{1,3},0)=(Y_{4,1}^{(1)},Y_{4,1}^{(2)},0,Y_{4,1}^{(4)}),\\
		i=4, j=2:\,&(0,-1,m_{2,3},0)=(Y_{4,2}^{(1)},Y_{4,2}^{(2)},0,Y_{4,2}^{(4)}),\\
		i=4, j=3:\, &(-1,0,m_{1,3},0)=(Y_{4,3}^{(1)},Y_{4,3}^{(2)},0,Y_{4,3}^{(4)}),
	\end{align*}which gives the system of equations
	\begin{align*}
		m_{1,3}&=0,&m_{2,3}&=0, &a&=0.
	\end{align*}
	Therefore we have $c_{(0,1),(2,1)}^{\ideal}=\begin{cases}
		0&a\neq0.\\
		1&a=0.
	\end{cases}$
	
	For $M_{(0,1),(1,2)}$, we get
	\begin{align*}
		i=3,j=1:\, (0,-1,0,0)=(Y_{3,1}^{(1)},0,Y_{3,1}^{(3)},Y_{3,1}^{(4)}),
	\end{align*}
	giving $c_{(0,1),(1,2)}^{\ideal}=0$.
	
	for $M_{(1),(3)}$,  we get
	\begin{align*}
		i=3,j=4:\, (1,0,0,0)=(0,Y_{3,4}^{(2)},Y_{3,4}^{(3)},Y_{3,4}^{(4)}),
	\end{align*}
	giving $c_{(1),(3)}^{\ideal}=0$. Together, we have 
	\[a_{q}^{\ideal}(M_a^{13})=\begin{cases}
		1&a\neq 0,\\
		q+1&a=0.
	\end{cases}.\]
	
	Summing all up, we get 	\[\zeta_{M_{a}^{13}(\Fq)}^{\triangleleft}(s)=\begin{cases}
		1+t+|V_{13,a}(\Fq)|t^{2}+t^{3}+t^{4}&a\neq0\\
		1+(q+1)t+2t^2+t^3+t^4&a=0.\qedhere
	\end{cases}\] 
\end{proof}

\begin{thm}
	Let 
	\[M_a^{14}:=\langle e_1,e_2,e_3,e_4:[e_4,e_1]=ae_3,[e_4,e_3]=e_1, [e_3,e_1]=e_2\rangle_{\Fq}.\]
	We have
	\[\zeta_{M_{a}^{14}(\Fq)}^{\triangleleft}(s)=\begin{cases}
		1+t+|V_{14,a}(\Fq)|t^{2}+t^{3}+t^{4}&a\neq0\\
		1+(q+1)t+t^2+t^3+t^4&a=0,
	\end{cases}\]  
	where \[|V_{14,a}(\Fq)|=|\{x\in\Fq:x^2-a=0\}|.\]
\end{thm}

\begin{proof}
	Let us count the number of ideals of index $q^3$. For $M_{(0,3),(1,0)}$, we get
	\begin{align*}
		i=1,j=3:\,&(-m_{1,4},1+m_{1,2}m_{1,4},m_{1,3}m_{1,4},m_{1,4}^2)=(Y_{1,3}^{(1)},0,0,0),
	\end{align*}which gives the system of equations
	\begin{align*}
		m_{1,4}^2&=0,&1+m_{1,2}m_{1,4}&=0.
	\end{align*}
	Since $m_{1,4}^2=0$ implies $1+m_{1,2}m_{1,4}=1=0$, we have $c_{(0,3),(1,0)}^{\ideal}=0$.
	
	For $M_{(1,2),(1,0)}$, we get
	\begin{align*}
		i=2,j=1:\,&(0,-m_{2,3},m_{2,3}^2-am_{2,4},m_{2,3}m_{2,4})=(0,Y_{2,1}^{(2)},0,0),\\
		i=2,j=3:\,&(-m_{2,4},0,0,0)=(0,Y_{2,3}^{(2)},0,0),\\
		i=2,j=4:\,&(m_{2,3},0,0,0)=(0,Y_{2,4}^{(2)},0,0),
	\end{align*}
	which gives the system of equations
	\begin{align*}
		m_{2,3}&=0, &m_{2,4}&=0,
	\end{align*}
	giving $c_{(1,2),(1,0)}^{\ideal}=1$.
	
	For $M_{(2,1),(1,0)}$, we get
	\begin{align*}
		i=3,j=1:\,&(0,-1,-am_{3,4},am_{3,4}^2)=(0,0,Y_{3,1}^{(3)},0),
	\end{align*}
	giving $c_{(2,1),(1,0)}^{\ideal}=0$. 
	
	For $M_{(3),(1)}$ we get 
	\begin{align*}
		i=4,j=3:\,(-1,0,0,0)=(0,0,0,Y_{4,3}^{(4)}),
	\end{align*}
	giving $c_{(3)(1)}^{\ideal}=0$. Together, we have 
	\[a_{q^3}^{\ideal}(M_{a}^{14})=1.\]

	Let us count the number of ideals of index $q^2$. For $M_{(0,2),(2,0)}$, we get
	\begin{align*}
		i=1,j=1:\,&(0,-m_{1,3},-am_{1,4}+m_{1,3}m_{2,3},m_{1,3}m_{2,4})=(Y_{1,1}^{(1)},Y_{1,1}^{(2)},0,0),\\
		i=1,j=3:\,&(-m_{1,4},1,m_{1,3}m_{1,4}-m_{2,3},m_{1,4}^2-m_{2,4})=(Y_{1,3}^{(1)},Y_{1,3}^{(2)},0,0),\\
		i=1,j=4:\,&(m_{1,3},0,a-m_{1,3}^2,-m_{1,3}m_{1,4})=(Y_{1,4}^{(1)},Y_{1,4}^{(2)},0,0),\\
		i=2,j=1:\,&(0,-m_{2,3},m_{2,3}^2-am_{2,4},m_{2,3}m_{2,4})=(Y_{2,1}^{(1)},Y_{2,1}^{(2)},0,0),\\
		i=2,j=3:\,&(-m_{2,4},0,m_{1,3}m_{2,4},m_{1,4}m_{2,4})=(Y_{2,3}^{(1)},Y_{2,3}^{(2)},0,0),\\
		i=2,j=4:\,&(m_{2,3},0,-m_{1,3}m_{2,3},-m_{1,4}m_{2,3})=(Y_{2,4}^{(1)},Y_{2,4}^{(2)},0,0),
	\end{align*}which gives the system of equations
	\begin{align*}
		m_{1,4}&=0,&m_{2,3}&=0,\\
		m_{2,4}&=0,&a-m_{1,3}^2&=0.
	\end{align*}
	Therefore we get
	$c_{(0,2),(2,0)}^{\ideal}=|V_{14,a}(\Fq)|$.
	
	For $M_{(0,1,1),(1,1,0)}$, we get
	\begin{align*}   i=3,j=1:\,&(0,-1,-am_{3,4},-am_{3,4}^2)=(Y_{3,1}^{(1)},0,Y_{3,1}^{(3)},0),
	\end{align*}
	giving $c_{(0,1,1),(1,1,0)}^{\ideal}=0$. 
	
	For $M_{(0,2),(1,1)}$, we get
	\begin{align*}
		i=1,j=3:\,(0,1,0,0)=(Y_{1,3}^{(1)},0,0,Y_{1,3}^{(4)}),
	\end{align*}
	giving $c_{(0,2),(1,1)}^{\ideal}=0$. 
	
	For $M_{(1,1),(2,0)}$, we get
	\begin{align*}
		i=3,j=4:\,&(1,0,0,0)=(0,Y_{3,4}^{(2)},Y_{3,4}^{(3)},0),
	\end{align*}
	giving $c_{(1,1),(2,0)}^{\ideal}=0$.
	
	For $M_{(1,1),(1,1)}$, we get
	\begin{align*}
		i=4,j=3:\,(-1,0,0,0)=(0,Y_{4,3}^{(2)},0,Y_{4,3}^{(4)}),
	\end{align*}
	giving $c_{(1,1),(1,1)}^{\ideal}=0$. 
	
	For $M_{(2),(2)}$, we get
	\begin{align*}
		i=3,j=1:\,(0,-1,0,0)=(0,0,Y_{3,1}^{(3)},Y_{3,1}^{(4)}),
	\end{align*}
	giving $c_{(2),(2)}^{\ideal}=0$. Together, we have 
	\[a_{q^2}^{\ideal}(M_a^{14})=|V_{14,a}(\Fq)|.\]
	
	Finally, let us count the number of ideals of index $q$. For $M_{(0,1),(3,0)}$, we get
	\begin{align*}
		i=1,j=1:\,&(0,0,-am_{1,4},am_{1,4}m_{3,4})=(Y_{1,1}^{(1)},Y_{1,1}^{(2)},Y_{1,1}^{(3)},0),\\
		i=1,j=3:\,&(-m_{1,4},1,0,m_{1,4}^2-m_{2,4})=(Y_{1,3}^{(1)},Y_{1,3}^{(2)},Y_{1,3}^{(3)},0),\\
		i=1,j=4:\,&(0,0,a,-am_{3,4})=(Y_{1,4}^{(1)},Y_{1,4}^{(2)},Y_{1,4}^{(3)},0),\\
		i=2,j=1:\,&(0,0,-am_{2,4},am_{2,4}m_{3,4})=(Y_{2,1}^{(1)},Y_{2,1}^{(2)},Y_{2,1}^{(3)},0),\\
		i=2,j=3:\,&(-m_{2,4},0,0,m_{1,4}m_{2,4})=(Y_{2,3}^{(1)},Y_{2,3}^{(2)},Y_{2,3}^{(3)},0),\\
		i=3,j=1:\,&(0,-1,-am_{3,4},m_{2,4}am_{3,4}^2)=(Y_{3,1}^{(1)},Y_{3,1}^{(2)},Y_{3,1}^{(3)},0),\\
		i=3,j=3:\,&(-m_{3,4},0,0,m_{1,4}m_{3,4})=(Y_{3,3}^{(1)},Y_{3,3}^{(2)},Y_{3,3}^{(3)},0),\\
		i=3,j=4:\,&(1,0,0,-m_{1,4})=(Y_{3,4}^{(1)},Y_{3,4}^{(2)},Y_{3,4}^{(3)},0),
	\end{align*}
	which gives the system of equations
	\begin{align*}
		m_{1,4}&=0,&m_{2,4}&=0, &am_{3,4}&=0.
	\end{align*}
	If $a\neq0$, then we must have $m_{3,4}=0$. 
	
	If $a=0$, then we have $q$ possible choices for $m_{3,4}\in\Fq$. 
	
	Therefore we have $c_{(0,1),(3,0)}^{\ideal}=\begin{cases}
		1&a\neq0.\\
		q&a=0.
	\end{cases}$
	
	For $M_{(0,1),(2,1)}$, we get
	\begin{align*}
		i=1,j=1:\, &(0,-m_{1,3},m_{1,3}m_{2,3},0)=(Y_{1,1}^{(1)},Y_{1,1}^{(2)},0,Y_{1,1}^{(4)}),\\
		i=1, j=3:\, &(0,1,-m_{2,3},0)=(Y_{1,3}^{(1)},Y_{1,3}^{(2)},0,Y_{1,3}^{(4)}),\\
		i=1, j=4:\, &(m_{1,3},0,a-m_{1,3}^2,0)=(Y_{1,4}^{(1)},Y_{1,4}^{(2)},0,Y_{1,4}^{(4)}),\\
		i=2, j=1:\, &(0,-m_{2,3},m_{2,3}^2,0)=(Y_{2,1}^{(1)},Y_{2,1}^{(2)},0,Y_{2,1}^{(4)}),\\
		i=2, j=4:\, &(m_{2,3},0,-m_{1,3}m_{2,3},0)=(Y_{2,4}^{(1)},Y_{2,4}^{(2)},0,Y_{2,4}^{(4)}),\\
		i=4, j=1:\,&(0,0,-a,0)=(Y_{4,1}^{(1)},Y_{4,1}^{(2)},0,Y_{4,1}^{(4)}),\\
		i=4, j=3:\, &(-1,0,m_{1,3},0)=(Y_{4,3}^{(1)},Y_{4,3}^{(2)},0,Y_{4,3}^{(4)}),
	\end{align*}which gives the system of equations
	\begin{align*}
		m_{1,3}&=0,&m_{2,3}&=0, &a&=0.
	\end{align*}
	Therefore we have $c_{(0,1),(2,1)}^{\ideal}=\begin{cases}
		0&a\neq0.\\
		1&a=0.
	\end{cases}$
	
	For $M_{(0,1),(1,2)}$, we get
	\begin{align*}
		i=1,j=3:\, (0,1,0,0)=(Y_{1,3}^{(1)},0,Y_{1,3}^{(3)},Y_{1,3}^{(4)}),
	\end{align*}
	giving $c_{(0,1),(1,2)}^{\ideal}=0$.
	
	for $M_{(1),(3)}$,  we get
	\begin{align*}
		i=3,j=4:\, (1,0,0,0)=(0,Y_{3,4}^{(2)},Y_{3,4}^{(3)},Y_{3,4}^{(4)}),
	\end{align*}
	giving $c_{(1),(3)}^{\ideal}=0$. Together, we have 
	\[a_{q}^{\ideal}(M_a^{14})=\begin{cases}
		1&a\neq 0,\\
		q+1&a=0.
	\end{cases}.\]
	
	Summing all up, we get 	\[\zeta_{M_{a}^{14}(\Fq)}^{\triangleleft}(s)=\begin{cases}
		1+t+|V_{14,a}(\Fq)|t^{2}+t^{3}+t^{4}&a\neq0\\
		1+(q+1)t+t^2+t^3+t^4&a=0.\qedhere
	\end{cases}\] 
\end{proof}

\subsection{Solvable $\Fq$-Lie algebras of dimension $n=4$, counting subalgebras}For subalgebras of $\Fq$-Lie algebras of dimension 4, since $a_1^{\leq}(L)=a_{q^{4}}^{\leq}(L)=1$ and $a_{q^{3}}^{\leq}(L)=\binom{4}{3}_q=1+q+q^2+q^3$, one only needs to compute $a_{q}^{\leq}(L)$ and $a_{q^{2}}^{\leq}(L)$.
\begin{thm}
	Let $M^{1}$ be the 4-dimensional abelian $\Fq$-Lie algebra. Then we have
	\[\zeta_{M^{1}(\Fq)}^{\leq}(s)=1+\binom{4}{1}_qt+\binom{4}{2}_qt^2+\binom{4}{3}_qt^3+t^4.\]
\end{thm}
\begin{proof}
	Note that $\zeta_{M^{1}(\Fq)}^{\leq}(s)=\zeta_{\Fq^4}(s)$.
\end{proof}
\begin{thm}
	Let 
	\[M^2:=\langle e_1,e_2,e_3,e_4:[e_4,e_1]=e_1,[e_4,e_2]=e_2,[e_4,e_3]=e_3\rangle_{\Fq}.\]
	We have
	\[\zeta_{M^2(\Fq)}^{\leq}(s)=\zeta_{M^{1}(\Fq)}^{\leq}(s)=1+\binom{4}{1}_qt+\binom{4}{2}_qt^2+\binom{4}{3}_qt^3+t^4.\]
\end{thm}

\begin{proof}
	By Theorem \ref{thm:sub.iso}.
\end{proof}

\begin{thm}
	For $a\in\Fq$, let 
	\[M_a^3:=\langle e_1,e_2,e_3,e_4:[e_4,e_1]=e_1,[e_4,e_2]=e_3,[e_4,e_3]=-ae_2+(a+1)e_3\rangle_{\Fq}.\]
	We have
	\begin{equation*}
		\zeta_{M_a^3(\Fq)}^{\leq}(s)=\begin{cases}
			1+(q^2+3q+1)t+(q^3+3q^2+q+1)t^2+\binom{4}{3}_q
			t^3+t^4&a=0,\\
			1+(q^2+q+1)t+(q^3+2q^2+q+1)t^2+\binom{4}{3}_qt^3+t^4&a=1,\\
			1+(q^2+2q+1)t+(q^3+3q^2+q+1)t^2+\binom{4}{3}_qt^3+t^4&a\neq0,1.\\
		\end{cases} 
	\end{equation*}
\end{thm}

\begin{proof}
	
	Let us count the number of ideals of index $q^2$. For $M_{(0,2),(2,0)}$, we get
	\begin{align*}
		g_{1,2,3}^{\leq}(M)&=-(1+m_{2,3})(m_{1,4}+am_{1,4}m_{2,3}-am_{1,3}m_{2,4})=0,\\
		g_{1,2,4}^{\leq}(M)&=-m_{2,4}(m_{1,4}+am_{1,4}m_{2,3}-am_{1,3}m_{2,4})=0.
	\end{align*}
	Solving the given system of equations give
	\[c_{(0,2),(2,0)}^{\leq}=\begin{cases}
		q^3+2q^2-2q&a\neq0,1,\\
		q^3+q^2-q&a=0,1.
	\end{cases}\]
	
	For $M_{(0,1,1),(1,1,0)}$, we get
	\begin{align*}
		g_{1,3,2}^{\leq}(M)&=am_{1,4}-m_{1,2}m_{3,4}=0,\\
		g_{1,3,4}^{\leq}(M)&=m_{3,4}(am_{1,4}-m_{1,2}m_{3,4})=0.
	\end{align*}
	If $a\neq 0$, then $m_{1,4}=m_{1,2}m_{3,4}/a$, giving $q^2$ solutions.
	
	If $a=0$, then the system of equations reduces to $m_{1,2}m_{3,4}=0$ with free $m_{1,4}\in\Fq$, giving $q(2q-1)$ solutions. 
	
	Therefore we have
	\[c_{(0,1,1),(1,1,0)}^{\leq}=\begin{cases}
		q^2&a\neq0,\\
		q(2q-1)&a=0.
	\end{cases}\]
	
	For $M_{(0,2),(1,1)}$, we get
	\begin{align*}
		g_{1,4,2}^{\leq}(M)&=-m_{1,2}-am_{1,3}=0,\\
		g_{1,4,3}^{\leq}(M)&=m_{1,2}+am_{1,3}=0,
	\end{align*}
	giving $c_{(0,2),(1,1)}^{\leq}=q$. 
	
	For $M_{(1,1),(2,0)}$, we get
	\begin{align*}
		g_{2,3,1}^{\leq}(M)&=0=0,\\
		g_{2,3,4}^{\leq}(M)&=-(m_{2,4}-m_{3,4})(am_{2,4}-m_{3,4})=0,
	\end{align*}
	giving $m_{3,4}=m_{2,4}$ or $m_{3,4}=am_{2,4}$. Therefore we have
	\[c_{(1,1),(2,0)}^{\leq}=
	\begin{cases}
		2q-1&a\neq1,\\
		q&a=1.
	\end{cases}\]
	
	For $M_{(1,1),(1,1)}$, we get
	\begin{align*}
		g_{2,4,1}^{\leq}(M)&=0=0,\\
		g_{2,4,3}^{\leq}(M)&=(1+m_{2,3})(1+am_{2,3})=0,
	\end{align*}
	giving $m_{2,3}=-1$ or $m_{2,3}=-1/a$ (if $a\neq0$). Therefore we have
	\[c_{(1,1),(1,1)}^{\leq}=
	\begin{cases}
		2&a\neq0,1,\\
		1&a=0,1.
	\end{cases}\]
	
	For $M_{(2),(2)}$, we get
	\begin{align*}
		g_{3,4,1}^{\leq}(M)&=0=0,\\
		g_{3,4,2}^{\leq}(M)&=-a=0,
	\end{align*}
	giving 
	\[c_{(2),(2)}^{\leq}=\begin{cases}
		0&a\neq0,\\
		1&a=0.
	\end{cases}\]
	
	Together, we have 
	\[a_{q^2}^{\leq}(M_a^3)=\begin{cases}
		q^3+3q^2+q+1&a\neq1,\\
		q^3+2q^2+q+1&a=1.
	\end{cases}\]

	Finally, let us count the number of ideals of index $q$. For $M_{(0,1),(3,0)}$, we get
	\begin{align*}
		g_{1,2,4}^{\leq}(M)&=m_{1,4}(m_{3,4}-m_{2,4})=0,\\
		g_{1,3,4}^{\leq}(M)&=am_{1,4}(m_{3,4}-m_{2,4})=0,\\
		g_{2,3,4}^{\leq}(M)&=(am_{2,4}-m_{3,4})(m_{3,4}-m_{2,4})=0.
	\end{align*}
	If $m_{3,4}=m_{2,4}$, then the system of equation always holds. Hence we have $q^2$ solutions. 
	
	If $m_{3,4}\neq m_{2,4}$, then one requires $m_{1,4}=0$ and $m_{3,4}=am_{2,4}$. For $a\neq1$, we get $q-1$ solutions. For $a=1$, the system of equation is unsolvable since we suppose $m_{3,4}\neq m_{2,4}$. 
	
	Therefore we have
	\[c_{(0,1),(3,0)}^{\leq}=\begin{cases}
		q^2+q-1&a\neq1,\\
		q^2&a=1.
	\end{cases}\]
	
	For $M_{(0,1),(2,1)}$, we get
	\begin{align*}
		g_{1,4,3}^{\leq}(M)&=am_{1,3}(1+m_{2,3})=0,\\
		g_{2,4,3}^{\leq}(M)&=(1+am_{2,3})(1+m_{2,3})=0,
	\end{align*}
	giving  \[c_{(0,1),(2,1)}^{\leq}=\begin{cases}
		q+1&a\neq0,1,\\
		q&a=0,1.
	\end{cases}\]
	
	For $M_{(0,1),(1,2)}$, we get
	\begin{align*}
		g_{1,4,2}^{\leq}(M)&=-m_{1,2}=0,\\
		g_{3,4,3}^{\leq}(M)&=-a=0.
	\end{align*}
	Therefore we have
	\[c_{(0,1),(1,2)}^{\leq}=\begin{cases}
		0&a\neq0,\\
		1&a=0.
	\end{cases}\]
	
	for $M_{(1),(3)}$,  we get no equations to solve,
	giving $c_{(1),(3)}=1$. Together, we have 
	\[a_{q}^{\leq}(M_a^3)=\begin{cases}
		q^2+2q+1&a\neq1,\\
		q^2+q+1&a=1.
	\end{cases}\]
	
	Summing all up, we get 	
	\begin{equation*}
		\zeta_{M_a^3(\Fq)}^{\leq}(s)=\begin{cases}
			1+(q^2+2q+1)t+(q^3+3q^2+q+1)t^2+\binom{4}{3}_q
			t^3+t^4&a\neq1,\\
			1+(q^2+q+1)t+(q^3+2q^2+q+1)t^2+\binom{4}{3}_qt^3+t^4&a=1.
		\end{cases} 
	\end{equation*}
\end{proof}

\begin{thm}
	Let 
	\[M^4:=\langle e_1,e_2,e_3,e_4:[e_4,e_2]=e_3,[e_4,e_3]=e_3\rangle_{\Fq}.\]
	We have
	\begin{equation*}
		\zeta_{M^4(\Fq)}^{\leq}(s)=1+(q^2+2q+1)t+(q^3+3q^2+q+1)t^2+\binom{4}{3}_qt^3+t^4.
	\end{equation*}
\end{thm}

\begin{proof}
	
	Let us count the number of ideals of index $q^2$. For $M_{(0,2),(2,0)}$, we get
	\begin{align*}
		g_{1,2,3}^{\leq}(M)&=-m_{1,4}-m_{1,4}m_{2,3}+m_{1,3}m_{2,4}=0,\\
		g_{1,2,4}^{\leq}(M)&=0=0.
	\end{align*}
	Solving the given system of equations gives
	\[c_{(0,2),(2,0)}^{\leq}=q^3+q^2-q\]
	
	For $M_{(0,1,1),(1,1,0)}$, we get
	\begin{align*}
		g_{1,3,2}^{\leq}(M)&=0=0,\\
		g_{1,3,4}^{\leq}(M)&=m_{3,4}(m_{1,4}-m_{1,2}m_{3,4})=0.
	\end{align*}
	Solving the given system of equations gives
	\[c_{(0,1,1),(1,1,0)}^{\leq}=2q^2-q.\]
	
	For $M_{(0,2),(1,1)}$, we get
	\begin{align*}
		g_{1,4,2}^{\leq}(M)&=0=0,\\
		g_{1,4,3}^{\leq}(M)&=m_{1,2}+m_{1,3}=0,
	\end{align*}
	giving $c_{(0,2),(1,1)}^{\leq}=q$. 
	
	For $M_{(1,1),(2,0)}$, we get
	\begin{align*}
		g_{2,3,1}^{\leq}(M)&=0=0,\\
		g_{2,3,4}^{\leq}(M)&=-m_{3,4}(-m_{2,4}+m_{3,4})=0,
	\end{align*}
	giving 
	\[c_{(1,1),(2,0)}^{\leq}=2q-1.\]
	
	For $M_{(1,1),(1,1)}$, we get
	\begin{align*}
		g_{2,4,1}^{\leq}(M)&=0=0,\\
		g_{2,4,3}^{\leq}(M)&=1+m_{2,3}=0,
	\end{align*}
	giving 
	\[c_{(1,1),(1,1)}^{\leq}=1.\]
	
	For $M_{(2),(2)}$, we get no equations to solve, 
	giving 
	\[c_{(2),(2)}^{\leq}=1.\]
	
	Together, we have 
	\[a_{q^2}^{\leq}(M^4)=q^3+3q^2+q+1.\]

	Finally, let us count the number of ideals of index $q$. For $M_{(0,1),(3,0)}$, we get
	\begin{align*}
		g_{1,2,4}^{\leq}(M)&=m_{1,4}m_{3,4}=0,\\
		g_{1,3,4}^{\leq}(M)&=m_{1,4}m_{3,4}=0,\\
		g_{2,3,4}^{\leq}(M)&=-m_{3,4}(m_{3,4}-m_{2,4})=0.
	\end{align*}
	Solving the system of equations gives
	\[c_{(0,1),(3,0)}^{\leq}=q^2+q-1.\]
	
	For $M_{(0,1),(2,1)}$, we get
	\begin{align*}
		g_{1,4,3}^{\leq}(M)&=m_{1,3}=0,\\
		g_{2,4,3}^{\leq}(M)&=1+m_{2,3}=0,
	\end{align*}
	giving  \[c_{(0,1),(2,1)}^{\leq}=1.\]
	
	For $M_{(0,1),(1,2)}$ and $M_{(1),(3)}$,  we get no equations to solve,
	giving $c_{(0,1),(1,2)}^{\leq}=q$ and $c_{(1),(3)}^{\leq}=1$. Together, we have 
	\[a_{q}^{\leq}(M^4)=q^2+2q+1\]
	
	Summing all up, we get 	
	\begin{equation*}
		\zeta_{M^4(\Fq)}^{\leq}(s)=1+(q^2+2q+1)t+(q^3+3q^2+q+1)t^2+\binom{4}{3}_qt^3+t^4.\qedhere
	\end{equation*}
\end{proof}

\begin{thm}
	Let 
	\[M^5:=\langle e_1,e_2,e_3,e_4:[e_4,e_2]=e_3\rangle_{\Fq}.\]
	We have
	\begin{equation*}
		\zeta_{M^5(\Fq)}^{\leq}(s)=1+(q^2+q+1)t+(q^3+2q^2+q+1)t^2+\binom{4}{3}_qt^3+t^4.
	\end{equation*}
\end{thm}

\begin{proof}
	
	Let us count the number of ideals of index $q^2$. For $M_{(0,2),(2,0)}$, we get
	\begin{align*}
		g_{1,2,3}^{\leq}(M)&=-m_{1,4}=0,\\
		g_{1,2,4}^{\leq}(M)&=0=0,
	\end{align*}
	giving
	\[c_{(0,2),(2,0)}^{\leq}=q^3\]
	
	For $M_{(0,1,1),(1,1,0)}$, we get
	\begin{align*}
		g_{1,3,2}^{\leq}(M)&=0=0,\\
		g_{1,3,4}^{\leq}(M)&=-m_{1,2}m_{3,4}^2=0,
	\end{align*}
	giving
	\[c_{(0,1,1),(1,1,0)}^{\leq}=2q^2-q.\]
	
	For $M_{(0,2),(1,1)}$, we get
	\begin{align*}
		g_{1,4,2}^{\leq}(M)&=0=0,\\
		g_{1,4,3}^{\leq}(M)&=m_{1,2}=0,
	\end{align*}
	giving $c_{(0,2),(1,1)}^{\leq}=q$. 
	
	For $M_{(1,1),(2,0)}$, we get
	\begin{align*}
		g_{2,3,1}^{\leq}(M)&=0=0,\\
		g_{2,3,4}^{\leq}(M)&=-m_{3,4}^2=0,
	\end{align*}
	giving 
	\[c_{(1,1),(2,0)}^{\leq}=q.\]
	
	For $M_{(1,1),(1,1)}$, we get
	\begin{align*}
		g_{2,4,1}^{\leq}(M)&=1=0,
	\end{align*}
	giving 
	\[c_{(1,1),(1,1)}^{\leq}=0.\]
	
	For $M_{(2),(2)}$, we get no equations to solve, 
	giving 
	\[c_{(2),(2)}^{\leq}=1.\]
	
	Together, we have 
	\[a_{q^2}^{\leq}(M^5)=q^3+2q^2+q+1.\]

	Finally, let us count the number of ideals of index $q$. For $M_{(0,1),(3,0)}$, we get
	\begin{align*}
		g_{1,2,4}^{\leq}(M)&=m_{1,4}m_{3,4}=0,\\
		g_{2,3,4}^{\leq}(M)&=-m_{3,4}^2=0,
	\end{align*}
	giving
	\[c_{(0,1),(3,0)}^{\leq}=q^2.\]
	
	For $M_{(0,1),(2,1)}$, we get
	\begin{align*}
		g_{2,4,3}^{\leq}(M)&=1=0,
	\end{align*}
	giving  \[c_{(0,1),(2,1)}^{\leq}=0.\]
	
	For $M_{(0,1),(1,2)}$ and $M_{(1),(3)}$,  we get no equations to solve,
	giving $c_{(0,1),(1,2)}^{\leq}=q$ and $c_{(1),(3)}^{\leq}=1$. Together, we have 
	\[a_{q}^{\leq}(M^5)=q^2+q+1\]
	
	Summing all up, we get 	
	\begin{equation*}
		\zeta_{M^5(\Fq)}^{\leq}(s)=1+(q^2+q+1)t+(q^3+2q^2+q+1)t^2+\binom{4}{3}_qt^3+t^4.\qedhere
	\end{equation*}
\end{proof}

Before we proceed to $ \zeta_{M_{a,b}^6(\Fq)}^{\leq}(s)$ and $ \zeta_{M_{a,b}^7(\Fq)}^{\leq}(s)$, we prove the following theorems first:

\begin{thm}\label{thm:M6a,b.var1}
For $a,b\in\Fq^{\times}$, let
\begin{align*}
    |V_{6,a,b}^{(3)}(\Fq)|&=|\{x\in\Fq:(a+b)x^3+(a+b^2)x^2+2abx+a^2\}|,\\
      |V_{6,a,b}^{(4)}(\Fq)|&=|\{x\in\Fq:(a+b^2)x^2+2abx+a^2\}|,
\end{align*}
Then the number of solutions $(u,v,x,y)\in\Fq^{4}$ satisfying the system of equations
\begin{align}
    -y+vx-uy-uv+auxy-buvx-avx^2+bu^2y&=0,\label{eq:M6L}\\
    -(1+bx)v^2-y(ax-bu)v+auy^2&=0\label{eq:M6Q},
\end{align}
is given by 
\[\begin{cases}
     q^2+(q^2-q)|V_{6,a,b}^{(3)}(\Fq)|&a\neq-b,\\
      q^2+(q^2-q)(1+|V_{6,a,b}^{(4)}(\Fq)|)&a=-b.
\end{cases}\]
\end{thm}
\begin{proof}
    We prove this by dividing into different cases.

First, suppose $y=0$. Then \eqref{eq:M6Q} becomes $(1+bx)v^2=0$. Since we are assuming $b\neq0$, we have either $v=0$ or $x=-\frac{1}{b}$.

If $v=0$, then \eqref{eq:M6L} becomes $0=0$ and we have $q^2$ free choices of $u,x\in\Fq$.

If $v\neq0$ and $x=-\frac{1}{b}$, then substituting $y=0$ and $x=-\frac{1}{b}$ in \eqref{eq:M6L} gives 
\[-\frac{v}{b}(1+\frac{a}{b})=0.\]
As we are assuming $v\neq0$, we get no solution if $a\neq-b$, and $(q^2-q)$ solutions if $a=-b$.

Now, suppose $y\neq0$, and write $w:=v/y$. Dividing \eqref{eq:M6L} by $y$ and \eqref{eq:M6Q} by $y^2$ gives
\begin{align}
    w(x-u-bxu-ax^2)&=1+u-axu-bu^2,\label{eq:M6L2}\\
    (1+bx)w^2+(ax-bu)w-au&=0\label{eq:M6Q2}.
\end{align}

If $w=0$, then since $a\neq0$ we have $u=0$ from \eqref{eq:M6Q2}. Substituting $u=w=0$ in \eqref{eq:M6L2} then gives us $0=1$, a contradiction. Therefore we must have $w\neq0$.

For $w\neq0$, note that one can re-write \eqref{eq:M6Q2} as 
\[xw(a+bw)-u(a+bw)+w^2=0.\]
Since $w\neq0$, this implies $a+bw\neq0$, and one can write 
\[u=xw+\frac{w^2}{a+bw}.\]
Substituting this into \eqref{eq:M6L2}, one gets, rather surprisingly, 
\begin{equation}\label{eq:V6ab3}
    (a+b)w^3+(a+b^2)w^2+2abw+a^2=0.
\end{equation}
In particular, $x$ disappears in \eqref{eq:M6L2}. Hence to solve the given system of equations, we need $w$ to be the solution of \eqref{eq:V6ab3}, $x\in\Fq$, $u$ is fixed by $w$ and $x$, and $y\neq0$.

Therefore we have
\[\begin{cases}
     q^2+(q^2-q)|V_{6,a,b}^{(3)}(\Fq)|&a\neq-b,\\
      q^2+(q^2-q)(1+|V_{6,a,b}^{(4)}(\Fq)|)&a=-b.
\end{cases}\]
solutions as required.
\end{proof}

\begin{thm}\label{thm:M7a,b.var1}
For $a,b\in\Fq^{\times}$, let
\begin{align*}
    |V_{7,a,b}^{(3)}(\Fq)|&=|\{x\in\Fq:ax^3+b^2x^2+2abx+a^2\}|.
\end{align*}
Then the number of solutions $(u,v,x,y)\in\Fq^{4}$ satisfying the system of equations
\begin{align}
    -y-uv+auxy-buvx-avx^2+bu^2y&=0,\label{eq:M7L}\\
    -(1+bx)v^2-y(ax-bu)v+auy^2&=0\label{eq:M7Q},
\end{align}
is given by 
\[q^2+(q^2-q)|V_{7,a,b}^{(3)}(\Fq)|.\]
\end{thm}
\begin{proof}
    Again, we prove this by dividing into different cases.

First, suppose $y=0$. Then \eqref{eq:M7Q} becomes $(1+bx)v^2=0$. Since we are assuming $b\neq0$, we have either $v=0$ or $x=-\frac{1}{b}$.

If $v=0$, then \eqref{eq:M7L} becomes $0=0$ and we have $q^2$ free choices of $u,x\in\Fq$.

If $v\neq0$ and $x=-\frac{1}{b}$, then substituting $y=0$ and $x=-\frac{1}{b}$ in \eqref{eq:M7L} gives $\frac{-av}{b^2}=0$. Since we assume $a,b,v\neq0$, it has no solution.

Now, suppose $y\neq0$, and write $w:=v/y$. Dividing \eqref{eq:M7L} by $y$ and \eqref{eq:M7Q} by $y^2$ gives
\begin{align}
    -w(u+bxu+ax^2)&=1-axu-bu^2,\label{eq:M7L2}\\
    (1+bx)w^2+(ax-bu)w-au&=0\label{eq:M7Q2}.
\end{align}

Again, if $w=0$, then since $a\neq0$ we have $u=0$ from \eqref{eq:M7Q2}. Substituting $u=w=0$ in \eqref{eq:M7L2} then gives us $0=1$, a contradiction. Therefore we must have $w\neq0$.

For $w\neq0$, note that one can re-write \eqref{eq:M7Q2} as 
\[xw(a+bw)-u(a+bw)+w^2=0.\]
Since $w\neq0$, this implies $a+bw\neq0$, and one can write 
\[u=xw+\frac{w^2}{a+bw}.\]
Substituting this into \eqref{eq:M7L2}, one gets
\begin{equation}\label{eq:V7ab3}
    aw^3+b^2w^2+2abw+a^2=0.
\end{equation}
 Hence to solve the given system of equations, we need $w$ to be the solution of \eqref{eq:V7ab3}, $x\in\Fq$, $u$ is fixed by $w$ and $x$, and $y\neq0$.

Therefore we have $q^2+(q^2-q)|V_{7,a,b}^{(3)}(\Fq)|$ solutions as required.\end{proof}
\begin{thm}
	For $a,b\in\Fq$, let 
	\[M_{a,b}^6:=\langle e_1,e_2,e_3,e_4:[e_4,e_1]=e_2,[e_4,e_2]=e_3,[e_4,e_3]=-ae_1+be_2+e_3\rangle_{\Fq}.\]
	We have (here $a,b\neq0$)
	\begin{align*}
    \zeta_{M_{a,b}^6(\Fq)}^{\leq}(s)=& 1+(1+q|V_{6,a,b}^{(2)}(\Fq)|)t+(1+q+q^2+(q^2-1)|V_{6,a,b}^{(3)}(\Fq)|+|V_{6,a,b}^{(1)}(\Fq)|)t^2\\
		&+\binom{4}{3}_qt^3+t^4,\;\;\;\;\;(a\neq-b)\\
    \zeta_{M_{a,-a}^6(\Fq)}^{\leq}(s)=& 1+(1+q|V_{6,a,-a}^{(2)}(\Fq)|)t+(1+q+q^2+(q^2-1)(1+|V_{6,a,-a}^{(4)}(\Fq)|)+|V_{6,a,-a}^{(1)}(\Fq)|)t^2\\
		&+\binom{4}{3}_qt^3+t^4,\\
		\zeta_{M_{a,0}^6(\Fq)}^{\leq}(s)=& 1+(1+q|V_{6,a,0}^{(2)}(\Fq)|)t+(1+q+q^2+(q^2-1)|V_{6,a,0}^{(2)}(\Fq)|+|V_{6,a,0}^{(1)}(\Fq)|)t^2\\
		&+\binom{4}{3}_qt^3+t^4,\\
		\zeta_{M_{0,b}^6(\Fq)}^{\leq}(s)=& 1+(1+q+q|V_{3,b}(\Fq)|)t+(1+q+2q^2+q^2|V_{3,b}(\Fq)|)t^2+\binom{4}{3}_qt^3+t^4,\\
		\zeta_{M_{0,0}^6(\Fq)}^{\leq}(s)=&1+(1+2q)t+(1+q+3q^2)t^2+\binom{4}{3}_qt^3+t^4.\qedhere
	\end{align*}
	where
	\begin{align*}
		|V_{6,a,b}^{(1)}(\Fq)|&=\left|\{x\in\Fq:-a^2x^3+ax^2+bx+1=0\}\right|,\\		
		|V_{6,a,b}^{(2)}(\Fq)|&=\left|\{x\in\Fq:ax^3-bx^2+x+1=0\}\right|,\\
        |V_{6,a,b}^{(3)}(\Fq)|&=|\{x\in\Fq:(a+b)x^3+(a+b^2)x^2+2abx+a^2\}|,\\
      |V_{6,a,b}^{(4)}(\Fq)|&=|\{x\in\Fq:(a+b^2)x^2+2abx+a^2\}|.
	\end{align*}
	
\end{thm}

\begin{proof}
	
	Let us count the number of ideals of index $q^2$. For $M_{(0,2),(2,0)}$, we get
	\begin{align*}
		g_{1,2,3}^{\leq}(M)=&-m_{1,4}+m_{1,3}m_{2,4}-m_{1,4}m_{2,3}-m_{2,3}m_{2,4}+am_{1,3}m_{1,4}m_{2,3}-bm_{1,3}m_{2,3}m_{2,4}\\
		&-am_{1,3}^2m_{2,4}+bm_{1,4}m_{2,3}^2=0,\\
		g_{1,2,4}^{\leq}(M)=&-m_{2,4}^2-am_{1,3}m_{1,4}m_{2,4}+bm_{1,4}m_{2,3}m_{2,4}-bm_{1,3}m_{2,4}^2+am_{1,4}^2m_{2,3}=0.
	\end{align*}    
		Write $m_{1,3}=x$, $m_{1,4}=y$, $m_{2,3}=u$, $m_{2,4}=v$. Then $g_{1,2,3}^{\leq}(M)$ and $g_{1,2,4}^{\leq}(M)$ are equivalent to \eqref{eq:M6L} and \eqref{eq:M6Q}, respectively. Hence by Theorem \ref{thm:M6a,b.var1}, for $a,b\in\Fq^{\times}$ one gets
        \[c_{(0,2),(2,0)}^{\leq}=\begin{cases}
          q^2+q(q-1)|V_{6,a,b}^{(3)}(\Fq)|&a\neq-b,\\
      q^2+q(q-1)(1+|V_{6,a,b}^{(4)}(\Fq)|)&a=-b.
\end{cases}\]
  For $a=0$ and/or $b=0$,  the computation becomes easier, and we eventually get
       
	\[c_{(0,2),(2,0)}^{\leq}=\begin{cases}  
    q^2+(q^2-q)|V_{6,a,b}^{(3)}(\Fq)|&a,b\neq0, a\neq-b,\\
      q^2+(q^2-q)(1+|V_{6,a,b}^{(4)}(\Fq)|)&a=-b\neq0,\\
		q^2+(q^2-q)|V_{6,a,0}^{(2)}(\Fq)|&a\neq0,b=0,\\
		2q^2+(q^2-q)|V_{3,b}(\Fq)|-q&a=0,b\neq0,\\
		2q^2-q&a=b=0.
	\end{cases}\]

	For $M_{(0,1,1),(1,1,0)}$, we get
	\begin{align*}
		g_{1,3,2}^{\leq}(M)&=a m_{1,2} m_{1,4}-b m_{1,4}+m_{3,4}=0,\\
		g_{1,3,4}^{\leq}(M)&=a m_{1,4}^2+m_{3,4} m_{1,4}-m_{1,2} m_{3,4}^2=0.
	\end{align*}
	From $g_{1,3,2}^{\leq}(M)=0$, we get $m_{3,4}=bm_{1,4}-am_{1,2}m_{1,4}$. Substituting this in $ g_{1,3,4}^{\leq}(M)$ gives
	\[m_{1,4}^2(-a^2m_{1,2}^3+2abm_{1,2}^2-(a+b^2)m_{1,2}+(a+b))=0.\]
Hence either $m_{1,4}=0$ or $-a^2m_{1,2}^3+2abm_{1,2}^2-(a+b^2)m_{1,2}+(a+b)=0$. First suppose $m_{1,4}=0$. Then we have $m_{3,4}=0$ and $m_{1,2}\in\Fq$, getting $q$ solutions.

Suppose $m_{1,4}\neq0$. If $a,b\neq=0$ and $a\neq-b$, then $m_{1,2}\neq0$ and the number of solutions satisfying $-a^2m_{1,2}^3+2abm_{1,2}^2-(a+b^2)m_{1,2}+(a+b)$ is equivalent to $|V_{6,a,b}^{(3)}(\Fq)|$. If $a=-b\neq0$, then either $m_{1,2}=0$ or $-a^2m_{1,2}^2+2abm_{1,2}-(a+b^2)=0$, giving $1+|V_{6,a,b}^{(4)}(\Fq)|$ solutions. Similar arguments for other cases give
	\[c_{(0,1,1),(1,1,0)}^{\leq}=\begin{cases}
    q+(q-1)|V_{6,a,b}^{(3)}(\Fq)|&a,b\neq0, a\neq-b,\\
    q+(q-1)(1+|V_{6,a,b}^{(4)}(\Fq)|)&a=-b\neq0,\\
		q+(q-1)|V_{6,a,0}^{(2)}(\Fq)|&a\neq0,b=0,\\
		2q-1&a=0,b\neq0,\\
		q^2&a=b=0.
	\end{cases}\]
	
	For $M_{(0,2),(1,1)}$, we get
	\begin{align*}
		g_{1,4,2}^{\leq}(M)&=1+bm_{1,3}-am_{1,2}m_{1,3}=0,\\
		g_{1,4,3}^{\leq}(M)&=m_{1,2}+m_{1,3}-am_{1,3}^2=0.
	\end{align*}
	From $g_{1,4,3}^{\leq}(M)=0$, we get $m_{1,2}=am_{1,3}^2-m_{1,3}$. Substituting this in $g_{1,4,2}^{\leq}(M)$ gives
	\[-a^2m_{1,3}^3+am_{1,3}^2+bm_{1,3}+1=0,\]
	giving $c_{(0,2),(1,1)}^{\leq}=|V_{6,a,b}^{(1)}(\Fq)|$. 
	
	For $M_{(1,1),(2,0)}$, we get
	\begin{align*}
		g_{2,3,1}^{\leq}(M)&=-am_{2,4}=0,\\
		g_{2,3,4}^{\leq}(M)&=bm_{2,4}^2+m_{2,4}m_{3,4}-m_{3,4}^2=0.
	\end{align*}
	Solving this system of equations gives
	\[c_{(1,1),(2,0)}^{\leq}=
	\begin{cases}
		1&a\neq0,\\
		(q-1)|V_{3,b}(\Fq)|+1&a=0,b\neq0,\\
		2q-1&a=b=0,
	\end{cases}\]

	For $M_{(1,1),(1,1)}$, we get
	\begin{align*}
		g_{2,4,1}^{\leq}(M)&=am_{2,3}=0,\\
		g_{2,4,3}^{\leq}(M)&=1+m_{2,3}-bm_{2,3}^2=0.
	\end{align*}
	If $a\neq0$, then we must have $m_{2,3}=0$, making $g_{2,4,3}^{\leq}(M)=0$ unsolvable.
	
	If $a=0$, then we have $|V_{3,b}(\Fq)|$ solutions. Therefore we have
	\[c_{(1,1),(1,1)}^{\leq}=
	\begin{cases}
		0&a\neq0\\
		|V_{3,b}(\Fq)|&a=0.
	\end{cases}\]
	
	For $M_{(2),(2)}$, we get
	\begin{align*}
		g_{3,4,1}^{\leq}(M)&=a=0,\\
		g_{3,4,2}^{\leq}(M)&=b=0,
	\end{align*}
	giving 
	\[c_{(2),(2)}^{\leq}=\begin{cases}
		0&a\neq0\textrm{ or }b\neq0\\
		1&a=b=0.
	\end{cases}\]
	
	Together, as $|V_{6,0,b}^{(1)}(\Fq)|=1$, $|V_{6,0,0}^{(1)}(\Fq)|=0$, and $|V_{3,0}^{(2)}(\Fq)|=1$, we have 
	\[a_{q^2}^{\leq}(M_{a,b}^6)=\begin{cases}
		1+q+q^2+(q^2-1)|V_{6,a,b}^{(3)}(\Fq)|+|V_{6,a,b}^{(1)}(\Fq)|&a,b\neq0,a\neq-b,\\
		1+q+q^2+(q^2-1)(1+|V_{6,a,b}^{(4)}(\Fq)|)+|V_{6,a,b}^{(1)}(\Fq)|&a=-b\neq0,\\
		1+q+q^2+(q^2-1)|V_{6,a,0}^{(2)}(\Fq)|+|V_{6,a,0}^{(1)}(\Fq)|&a\neq0,b=0,\\
		1+q+2q^2+q^2|V_{3,b}(\Fq)|&a=0,b\neq0,\\
		1+q+3q^2&a=b=0
	\end{cases}\]

	Finally, let us count the number of ideals of index $q$. For $M_{(0,1),(3,0)}$, we get
	\begin{align*}
		g_{1,2,4}^{\leq}(M)&=-m_{2,4}^2+m_{1,4}m_{3,4}=0,\\
		g_{1,3,4}^{\leq}(M)&=a m_{1,4}^2+b m_{2,4} m_{1,4}+m_{3,4} m_{1,4}-m_{2,4} m_{3,4}=0,\\
		g_{2,3,4}^{\leq}(M)&=a m_{1,4} m_{2,4}+b m_{2,4}^2+m_{3,4} m_{2,4}-m_{3,4}^2=0.
	\end{align*}
	If $m_{1,4}=0$, it forces $m_{2,4}=m_{3,4}=0$, giving us 1 solution. 
	
	If $m_{1,4}\neq0$, we can write $m_{3,4}=m_{2,4}^2/m_{1,4}$. Substituting this in $g_{1,3,4}^{\leq}(M)=g_{2,3,4}^{\leq}(M)=0$ gives
	\[am_{1,4}^3+bm_{1,4}^2m_{2,4}+m_{2,4}^2m_{1,4}-m_{2,4}^3=0.\]
	If $m_{2,4}=0$, then we need $am_{1,4}^3=0$. Since we are assuming $m_{1,4}\neq0$, we get no solution if $a\neq0$, and $q-1$ solutions if $a=0$.
	
	If $m_{2,4}\neq0$, we get $(q-1)|V_{6,a,b}^{(2)}(\Fq)|$ solutions.
	
	Therefore we have
	\[c_{(0,1),(3,0)}^{\leq}=\begin{cases}
		1+(q-1)|V_{6,a,b}^{(2)}(\Fq)|&a\neq0,\\
		q+(q-1)|V_{6,a,b}^{(2)}(\Fq)|&a=0.
	\end{cases}\]
	
	For $M_{(0,1),(2,1)}$, we get
	\begin{align*}
		g_{1,4,3}^{\leq}(M)&=m_{1,3}-m_{2,3}-a m_{1,3}^2-b m_{2,3} m_{1,3}=0,\\
		g_{2,4,3}^{\leq}(M)&=1+m_{2,3}-a m_{1,3} m_{2,3}-b m_{2,3}^2=0.
	\end{align*}
	Solving this system of equations gives
	\begin{align*}
		m_{1,3}&=-m_{2,3}^2&1+m_{2,3}-bm_{2,3}^2+am_{2,3}^3&=0,
	\end{align*}
	giving  \[c_{(0,1),(2,1)}^{\leq}=|V_{6,a,b}^{(2)}(\Fq)|.\]
	
	For $M_{(0,1),(1,2)}$, we get
	\begin{align*}
		g_{1,4,2}^{\leq}(M)&=1=0,
	\end{align*}
	giving
	\[c_{(0,1),(1,2)}^{\leq}=0.\]
	
	for $M_{(1),(3)}$,  we get 
	\begin{align*}
		g_{3,4,1}^{\leq}(M)&=a=0,
	\end{align*}
	giving \[c_{(1),(3)}^{\leq}=\begin{cases}
		0&a\neq0,\\
		1&a=0.
	\end{cases}\]
	
	Together, we have 
	\[a_{q}^{\leq}(M_{a,b}^6)=\begin{cases}
		1+q|V_{6,a,b}^{(2)}(\Fq)|&a\neq0\\
		1+q+q|V_{3,b}(\Fq)|&a=0.
	\end{cases}\]
	
	Summing all up, for $a,b\in\Fq^{\times}$ we get

	\begin{align*}
    \zeta_{M_{a,b}^6(\Fq)}^{\leq}(s)=& 1+(1+q|V_{6,a,b}^{(2)}(\Fq)|)t+(1+q+q^2+(q^2-1)|V_{6,a,b}^{(3)}(\Fq)|+|V_{6,a,b}^{(1)}(\Fq)|)t^2\\
		&+\binom{4}{3}_qt^3+t^4,\;\;\;\;\;(a\neq-b)\\
    \zeta_{M_{a,-a}^6(\Fq)}^{\leq}(s)=& 1+(1+q|V_{6,a,-a}^{(2)}(\Fq)|)t+(1+q+q^2+(q^2-1)(1+|V_{6,a,-a}^{(4)}(\Fq)|)+|V_{6,a,-a}^{(1)}(\Fq)|)t^2\\
		&+\binom{4}{3}_qt^3+t^4,\\
		\zeta_{M_{a,0}^6(\Fq)}^{\leq}(s)=& 1+(1+q|V_{6,a,0}^{(2)}(\Fq)|)t+(1+q+q^2+(q^2-1)|V_{6,a,0}^{(2)}(\Fq)|+|V_{6,a,0}^{(1)}(\Fq)|)t^2\\
		&+\binom{4}{3}_qt^3+t^4,\\
		\zeta_{M_{0,b}^6(\Fq)}^{\leq}(s)=& 1+(1+q+q|V_{3,b}(\Fq)|)t+(1+q+2q^2+q^2|V_{3,b}(\Fq)|)t^2+\binom{4}{3}_qt^3+t^4,\\
		\zeta_{M_{0,0}^6(\Fq)}^{\leq}(s)=&1+(1+2q)t+(1+q+3q^2)t^2+\binom{4}{3}_qt^3+t^4.\qedhere
	\end{align*}
	
\end{proof}

\begin{thm}\label{thm:sol.Ma}
	For $a,b\in\Fq$, let 
	\[M_{a,b}^7:=\langle e_1,e_2,e_3,e_4:[e_4,e_1]=e_2,[e_4,e_2]=e_3,[e_4,e_3]=-ae_1+be_2\rangle_{\Fq}.\]
	We have

	\begin{align*}    
		\zeta_{M_{a,b}^7(\Fq)}^{\leq}(s)=& 1+(1+q|V_{7,a,b}^{(2)}(\Fq)|)t+(1+q+q^2+(q^2-1)|V_{7,a,b}^{(3)}(\Fq)|+|V_{7,a,b}^{(1)}(\Fq)|)t^2\\
		&+\binom{4}{3}_qt^3+t^4,\\
		\zeta_{M_{a,0}^7(\Fq)}^{\leq}(s)=& 1+(1+q|V_{7,a,0}^{(2)}(\Fq)|)t+(1+q+q^2+(q^2-1)|V_{7,a,0}^{(2)}(\Fq)|+|V_{7,a,0}^{(1)}(\Fq)|)t^2\\
		&+\binom{4}{3}_qt^3+t^4,\\
		\zeta_{M_{0,b}^7(\Fq)}^{\leq}(s)=& 1+(1+q+q|V_{4,b}(\Fq)|)t+(1+q+2q^2+q^2|V_{4,b}(\Fq)|)t^2+\binom{4}{3}_qt^3+t^4,\\
		\zeta_{M_{0,0}^7(\Fq)}^{\leq}(s)=&1+(1+q)t+(1+q+2q^2)t^2+\binom{4}{3}_qt^3+t^4,
	\end{align*}
where	
	\begin{align*}
		|V_{7,a,b}^{(1)}(\Fq)|&=\left|\{x\in\Fq:-a^2x^3+bx+1=0 \}\right|,\\
		|V_{7,a,b}^{(2)}(\Fq)|&=\left|\{x\in\Fq:ax^3-bx^2+1=0\}\right|,\\
    |V_{7,a,b}^{(3)}(\Fq)|&=|\{x\in\Fq:ax^3+b^2x^2+2abx+a^2\}|.
	\end{align*}
	
\end{thm}

\begin{proof}
	
	Let us count the number of ideals of index $q^2$. For $M_{(0,2),(2,0)}$, we get
	\begin{align*}
		g_{1,2,3}^{\leq}(M)=&-m_{1,4}-m_{2,3}m_{2,4}+am_{1,3}m_{1,4}m_{2,3}-bm_{1,3}m_{2,3}m_{2,4}-am_{1,3}^2m_{2,4}+bm_{1,4}m_{2,3}^2=0,\\
		g_{1,2,4}^{\leq}(M)=&-m_{2,4}^2-am_{1,3}m_{1,4}m_{2,4}+bm_{1,4}m_{2,3}m_{2,4}-bm_{1,3}m_{2,4}^2+am_{1,4}^2m_{2,3}=0.
	\end{align*}
	Again, 	write $m_{1,3}=x$, $m_{1,4}=y$, $m_{2,3}=u$, $m_{2,4}=v$. Then $g_{1,2,3}^{\leq}(M)$ and $g_{1,2,4}^{\leq}(M)$ are equivalent to \eqref{eq:M7L} and \eqref{eq:M7Q}, respectively. Hence by Theorem \ref{thm:M7a,b.var1}, for $a,b\in\Fq^{\times}$ one gets
        \[c_{(0,2),(2,0)}^{\leq}= q^2+(q^2-q)|V_{7,a,b}^{(3)}(\Fq)|.\]
      For $a=0$ and/or $b=0$, the  computation gets easier, giving
       
	\[c_{(0,2),(2,0)}^{\leq}=\begin{cases}  
     q^2+(q^2-q)|V_{7,a,b}^{(3)}(\Fq)|&a,b\neq0,\\
		q^2+(q^2-q)|V_{7,a,0}^{(2)}(\Fq)|&a\neq0,b=0,\\
		2q^2+(q^2-q)|V_{3,b}(\Fq)|-q&a=0,b\neq0,\\
		2q^2-q&a=b=0.
	\end{cases}\]

	For $M_{(0,1,1),(1,1,0)}$, we get
	\begin{align*}
		g_{1,3,2}^{\leq}(M)&=a m_{1,2} m_{1,4}-b m_{1,4}+m_{3,4}=0,\\
		g_{1,3,4}^{\leq}(M)&=a m_{1,4}^2-m_{1,2} m_{3,4}^2=0.
	\end{align*}
	From $g_{1,3,2}^{\leq}(M)=0$, we get $m_{3,4}=bm_{1,4}-am_{1,2}m_{1,4}$. Substituting this in $ g_{1,3,4}^{\leq}(M)$ gives
	\[m_{1,4}^2(-a^2m_{1,2}^3+2abm_{1,2}^2-b^2m_{1,2}+a)=0.\]
One gets
	\[c_{(0,1,1),(1,1,0)}^{\leq}=\begin{cases}
     q+(q-1)|V_{7,a,b}^{(3)}(\Fq)|&a,b\neq0,\\
		q+(q-1)|V_{7,a,0}^{(2)}(\Fq)|&a\neq0,b=0,\\
		2q-1&a=0,b\neq0,\\
		q^2&a=b=0.
	\end{cases}\]
	
	For $M_{(0,2),(1,1)}$, we get
	\begin{align*}
		g_{1,4,2}^{\leq}(M)&=1+bm_{1,3}-am_{1,2}m_{1,3}=0,\\
		g_{1,4,3}^{\leq}(M)&=m_{1,2}-am_{1,3}^2=0.
	\end{align*}
	From $g_{1,4,3}^{\leq}(M)=0$, we get $m_{1,2}=am_{1,3}^2$. Substituting this in $g_{1,4,2}^{\leq}(M)$ gives
	\[-a^2m_{1,3}^3+bm_{1,3}+1=0,\]
	giving $c_{(0,2),(1,1)}^{\leq}=|V_{7,a,b}^{(1)}(\Fq)|$. 
	
	For $M_{(1,1),(2,0)}$, we get
	\begin{align*}
		g_{2,3,1}^{\leq}(M)&=-am_{2,4}=0,\\
		g_{2,3,4}^{\leq}(M)&=bm_{2,4}^2-m_{3,4}^2=0.
	\end{align*}
	Solving this system of equations gives
	\[c_{(1,1),(2,0)}^{\leq}=
	\begin{cases}
		1&a\neq0,\\
		(q-1)|V_{4,b}(\Fq)|+1&a=0,b\neq0,\\
		q&a=b=0.
	\end{cases}\]

	For $M_{(1,1),(1,1)}$, we get
	\begin{align*}
		g_{2,4,1}^{\leq}(M)&=am_{2,3}=0,\\
		g_{2,4,3}^{\leq}(M)&=1-bm_{2,3}^2=0.
	\end{align*}
	If $a\neq0$, then we must have $m_{2,3}=0$, making $g_{2,4,3}^{\leq}(M)=0$ unsolvable.
	
	If $a=0$, then we have $|V_{4,b}(\Fq)|$ solutions. Therefore we have
	\[c_{(1,1),(1,1)}^{\leq}=
	\begin{cases}
		0&a\neq0\\
		|V_{4,b}(\Fq)|&a=0.
	\end{cases}\]
	
	For $M_{(2),(2)}$, we get
	\begin{align*}
		g_{3,4,1}^{\leq}(M)&=a=0,\\
		g_{3,4,2}^{\leq}(M)&=b=0,
	\end{align*}
	giving 
	\[c_{(2),(2)}^{\leq}=\begin{cases}
		0&a\neq0\textrm{ or }b\neq0\\
		1&a=b=0.
	\end{cases}\]
	
	Together, as $|V_{7,0,b}^{(1)}(\Fq)|=1$, $|V_{7,0,0}^{(1)}(\Fq)|=0$, and $|V_{7,0,0}^{(2)}(\Fq)|=0$, we have 
	\[a_{q^2}^{\leq}(M_{a,b}^7)=\begin{cases}
    1+q+q^2+(q^2-1)|V_{7,a,b}^{(3)}(\Fq)|+|V_{7,a,b}^{(1)}(\Fq)|&a,b\neq0,\\
		1+q+q^2+(q^2-1)|V_{7,a,0}^{(2)}(\Fq)|+|V_{7,a,0}^{(1)}(\Fq)|&a\neq0,b=0,\\
		1+q+2q^2+q^2|V_{4,b}(\Fq)|&a=0,b\neq0,\\
		1+q+2q^2&a=b=0
	\end{cases}\]

	Finally, let us count the number of ideals of index $q$. For $M_{(0,1),(3,0)}$, we get
	\begin{align*}
		g_{1,2,4}^{\leq}(M)&=-m_{2,4}^2+m_{1,4}m_{3,4}=0,\\
		g_{1,3,4}^{\leq}(M)&=a m_{1,4}^2+b m_{2,4} m_{1,4}-m_{2,4} m_{3,4}=0,\\
		g_{2,3,4}^{\leq}(M)&=a m_{1,4} m_{2,4}+b m_{2,4}^2-m_{3,4}^2=0.
	\end{align*}
	
	If $m_{1,4}=0$, it forces $m_{2,4}=m_{3,4}=0$, giving us 1 solution. 
	
	If $m_{1,4}\neq0$, we can write $m_{3,4}=m_{2,4}^2/m_{1,4}$. Substituting this in $g_{1,3,4}^{\leq}(M)=g_{2,3,4}^{\leq}(M)=0$ gives
	\[am_{1,4}^3+bm_{1,4}^2m_{2,4}-m_{2,4}^3=0.\]
	
	If $m_{2,4}=0$, then we need $am_{1,4}^3=0$. Since we are assuming $m_{1,4}\neq0$, we get no solution if $a\neq0$, and $q-1$ solutions if $a=0$.
	
	If $m_{2,4}\neq0$, we get $(q-1)|V_{7,a,b}^{(2)}(\Fq)|$ solutions.
	giving us $(q-1)|V_{7,a,b}^{(2)}(\Fq)|$ solutions.
	
	Therefore we have
	\[c_{(0,1),(3,0)}^{\leq}=\begin{cases}  
		1+(q-1)|V_{7,a,b}^{(2)}(\Fq)|&a\neq0,\\
		q+(q-1)|V_{7,a,b}^{(2)}(\Fq)|&a=0.\end{cases}\]
	
	For $M_{(0,1),(2,1)}$, we get
	\begin{align*}
		g_{1,4,3}^{\leq}(M)&=-m_{2,3}-a m_{1,3}^2-b m_{2,3} m_{1,3}=0,\\
		g_{2,4,3}^{\leq}(M)&=1-a m_{1,3} m_{2,3}-b m_{2,3}^2=0.
	\end{align*}
	Solving this system of equations gives
	\begin{align*}
		m_{1,3}&=-m_{2,3}^2&1-bm_{2,3}^2+am_{2,3}^3&=0,
	\end{align*}
	giving  \[c_{(0,1),(2,1)}^{\leq}=|V_{7,a,b}^{(2)}(\Fq)|.\]
	
	For $M_{(0,1),(1,2)}$, we get
	\begin{align*}
		g_{1,4,2}^{\leq}(M)&=1=0,
	\end{align*}
	giving
	\[c_{(0,1),(1,2)}^{\leq}=0.\]
	
	for $M_{(1),(3)}$,  we get 
	\begin{align*}
		g_{3,4,1}^{\leq}(M)&=a=0,
	\end{align*}
	giving \[c_{(1),(3)}^{\leq}=\begin{cases}
		0&a\neq0,\\
		1&a=0.
	\end{cases}\]
	
	Together, we have 
	\[a_{q}^{\leq}(M_{a,b}^7)=\begin{cases}
		1+q|V_{7,a,b}^{(2)}(\Fq)|&a\neq0\\
		1+q+q|V_{7,0,b}^{(2)}(\Fq)|&a=0.
	\end{cases}\]
	
	Summing all up, we get

	\begin{align*}    
		\zeta_{M_{a,b}^7(\Fq)}^{\leq}(s)=& 1+(1+q|V_{7,a,b}^{(2)}(\Fq)|)t+(1+q+q^2+(q^2-1)|V_{7,a,b}^{(3)}(\Fq)|+|V_{7,a,b}^{(1)}(\Fq)|)t^2\\
		&+\binom{4}{3}_qt^3+t^4,\\
		\zeta_{M_{a,0}^7(\Fq)}^{\leq}(s)=& 1+(1+q|V_{7,a,0}^{(2)}(\Fq)|)t+(1+q+q^2+(q^2-1)|V_{7,a,0}^{(2)}(\Fq)|+|V_{7,a,0}^{(1)}(\Fq)|)t^2\\
		&+\binom{4}{3}_qt^3+t^4,\\
		\zeta_{M_{0,b}^7(\Fq)}^{\leq}(s)=& 1+(1+q+q|V_{4,b}(\Fq)|)t+(1+q+2q^2+q^2|V_{4,b}(\Fq)|)t^2+\binom{4}{3}_qt^3+t^4,\\
		\zeta_{M_{0,0}^7(\Fq)}^{\leq}(s)=&1+(1+q)t+(1+q+2q^2)t^2+\binom{4}{3}_qt^3+t^4.\qedhere
	\end{align*}
	
\end{proof}	

\begin{rem}
	Note that we also have $\zeta_{M_{0,0}^7(\Fq)}^{\leq}(s)=\zeta_{M_3(\Fq)}^{\leq}(s)$ in \cite[Theorem 3.7]{Lee25}.
\end{rem}

\begin{thm}
	Let 
	\[M^8:=\langle e_1,e_2,e_3,e_4:[e_1,e_2]=e_2,[e_3,e_4]=e_4\rangle_{\Fq}.\]
	We have
	\begin{equation*}
		\zeta_{M^8(\Fq)}^{\leq}(s)=1+(1+3q)t+(1+q+4q^2)t^2+\binom{4}{3}_qt^3+t^4.
	\end{equation*}
\end{thm}

\begin{proof}
	
	Let us count the number of ideals of index $q^2$. For $M_{(0,2),(2,0)}$, we need to solve
	\begin{align*}
		g_{1,2,3}^{\leq}(M)&=m_{2,3}=0,\\
		g_{1,2,4}^{\leq}(M)&=m_{2,4}+m_{1,4}m_{2,3}-m_{1,3}m_{2,4}=0.
	\end{align*}
	Solving the given system of equations gives
	\[c_{(0,2),(2,0)}^{\leq}=2q^2-q\]
	
	For $M_{(0,1,1),(1,1,0)}$, we get
	\begin{align*}
		g_{1,3,2}^{\leq}(M)&=0=0,\\
		g_{1,3,4}^{\leq}(M)&=m_{1,4}=0.
	\end{align*}
	Solving the given system of equations gives
	\[c_{(0,1,1),(1,1,0)}^{\leq}=q^2.\]
	
	For $M_{(0,2),(1,1)}$, we get no equations to solve, giving
	giving $c_{(0,2),(1,1)}^{\leq}=q^2$. 
	
	For $M_{(1,1),(2,0)}$, we get
	\begin{align*}
		g_{2,3,1}^{\leq}(M)&=0=0,\\
		g_{2,3,4}^{\leq}(M)&=m_{2,4}=0,
	\end{align*}
	giving 
	\[c_{(1,1),(2,0)}^{\leq}=q.\]
	
	For $M_{(1,1),(1,1)}$, and $M_{(2),(2)}$, we get no equations to solve, giving $c_{(1,1),(1,1)}^{\leq}=q$ and $c_{(2),(2)}^{\leq}=1$.
	
	Together, we have 
	\[a_{q^2}^{\leq}(M_a^3)=1+q+4q^2.\]

	Finally, let us count the number of ideals of index $q$. For $M_{(0,1),(3,0)}$, we get
	\begin{align*}
		g_{1,2,4}^{\leq}(M)&=m_{2,4}=0,\\
		g_{1,3,4}^{\leq}(M)&=m_{1,4}=0,\\
		g_{2,3,4}^{\leq}(M)&=m_{2,4}=0.
	\end{align*}
	Solving the system of equations gives
	\[c_{(0,1),(3,0)}^{\leq}=q.\]
	
	For $M_{(0,1),(2,1)}$, we get
	\begin{align*}
		g_{1,2,3}^{\leq}(M)&=m_{2,3}=0,
	\end{align*}
	giving  \[c_{(0,1),(2,1)}^{\leq}=q.\]
	
	For $M_{(0,1),(1,2)}$ and $M_{(1),(3)}$,  we get no equations to solve,
	giving $c_{(0,1),(1,2)}^{\leq}=q$ and $c_{(1),(3)}^{\leq}=1$. Together, we have 
	\[a_{q}^{\leq}(M_a^3)=1+3q\]
	
	Summing all up, we get 	
	\begin{equation*}
		\zeta_{M^8(\Fq)}^{\leq}(s)=1+(1+3q)t+(1+q+4q^2)t^2+\binom{4}{3}_qt^3+t^4.
	\end{equation*}
\end{proof}

\begin{thm}\label{thm:sol.Ma9.sub}
	Let $a\in\Fq$ such that $x^2-x-a$ has no roots in $\Fq$, and let
	\[M_a^9:=\langle e_1,e_2,e_3,e_4:[e_4,e_1]=e_1+ae_2,[e_4,e_2]=e_1,[e_3,e_1]=e_1,[e_3,e_2]=e_2\rangle_{\Fq},\]
	
	We have
	\begin{equation*}
		\zeta_{M_a^9(\Fq)}^{\leq}(s)=1+(1+q)t+(1+q+2q^2)t^2+\binom{4}{3}_qt^3+t^4.
	\end{equation*}
\end{thm}

\begin{proof}
	
	Let us count the number of ideals of index $q^2$. For $M_{(0,2),(2,0)}$, we get
	\begin{align*}
		g_{1,2,3}^{\leq}(M)&=m_{1,3}m_{1,4}-m_{1,3}m_{2,4}-am_{2,3}m_{2,4}=0,\\
		g_{1,2,4}^{\leq}(M)&=m_{1,4}^2-m_{2,3} m_{1,4}-m_{2,4} m_{1,4}+m_{1,3} m_{2,4}-a m_{2,4}^2=0.
	\end{align*}
	Solving the given system of equations give
	\[c_{(0,2),(2,0)}^{\leq}=2q^2-1\]
	
	For $M_{(0,1,1),(1,1,0)}$, we get
	\begin{align*}
		g_{1,3,2}^{\leq}(M)&=a m_{3,4}-m_{3,4} m_{1,2}-m_{3,4} m_{1,2}^2=0,\\
		g_{1,3,4}^{\leq}(M)&=-m_{1,4}-m_{3,4} m_{1,4}-m_{1,2} m_{3,4} m_{1,4}=0.
	\end{align*}
	If $m_{3,4}\neq 0$, then 
	\[a m_{3,4}-m_{3,4} m_{1,2}-m_{3,4} m_{1,2}^2=m_{3,4}(a-m_{1,2}-m_{1,2}^2)=0,\]
	hence $m_{1,2}$ needs to be a solution of $a-x-x^2$, which is impossible by the definition of $M_{a}^9$. Hence we get no solution.
	
	If $m_{3,4}=0$, then we get $m_{1,4}=0$ and $m_{1,2}\in\Fq$, giving $q$ solutions.
	
	Therefore we have
	\[c_{(0,1,1),(1,1,0)}^{\leq}=q\]
	
	For $M_{(0,2),(1,1)}$, we get
	\begin{align*}
		g_{1,4,2}^{\leq}(M)&=a-m_{1,2}-m_{1,2}^2=0.
	\end{align*}
	Again, since $m_{1,2}$ needs to be a solution of $a-x-x^2$, we get $c_{(0,2),(1,1)}^{\leq}=0$. 
	
	For $M_{(1,1),(2,0)}$, we get
	\begin{align*}
		g_{2,3,1}^{\leq}(M)&=m_{3,4}=0,\\
		g_{2,3,4}^{\leq}(M)&=-m_{2,4}=0,
	\end{align*}
	giving
	\[c_{(1,1),(2,0)}^{\leq}=1.\]
	
	For $M_{(1,1),(1,1)}$, we get
	\begin{align*}
		g_{2,4,1}^{\leq}(M)&=1=0,
	\end{align*}
	giving 
	\[c_{(1,1),(1,1)}^{\leq}=0.\]
	
	For $M_{(2),(2)}$, we get no equations to solve, giving
	\[c_{(2),(2)}^{\leq}=1.\]
	
	Together, we have 
	\[a_{q^2}^{\leq}(M_a^9)=1+q+2q^2.\]

	Finally, let us count the number of ideals of index $q$. For $M_{(0,1),(3,0)}$, we get
	\begin{align*}
		g_{1,2,4}^{\leq}(M)&=m_{1,4}^2-m_{1,4}m_{2,4}-am_{2,4}^2=0,\\
		g_{1,3,4}^{\leq}(M)&=-m_{1,4}-m_{1,4}m_{3,4}-am_{2,4}m_{3,4}=0,\\
		g_{2,3,4}^{\leq}(M)&=-m_{2,4}-m_{1,4}m_{3,4}=0.
	\end{align*}
	From $g_{2,3,4}^{\leq}(M)$ we get $m_{2,4}=-m_{1,4}m_{3,4}$. Substituting this in $  g_{1,3,4}^{\leq}(M)$ gives $m_{1,4}(-1-m_{3,4}+am_{3,4}^2)=0$. by the definition of $M_{a}^{9}$ we get $m_{1,4}=m_{2,4}=0$ and $m_{3,4}\in\Fq$, giving
	\[c_{(0,1),(3,0)}^{\leq}=q.\]
	
	For $M_{(0,1),(2,1)}$, we get
	\begin{align*}
		g_{1,4,3}^{\leq}(M)&=-m_{1,3}-am_{2,3}=0,\\
		g_{2,4,3}^{\leq}(M)&=-m_{1,3}=0.
	\end{align*}
	since $a\neq0$, we get \[c_{(0,1),(2,1)}^{\leq}=1.\]
	
	For $M_{(0,1),(1,2)}$, we get
	\begin{align*}
		g_{1,4,2}^{\leq}(M)&=a-m_{1,2}-m_{1,2}^2=0,
	\end{align*}
	which is again impossible. Therefore we have
	\[c_{(0,1),(1,2)}^{\leq}=0.\]
	
	for $M_{(1),(3)}$,  we get
	\begin{align*}
		g_{2,4,1}^{\leq}(M)&=1=0,
	\end{align*}
	giving \[c_{(1),(3)}=0.\] 
	
	Together, we have 
	\[a_{q}^{\leq}(M_a^9)=1+q.\]
	
	Summing all up, we get 	
	\begin{equation*}
		\zeta_{M_a^9(\Fq)}^{\leq}(s)=1+(1+q)t+(1+q+2q^2)t^2+\binom{4}{3}_qt^3+t^4.
	\end{equation*}
\end{proof}	

\begin{thm}
	Let 
	\[M^{12}:=\langle e_1,e_2,e_3,e_4:[e_4,e_1]=e_1,[e_4,e_2]=2e_2,[e_4,e_3]=e_3, [e_3,e_1]=e_2\rangle_{\Fq}.\]
	We have
	\begin{equation*}
		\zeta_{M^{12}(\Fq)}^{\leq}(s)=1+(1+q+q^2)t+(1+q+2q^2+q^3)t^2+\binom{4}{3}_qt^3+t^4.
	\end{equation*}
\end{thm}

\begin{proof}
	
	Let us count the number of subalgebras of index $q^2$. For $M_{(0,2),(2,0)}$, we need to solve
	\begin{align*}
		g_{1,2,3}^{\leq}(M)&=m_{2,3}(m_{1,4}-m_{2,3})=0,\\
		g_{1,2,4}^{\leq}(M)&=m_{2,4}(m_{1,4}-m_{2,3})=0.
	\end{align*}
	Solving the given system of equations gives
	\[c_{(0,2),(2,0)}^{\leq}=q^3+q^2-q\]
	
	For $M_{(0,1,1),(1,1,0)}$, we get
	\begin{align*}
		g_{1,3,2}^{\leq}(M)&=1+m_{1,2}m_{3,4}=0,\\
		g_{1,3,4}^{\leq}(M)&=0=0.
	\end{align*}
	Solving the given system of equations gives
	\[c_{(0,1,1),(1,1,0)}^{\leq}=q^2-q.\]
	
	For $M_{(0,2),(1,1)}$, we get 
	\begin{align*}
		g_{1,4,2}^{\leq}(M)&=m_{1,2}=0,\\
		g_{1,4,3}^{\leq}(M)&=0=0,
	\end{align*} giving
	giving $c_{(0,2),(1,1)}^{\leq}=q$. 
	
	For $M_{(1,1),(2,0)}$, we get
	\begin{align*}
		g_{2,3,1}^{\leq}(M)&=0=0,\\
		g_{2,3,4}^{\leq}(M)&=-m_{2,4}m_{3,4}=0,
	\end{align*}
	giving 
	\[c_{(1,1),(2,0)}^{\leq}=2q-1.\]
	
	For $M_{(1,1),(1,1)}$, we get
	\begin{align*}
		g_{2,4,1}^{\leq}(M)&=0=0,\\
		g_{2,4,3}^{\leq}(M)&=-m_{2,3}=0,
	\end{align*}
	giving 
	\[c_{(1,1),(1,1)}^{\leq}=1.\]

	For $M_{(2),(2)}$, we get no equations to solve, giving  \[c_{(2),(2)}^{\leq}=1.\]
	
	Together, we have 
	\[a_{q^2}^{\leq}(M^{12})=1+q+2q^2+q^3.\]

	Finally, let us count the number of ideals of index $q$. For $M_{(0,1),(3,0)}$, we get
	\begin{align*}
		g_{1,2,4}^{\leq}(M)&=m_{1,4}m_{2,4}=0,\\
		g_{1,3,4}^{\leq}(M)&=-m_{2,4}=0,\\
		g_{2,3,4}^{\leq}(M)&=-m_{2,4}m_{3,4}=0.
	\end{align*}
	Solving the system of equations gives
	\[c_{(0,1),(3,0)}^{\leq}=q^2.\]
	
	For $M_{(0,1),(2,1)}$, we get
	\begin{align*}
		g_{1,2,3}^{\leq}(M)&=m_{2,3}^2=0,\\    
		g_{2,4,3}^{\leq}(M)&=-m_{2,3}=0,
	\end{align*}
	giving  \[c_{(0,1),(2,1)}^{\leq}=q.\]
	
	For $M_{(0,1),(1,2)}$, we get 
	\begin{align*}
		g_{1,3,2}^{\leq}(M)&=1=0,
	\end{align*}
	giving  \[c_{(0,1),(2,1)}^{\leq}=0.\]

	For $M_{(1),(3)}$,  we get no equations to solve,
	giving \[c_{(1),(3)}^{\leq}=1.\] 
	
	Together, we have 
	\[a_{q}^{\leq}(M^{12})=1+q+q^2\]
	
	Summing all up, we get 	
	\begin{equation*}
		\zeta_{M^{12}(\Fq)}^{\leq}(s)=1+(1+q+q^2)t+(1+q+2q^2+q^3)t^2+\binom{4}{3}_qt^3+t^4.
	\end{equation*}
\end{proof}

\begin{thm}
	For $a\in\Fq$, let 
	\[M_a^{13}:=\langle e_1,e_2,e_3,e_4:[e_4,e_1]=e_1+ae_3,[e_4,e_2]=e_2,[e_4,e_3]=e_1, [e_3,e_1]=e_2\rangle_{\Fq}.\]
	We have
	\begin{equation*}
		\zeta_{M_a^{13}(\Fq)}^{\leq}(s)= 1+(1+q|V_{13,a}(\Fq)|)t+( 1+q+q^2+q^2|V_{13,a}(\Fq)|)t^2+\binom{4}{3}_qt^3+t^4,
	\end{equation*}
	where \[|V_{13,a}(\Fq)|=|\{x\in\Fq:x^2+x-a=0\}|.\]
\end{thm}

\begin{proof}
	
	Let us count the number of ideals of index $q^2$. For $M_{(0,2),(2,0)}$, we get
	\begin{align*}
		g_{1,2,3}^{\leq}(M)&=a m_{2,4}-m_{2,4} m_{1,3}^2+m_{1,4} m_{2,3} m_{1,3}-m_{2,4} m_{1,3}-m_{2,3}^2+m_{1,4} m_{2,3}=0,\\
		g_{1,2,4}^{\leq}(M)&=m_{2,3} m_{1,4}^2-m_{1,3} m_{2,4} m_{1,4}-m_{2,3} m_{2,4}=0.
	\end{align*}
	Solving the given system of equations give
	\[c_{(0,2),(2,0)}^{\leq}=q^2+(q^2-q)|V_{9,a}(\Fq)|.\]
	
	For $M_{(0,1,1),(1,1,0)}$, we get
	\begin{align*}
		g_{1,3,2}^{\leq}(M)&=1+m_{1,2}m_{1,4}=0,\\
		g_{1,3,4}^{\leq}(M)&=m_{1,4}^2-m_{1,4}m_{3,4}-am_{3,4}^2=0.
	\end{align*}
	If $a=0$, we get $q-1$ solutions. If $a\neq0$, we get $(q-1)|V_{9,a}(\Fq)|$ solutions. 
	Therefore we have
	\[c_{(0,1,1),(1,1,0)}^{\leq}=\begin{cases}
		(q-1)|V_{9,a}(\Fq)|&a\neq0\\
		q-1&a=0.
	\end{cases}\]
	
	For $M_{(0,2),(1,1)}$, we get
	\begin{align*}
		g_{1,4,2}^{\leq}(M)&=-m_{1,2}m_{1,3}=0,\\
		g_{1,4,3}^{\leq}(M)&=a-m_{1,3}-m_{1,3}^2=0.
	\end{align*}
	If $a=0$, Then either $m_{1,3}=0$ and $m_{1,2}\in\Fq$, giving $q$ solutions, or $m_{1,3}=-1$ and $m_{1,2}=0$, giving 1 solution.
	
	If $a\neq0$, then $m_{1,3}$ needs to be a solution of $a-m_{1,3}-m_{1,3}^2$, and since $m_{1,3}\neq 0$ we have $m_{1,2}=0$. 
	
	Therefore we get
	\[c_{(0,2),(1,1)}^{\leq}=\begin{cases}
		|V_{13,a}(\Fq)|&a\neq0,\\
		q+1&a=0.
	\end{cases}\]

	For $M_{(1,1),(2,0)}$, we get
	\begin{align*}
		g_{2,3,1}^{\leq}(M)&=-m_{2,4}=0,\\
		g_{2,3,4}^{\leq}(M)&=-m_{2,4}m_{3,4}=0,
	\end{align*}
	giving
	\[c_{(1,1),(2,0)}^{\leq}=q.\]
	
	For $M_{(1,1),(1,1)}$, we get
	\begin{align*}
		g_{2,4,1}^{\leq}(M)&=m_{2,3}=0,\\
		g_{2,4,3}^{\leq}(M)&=-m_{2,3}=0,
	\end{align*}
	giving 
	\[c_{(1,1),(1,1)}^{\leq}=1.\]
	
	For $M_{(2),(2)}$, we get 
	\begin{align*}
		g_{3,4,1}^{\leq}(M)&=-1=0,
	\end{align*} giving
	\[c_{(2),(2)}^{\leq}=0.\]
	
	Together, we have 
	\[a_{q^2}^{\leq}(M_a^{13})=\begin{cases}
		1+q+q^2+q^2|V_{9,a}(\Fq)|&a\neq0,\\
		1+q+3q^2&a=0.
	\end{cases}\]

	Finally, let us count the number of ideals of index $q$. For $M_{(0,1),(3,0)}$, we get
	\begin{align*}
		g_{1,2,4}^{\leq}(M)&=-am_{2,4}m_{3,4}=0,\\
		g_{1,3,4}^{\leq}(M)&=-m_{2,4}+m_{1,4}(m_{1,4}-m_{3,4})-am_{3,4}^2,\\
		g_{2,3,4}^{\leq}(M)&=m_{2,4}(m_{1,4}-m_{3,4})=0.
	\end{align*}
	If $a=0$, them $m_{2,4}=m_{1,4}(m_{1,4}-m_{3,4})$ and $m_{1,4}(m_{1,4}-m_{3,4})^2=0$, giving $2q-1$ solutions.
	
	If $a\neq 0$, from $   g_{1,2,4}^{\leq}(M)$ we get either $m_{2,4}=0$ or $m_{3,4}=0$. Suppose $m_{3,4}=0$. This would mean $-m_{2,4}+m_{1,4}^2=m_{2,4}m_{1,4}=0$, forcing $m_{1,4}=m_{2,4}=0$. Suppose $m_{2,4}=0$ and $m_{3,4}\neq0$. Then for each $m_{3,4}$, $m_{1,4}/m_{3,4}$ has to be a solution of $x^2-x-a$, giving $(q-1)|V_{9,a}(\Fq)|$ solutions. Therefore, we get
	\[c_{(0,1),(3,0)}^{\leq}=\begin{cases}
		(q-1)|V_{9,a}(\Fq)|+1&a\neq0,\\
		2q-1&a=0.
	\end{cases}\]
	
	For $M_{(0,1),(2,1)}$, we get
	\begin{align*}
		g_{1,2,3}^{\leq}(M)&=-m_{2,3}^2=0,\\
		g_{1,4,3}^{\leq}(M)&=a-m_{1,3}-m_{1,3}^2=0,\\
		g_{2,4,3}^{\leq}(M)&=-(1+m_{1,3})m_{2,3}=0,
	\end{align*}
	giving \[c_{(0,1),(2,1)}^{\leq}=|V_{9,a}(\Fq)|.\]
	
	For $M_{(0,1),(1,2)}$, we get
	\begin{align*}
		g_{1,3,2}^{\leq}(M)&=1=0,
	\end{align*}
	giving
	\[c_{(0,1),(1,2)}^{\leq}=0.\]
	
	for $M_{(1),(3)}$,  we get
	\begin{align*}
		g_{2,4,1}^{\leq}(M)&=1=0,
	\end{align*}
	giving \[c_{(1),(3)}=0.\] 
	
	Together, we have 
	\[a_{q}^{\leq}(M_a^{13})=
	\begin{cases}
		q|V_{9,a}(\Fq)|+1&a\neq0,\\
		2q+1&a=0.
	\end{cases}\]
	
	Summing all up, we get 	
	\begin{equation*}
		\zeta_{M_a^{13}(\Fq)}^{\leq}(s)=
		\begin{cases}
			1+(q|V_{13,a}(\Fq)|+1)t+( 1+q+q^2+q^2|V_{13,a}(\Fq)|)t^2+\binom{4}{3}_qt^3+t^4&a\neq0,\\
			1+(2q+1)t+(3q^2+q+1)t^2+\binom{4}{3}_qt^3+t^4&a=0.
		\end{cases}
	\end{equation*}
	In fact, since $|V_{13,0}(\Fq)|=2$ for all $q$, one can simply write
	\begin{equation*}
		\zeta_{M_a^{13}(\Fq)}^{\leq}(s)= 1+(1+q|V_{13,a}(\Fq)|)t+( 1+q+q^2+q^2|V_{13,a}(\Fq)|)t^2+\binom{4}{3}_qt^3+t^4.\qedhere
	\end{equation*}
\end{proof}

\begin{thm}
	For $a\in\Fq^{\times}$, let 
	\[M_a^{14}:=\langle e_1,e_2,e_3,e_4:[e_4,e_1]=ae_3,[e_4,e_3]=e_1, [e_3,e_1]=e_2\rangle_{\Fq}.\]
	We have
	\begin{equation*}
		\zeta_{M_a^{14}(\Fq)}^{\leq}(s)=1+(1+q|V_{14,a}(\Fq)|)t+(1+q+q^2+q^2|V_{14,a}(\Fq)|)t^2+\binom{4}{3}_qt^3+t^4,
	\end{equation*}
	where \[|V_{14,a}(\Fq)|=|\{x\in\Fq:x^2-a=0\}|.\]
\end{thm}

\begin{proof}
	
	Let us count the number of subalgebras of index $q^2$. For $M_{(0,2),(2,0)}$, we get
	\begin{align*}
		g_{1,2,3}^{\leq}(M)&=a m_{2,4}-m_{2,4} m_{1,3}^2+m_{1,4} m_{2,3} m_{1,3}-m_{2,3}^2=0,\\
		g_{1,2,4}^{\leq}(M)&=m_{2,3} m_{1,4}^2-m_{1,3} m_{2,4} m_{1,4}-m_{2,3} m_{2,4}=0.
	\end{align*}
	Solving the given system of equations give
	\[c_{(0,2),(2,0)}^{\leq}=q^2+(q^2-q)|V_{14,a}(\Fq)|.\]
	
	For $M_{(0,1,1),(1,1,0)}$, we get
	\begin{align*}
		g_{1,3,2}^{\leq}(M)&=1+m_{1,2}m_{1,4}=0,\\
		g_{1,3,4}^{\leq}(M)&=m_{1,4}^2-am_{3,4}^2=0.
	\end{align*}
	Since $a\neq0$, we get 
	\[c_{(0,1,1),(1,1,0)}^{\leq}=
	(q-1)|V_{14,a}(\Fq)|.\]
	
	For $M_{(0,2),(1,1)}$, we get
	\begin{align*}
		g_{1,4,2}^{\leq}(M)&=-m_{1,2}m_{1,3}=0,\\
		g_{1,4,3}^{\leq}(M)&=a-m_{1,3}^2=0.
	\end{align*}
	First, $m_{1,3}$ needs to be a solution of $a-m_{1,3}^2$. Since $a\neq0$, $m_{1,3}\neq 0$ and we have $m_{1,2}=0$. 
	
	Therefore we get
	\[c_{(0,2),(1,1)}^{\leq}=|V_{14,a}(\Fq)|.\]

	For $M_{(1,1),(2,0)}$, we get
	\begin{align*}
		g_{2,3,1}^{\leq}(M)&=-m_{2,4}=0,\\
		g_{2,3,4}^{\leq}(M)&=0=0,
	\end{align*}
	giving
	\[c_{(1,1),(2,0)}^{\leq}=q.\]
	
	For $M_{(1,1),(1,1)}$, we get
	\begin{align*}
		g_{2,4,1}^{\leq}(M)&=m_{2,3}=0,\\
		g_{2,4,3}^{\leq}(M)&=0=0,
	\end{align*}
	giving 
	\[c_{(1,1),(1,1)}^{\leq}=1.\]
	
	For $M_{(2),(2)}$, we get 
	\begin{align*}
		g_{3,4,1}^{\leq}(M)&=1=0,
	\end{align*} giving
	\[c_{(1,1),(1,1)}^{\leq}=0.\]
	
	Together, we have 
	\[a_{q^2}^{\leq}(M_a^{14})=1+q+q^2+q^2|V_{14,a}(\Fq)|.\]

	Finally, let us count the number of ideals of index $q$. For $M_{(0,1),(3,0)}$, we get
	\begin{align*}
		g_{1,2,4}^{\leq}(M)&=-am_{2,4}m_{3,4}=0,\\
		g_{1,3,4}^{\leq}(M)&=-m_{2,4}+m_{1,4}^2-am_{3,4}^2,\\
		g_{2,3,4}^{\leq}(M)&=m_{1,4}m_{2,4}=0.
	\end{align*}.
	
	Since $a\neq 0$, from $   g_{1,2,4}^{\leq}(M)$ we get either $m_{2,4}=0$ or $m_{3,4}=0$. Suppose $m_{3,4}=0$. This would mean $-m_{2,4}+m_{1,4}^2=m_{2,4}m_{1,4}=0$, forcing $m_{1,4}=m_{2,4}=0$. Suppose $m_{2,4}=0$ and $m_{3,4}\neq0$. Then for each $m_{3,4}$, $m_{1,4}/m_{3,4}$ has to be a solution of $x^2-a$, giving $(q-1)|V_{14,a}(\Fq)|$ solutions. Therefore, we get
	\[c_{(0,1),(3,0)}^{\leq}=
	(q-1)|V_{14,a}(\Fq)|+1.\]
	
	For $M_{(0,1),(2,1)}$, we get
	\begin{align*}
		g_{1,2,3}^{\leq}(M)&=-m_{2,3}^2=0,\\
		g_{1,4,3}^{\leq}(M)&=a-m_{1,3}^2=0,\\
		g_{2,4,3}^{\leq}(M)&=-m_{1,3}m_{2,3}=0,
	\end{align*}
	giving \[c_{(0,1),(2,1)}^{\leq}=|V_{14,a}(\Fq)|.\]
	
	For $M_{(0,1),(1,2)}$, we get
	\begin{align*}
		g_{1,3,2}^{\leq}(M)&=1=0,
	\end{align*}
	giving
	\[c_{(0,1),(1,2)}^{\leq}=0.\]
	
	for $M_{(1),(3)}$,  we get
	\begin{align*}
		g_{2,4,1}^{\leq}(M)&=1=0,
	\end{align*}
	giving \[c_{(1),(3)}=0.\] 
	
	Together, we have 
	\[a_{q}^{\leq}(M_a^{14})=1+q|V_{14,a}(\Fq)|.\]
	
	Summing all up, we get 	
	\begin{equation*}
		\zeta_{M_a^{14}(\Fq)}^{\leq}(s)=1+(1+q|V_{14,a}(\Fq)|)t+(1+q+q^2+q^2|V_{14,a}(\Fq)|)t^2+\binom{4}{3}_qt^3+t^4.
	\end{equation*}
\end{proof}	

\section{Further remarks and questions}

\subsection{Zeta functions of solvable Lie algebras over $\mcO_{\mfp}$}

In \textrm{Zeta} \cite{Ross.Zeta}, a software package developed by Rossmann, there is a ``database'' of algebras. In particular, this database includes $L^{i}$, $L_a^{i}$, $M^{i}$, $M_{a}^{i}$, and $M_{a,b}^{i}$ for  specific values of $a$, $b$, and $i$.  If possible, Zeta also gives explicit descriptions of $\zeta_{L(\mcO_{\mfp})}^{*}$. For instance, Zeta gave us

\begin{align*}
	\zeta_{L_1^3(\mcO_{\mfp})}^{\ideal}&=\frac{1+(|V_{3,1}(\Fq)|-1)t^2}{(1-t)(1-t^2)(1-t^3)},\\
	\zeta_{L_1^3(\mcO_{\mfp})}^{\leq}&=\frac{1+(q|V_{3,1}(\Fq)|-1)t+(q-q|V_{3,1}(\Fq)|)t^2-q^2t^3}{(1-t)(1-qt)(1-q^2t^2)^2}.
\end{align*}

Comparing with 
\begin{align*}
	\zeta_{L_a^{3}(\Fq)}^{\ideal}(s)&=1+t+|V_{3,a}(\Fq)|t^2+t^3,\\
	\zeta_{L_a^{3}(\Fq)}^{\leq}(s)&=1+(1+|V_{3,a}(\Fq)|q)t+(1+q+q^2)t^2+t^3,
\end{align*}
for $a\in\Fq$, we  observe that $|V_{3,1}(\Fq)|$ appears in all four of $\zeta_{L_1^3(\mcO_{\mfp})}^{\ideal}$, $\zeta_{L_1^3(\mcO_{\mfp})}^{\leq}$, $\zeta_{L_1^3(\Fq)}^{\ideal}$, and $\zeta_{L_1^3(\Fq)}^{\leq}$. 
Can we expect a similar phenomenon for all $a\in\Fq^{\times}$? In fact, the answer is yes.
Consider 
\[L_{a}^3:=\langle e_{1},e_{2},e_{3}\,\mid\,[e_3,e_1]=e_2, [e_3,e_2]=ae_1+e_2\rangle_{\mcO}\]
as a 3-dimensional solvable $\mcO$-Lie algebra. A result by Klopsch and Voll  \cite[Theorem 1.1]{KlopschVoll/09MZ} suggests that we have

\begin{equation}\label{eq:quadratic.f}
	\zeta_{L_{a}^{3}(\mcO_{\mfp})}^{\leq}(s)=\zeta_{\mcO_{\mfp}^{3}}(s)-Z_{f}(s-2)\frac{q^{2}t}{(1-q^{-1})(1-q^2t)(1-q^2t^2)},
\end{equation}
where 
\[\zeta_{\mcO_{\mfp}^3}(s)=\frac{1}{(1-t)(1-qt)(1-q^2t)}\]
is the zeta function of the abelian 3-dimensional $\mcO_{\mfp}$-Lie algebra, and 
\[Z_f(s)=\int_{\mcO_{\mfp}^{3}}|f(\bfx)|_{q}^{s}d\mu\]
is Igusa's local zeta function associated for $f$ (we refer to \cite{KlopschVoll/09MZ} for details).

For $L_{a}^{3}$ the quadratic form $f$ in \eqref{eq:quadratic.f} becomes
\begin{equation*}\label{eq:la3}
	f(\bfx):=\bfx A\bfx^{t}=-ax_1^2-x_1x_2+x_2^2,
\end{equation*}
where $\bfx=(x_1,x_2,x_3)$, $\lambda_{ij}^{k}$ is the structure constants, and
\begin{align*}
	A:=\begin{pmatrix}
		\lambda_{23}^1&\lambda_{31}^1&\lambda_{12}^1\\  
		\lambda_{23}^2&\lambda_{31}^2&\lambda_{12}^2\\ 
		\lambda_{23}^3&\lambda_{31}^3&\lambda_{12}^3 
	\end{pmatrix}=\begin{pmatrix}
		-a&0&0\\
		-1&1&0\\
		0&0&0
	\end{pmatrix}.
\end{align*} 

The point is, $f(\bfx)=-ax_1^2-x_1x_2+x_2^2$ is a projective version of $V_{3,a}:ax^2+x-1=0$. Hence it is absolutely natural for $|V_{3,a}(\Fq)|$ to appear in both $\zeta_{L_a^3(\Fq)}^{\leq}$ and $\zeta_{L_a^3(\mcO_{\mfp})}^{\leq}$.

For $\zeta_{L_a^3(\mcO_{\mfp})}^{\ideal}$, explicit computations give us 
\begin{equation*}
	\zeta_{L_a^3(\mcO_{\mfp})}^{\ideal}=\frac{1+(|V_{3,a}(\Fq)|-1)t^2}{(1-t)(1-t^2)(1-t^3)},
\end{equation*}
and this was not recorded anywhere.

In this article we provided a complete symbolic description of subalgebra and ideal zeta functions of solvable $\Fq$-Lie algebras of dimension $n\leq4$. On the other hand, over $\mcO_{\mfp}$ we only have some numeric formulas for solvable $\mcO_{\mfp}$-Lie algebras of dimension dimension $n\leq4$, especially for non-nilpotent cases. Would our results help us to compute symbolic formulas for zeta functions of these $\mcO_{\mfp}$-Lie algebras?

\subsection{Classification of solvable Lie algebras over $\Fq$ and the isospectrality of $\zeta_{L(\Fq)}^{*}$}

Let $\mcL_1$ and $\mcL_2$ be two $R$-Lie algebras. In \cite{duSMcDS/99, duSG/06, duSWoodward/08, KlopschVoll/09MZ}, the authors called $\mcL_1$ and $\mcL_2$ `isospectral' if $\zeta_{\mcL_1}^{\leq}(s)=\zeta_{\mcL_2}^{\leq}(s)$, and `normally isospectral' if  $\zeta_{\mcL_1}^{\ideal}(s)=\zeta_{\mcL_2}^{\ideal}(s)$. 

Are $\mcL_1$ and $\mcL_2$ isomorphic if they are (normally) isospectral? Questions of this nature
are traditionally called isospectrality problems, and the answer is no. There are infinitely many examples of non-isomorphic but (normally) isospectral pairs of $R$-Lie algebras. Nevertheless, studying which structural invariants of $\mcL_1$ and $\mcL_2$ remain the same for isospectral pairs is still an interesting problem.

Let us formally define the notion of isospectrality for our setting as follows:
\begin{dfn}
	Let $\mcL_1$ and $\mcL_2$ be two $R$-Lie algebras.  We call $\mcL_1$ and $\mcL_2$ are
	\begin{enumerate}
		\item $\leq$-isospectral if $\zeta_{\mcL_1}^{\leq}(s)=\zeta_{\mcL_2}^{\leq}(s)$,
		\item $\ideal$-isospectral if $\zeta_{\mcL_1}^{\ideal}(s)=\zeta_{\mcL_2}^{\ideal}(s)$.
	\end{enumerate}
\end{dfn}
 Theorem \ref{thm:sub.iso} shows that for $n\geq2$, there always exist a pair of two non-isomorphic but $\leq$-isospectral $\Fq$-Lie algebras of dimension $n$, one abelian and one solvable but non-nilpotent. This observation suggests that being ${\leq}$-isospectral may not be a strong property: even nilpotency does not survive. Furthermore, note that over odd prime $p$, $|V_{3,2}(\Fp)|=|V_{4,1}(\Fp)|=2$, making a non-isomorphic pair of $\Fp$-Lie algebras $L_{2}^{3}(\Fp)$ and $L_{1}^{4}(\Fp)$ both $\leq$-isospectral and $\ideal$-isospectral.

On the other hand, for instance suppose $\zeta_{L_a^{4}(\Fq)}^{*}(s)=\zeta_{L_b^{4}(\Fq)}^{*}(s)$. This implies $|V_{4,a}(\Fq)|=|V_{4,b}(\Fq)|$, which happens if and only if there is an $\alpha\in\Fq^{\times}$ with $a=\alpha^2b$, and this is exactly the same condition that makes $L_{a}^{4}(\Fq)\cong L_{b}^{4}(\Fq)$. Similar observation can be made for $M_{a}^{14}(\Fq)$. In addition, as demonstrated in Theorem \ref{thm:sol.Ma9.ideal} and \ref{thm:sol.Ma9.sub}, the very equation $T^2-T-a$ that played a crucial role in classifying $M_{a}^{9}$ in \cite{deG/05} also plays an important role in computing $\zeta_{M_{a}^{9}(\Fq)}^{*}(s)$.

Classifying solvable Lie algebras over finite fields is a difficult problem, and in fact we do not have a complete classification beyond dimension 4. Do our results suggest that understanding the nature of zeta functions of solvable $\Fq$-Lie algebras and their *-isospectrality may contribute on this classification problem, providing a potential tool for an isomorphism test? One way or the other, it would be very helpful to compute more examples of $\zeta_{L(\Fq)}^{\triangleleft}(s)$ for higher dimensional solvable cases for further studies.


\bibliographystyle{amsplain}
\bibliography{Lee_solvable}

\end{document}